\documentclass[journal]{IEEEtran}  % Comment this line out
\usepackage{amsmath}
\usepackage{amssymb}
\usepackage{amsthm}
\usepackage{enumitem}
\usepackage{graphics} % for pdf, bitmapped graphics files
\usepackage{epsfig} % for postscript graphics files
\usepackage{epstopdf}
\usepackage{mathptmx} % assumes new font selection scheme installed
\usepackage{times} % assumes new font selection scheme installed
\usepackage{amsmath} % assumes amsmath package installed
\usepackage{amssymb}  % assumes amsmath package installed
\usepackage{subfigure}% in preamble
\usepackage[utf8]{inputenc}
\usepackage[english]{babel}
\usepackage{dsfont}
\usepackage{color}

\usepackage{multicol}
\newtheorem{thm}{Theorem}
\newtheorem{defn}{Definition}
\newtheorem{prop}{Proposition}
\newtheorem{lem}{Lemma}

\newtheorem{cor}{Corollary}

\newtheorem{ex}{Example}
\newtheorem{cex}{Counterexample}

\newcommand{\mcl}[1]{\mathcal{#1}}

\newcommand{\mbf}[1]{\mathbf{#1}}

\newcommand{\R}{\mathbb{R}}

\newcommand{\N}{\mathbb{N}}

\newcommand{\norm}[1]{\lVert{#1}\rVert}

\newcommand{\eps}{\varepsilon}

\DeclareMathOperator{\sign}{sign}
\DeclareMathOperator*{\esssup}{ess\,sup}
\DeclareMathOperator*{\essinf}{ess\,inf}

\title{\LARGE \bf
Polynomial Approximation of Value Functions and Nonlinear Controller Design with Performance Bounds
}

\author{Morgan Jones,%
		\thanks{M. Jones is with the Department of Automatic Control and Systems Engineering,
	The University of Sheffield,   {\tt \scriptsize morgan.jones@sheffield.ac.uk} } 
	Matthew M. Peet% <-this % stops a space
	\thanks{M. Peet is with the School for the Engineering of Matter, Transport and Energy, Arizona State University, Tempe, AZ, 85298 USA. e-mail: {\tt \small mpeet@asu.edu } }
}

\begin{document}

\maketitle
\thispagestyle{plain}
\pagestyle{plain}

%%%%%%%%%%%%%%%%%%%%%%%%%%%%%%%%%%%%%%%%%%%%%%%%%%%%%%%%%%%%%%%%%%%%%%%%%%%%%%%%

\begin{abstract}
For any suitable Optimal Control Problem (OCP) there exists a value function, defined as the unique viscosity solution to the Hamilton-Jacobi-Bellman (HJB) Partial-Differential-Equation (PDE), and which can be used to design an optimal feedback controller for the given OCP. In this paper we approximately solve the HJB-PDE by proposing a sequence of Sum-Of-Squares (SOS) problems, each of which yields a polynomial sub-solution to the HJB-PDE. We show that the resulting sequence of polynomial sub-solutions converges to the value function of the OCP in the $L^1$ norm. Furthermore, for each polynomial sub-solution in this sequence we show that the associated sequence of sublevel sets converges to the sublevel set of the value function of the OCP in the volume metric. Next, for any approximate value function, obtained from an SOS program or any other method (e.g. discretization), we construct an associated feedback controller, and show that sub-optimality of this controller as applied to the OCP is bounded by the distance between the approximate and true value function of the OCP in the $W^{1,\infty}$ (Sobolev) norm. Finally, we demonstrate numerically that by solving our proposed SOS problem we are able to accurately approximate value functions, design controllers and estimate reachable sets.

%%%%%%%%%%%%%%%%%%%%%%%%%%%%%%%%%%%%%%%%%%%%%%%%%%%%%%%%%%%%%%%%%%%%%%%%%%%%%%%%%%%%%%% Stuff I took out of abstract%%%%%%%%%%%%%%

%Solving the HJB analytically is rarely possible.

%, and existing numerical approximation schemes largely rely on discretization - implying that the resulting approximate value functions may not have the useful property of being uniformly less than or equal to the true value function (ie be sub-value functions). Furthermore, controllers obtained from such schemes currently have no associated bound on performance. To address these issues, for a given OCP, 

%This result implies approximation of value functions in the $W^{1,\infty}$ norm results in feedback controllers with performance that can be made arbitrarily close to optimality.

\end{abstract}

\vspace{-0.6cm}

\section{Introduction}
Consider a nested family of Optimal Control Problems (OCPs), each initialized by $(x_0,t_0) \in \R^n \times [0,T]$, and each an optimization problem of the form

\vspace{-0.4cm}
{
\begin{align} \nonumber
 & (\mbf u^*, x^*) \in \arg \inf_{\mbf u,x} \bigg\{  \int_{t_0}^T c(x(t), \mbf u(t) ,t) dt + g(x(T))  \bigg\} \text{ subject to, }\\ \label{intro: optimal control probelm}
&   \dot{x}(t) \hspace{-0.05cm} = \hspace{-0.05cm} f(x(t), \mbf u(t)),  \text{ } \mbf u(t) \hspace{-0.05cm} \in \hspace{-0.05cm} U,  \text{ for all } t \hspace{-0.05cm} \in \hspace{-0.05cm} [t_0,T], \text{ } \hspace{-0.05cm} x(t_0) \hspace{-0.05cm} = \hspace{-0.05cm}x_0.
\end{align} }
{The problem of solving OCPs~\eqref{intro: optimal control probelm} plays a central role in many practical applications, for instance in the design of non-pharmaceutical interventions in epidemics~\cite{kantner2020beyond}, optimal train operation~\cite{khmelnitsky2000optimal}, optimal maintenance strategies for manufacturing systems~\cite{huang2018real}, etc.   }

{Solving OCPs directly can be challenging. Fortunately, the problem of solving a family of OCPs~\eqref{intro: optimal control probelm} can be reduced to the problem of solving a Partial Differential Equation (PDE)~\cite{liberzon2011calculus}}. From the principle of optimality, if $(\mbf u^*, x^*)$ solve the OCP for $(x_0,t_0)$, then {$(\tau_t \mbf u^*, x^*(t))$ (where $\tau_t \mbf u^*(s) =  \mbf u^*(t+s)$ for all $s \ge 0$)} solves the OCP for $(x^*(t),t)$ for any $t \in[t_0,T]$. This can be used to show that if a function, $V$, satisfies the Hamilton Jacobi Bellman (HJB) PDE, defined as
	\vspace{-0.15cm}	\begin{align}  \nonumber
&\nabla_t V(x,t) + \inf_{u \in U} \left\{ c(x,u,t) + \nabla_x V(x,t)^T f(x,u) \right\} = 0 \\
& \hspace{3.75cm} \text{for all } (x,t) \in \R^n \times (0,T), \label{intro: general HJB PDE}\\ \nonumber
& V(x,T)= g(x) \quad \text{for all } x \in \R^n,
\end{align}
then necessary and sufficient conditions for $(\mbf u^*, x^*)$ to solve OCP~\eqref{intro: optimal control probelm} initialized by $(x_0,t_0)$ are
\vspace{-0.2cm}
	\begin{align} \nonumber
& \mbf u ^*(t)= k(x^*(t), t),\;  	\dot{x}^*(t)= f(x^*(t),\mbf u ^*(t)),\;\;\; \text{and}\;\;\; x^*(t_0)=x_0,\notag \\
& \;\text{where}\;\; k(x,t) \in \arg \inf_{u \in U}\left\{ c(x,u,t) + \nabla_x V(x,t)^T f(x,u)\right\}. \label{intro: opt controller}
	\end{align}
	For a given family of OCPs~\eqref{intro: optimal control probelm}, if $V$ satisfies Eq.~\eqref{intro: general HJB PDE}, then $V$  is called the {Value Function} (VF) of the OCP. If $V$ is the VF, then for any $(x,t)$, the value $V(x,t)$ determines the optimal objective value of OCP~\eqref{intro: optimal control probelm} initialized by $(x,t)$. Furthermore, the VF yields a solution to the OCP~\eqref{intro: optimal control probelm} initialized by $(x_0,t_0)$ through application of Eq.~\eqref{intro: opt controller}. We call any $k: \Omega \times [0,T] \to U$ that satisfies Eq.~\eqref{intro: opt controller} a controller and we say this controller is the optimal controller for the OCP when $V$ is the VF of the OCP.
	
Thus knowledge of the VF allows us to solve the nested family of OCPs in~\eqref{intro: optimal control probelm}. Unfortunately, to find the VF, we must solve the HJB PDE, given in Eq.~\eqref{intro: general HJB PDE}, and this PDE has no analytic solution. In the absence of an analytic solution, we often parameterize a family of candidate VFs and search for one which satisfies the HJB PDE. However, this is a non-convex optimization problem since the HJB PDE is nonlinear. In this paper we view the search for a VF through the lens of convex optimization. Moreover, given an OCP, we are particularly interested in computing a sub-VF, a function that is uniformly less than or equal to the VF of the OCP (ie a function $\tilde{V}$ such that $\tilde{V}(x,t) \le V(x,t)$ for all $(x,t) \in \R^n \times [0,T]$ where $V$ is the VF of the OCP). We consider what happens when we relax the nonlinear equality constraints imposed by the HJB PDE to linear inequality constraints and tighten the optimization problem's feasible set to polynomials. In this context, given an OCP, we consider the following questions.
%
%Moreover, if we were to search for a solution to Eq.~\eqref{intro: general HJB PDE} over a parameterized set of differentiable functions, $V\in \mbf V$, then the HJB defines a infinite-dimensional set of nonlinear equality constraints and hence the problem is intractable.
%Therefore, from a computational perspective, the problem of finding a VF, and hence solving the OCP via \eqref{intro: opt controller}, is an infinite dimensional non-convex optimization problem; a class of problems considered in general to be intractable to solve \cite{sun2015nonconvex}.
%This naturally leads us to consider the following questions.
%	In this paper we seek to answer two important questions:
	\begin{enumerate}[start=1,label={\bfseries Q\arabic*:}]
				\item Can we pose a sequence of convex optimization problems, each yielding a polynomial sub-VF that can be made arbitrarily ``close" to the VF of the OCP?
                \item Can we bound the sub-optimality in performance of a controller constructed from some function ${V}$ by the ``distance" between ${V}$ and the VF of the OCP?
		\end{enumerate}
	\vspace{-0.4cm}
	\subsection{\textbf{Q1}: Optimal Polynomial Sub-Value Functions}
Over the years, many numerical methods have been proposed for solving the HJB PDE~\eqref{intro: general HJB PDE} for a given OCP. Within this literature, a substantial number of the algorithms are based on a finite-dimensional projection of the spatial domain (griding/meshing/discretization of the state space). In this class of algorithms we include (mixed) finite elements methods - an important example of which is~\cite{gallistl2020mixed}. Specifically, the approach in~\cite{gallistl2020mixed} yields an approximate VF with an error bound on the first order mixed $L^2$ norm - a bound which converges as the number of elements is increased (assuming the Cordes condition holds). Other examples of this class of methods include the discretization approaches in~\cite{achdou2008homogenization,kalise2018polynomial}. For example, in~\cite{achdou2008homogenization}, we find an algorithm which yields an approximate VF with an $L^{\infty}$ error bound which converges as the level of discretization increases. Alternative non-grid based algorithms include the method of characteristics~\cite{liberzon2011calculus}, which can be used to compute evaluations of VF at fixed $(x,t) \in \R^n$, and max-plus methods~\cite{mceneaney2007curse}. The result in~\cite{mceneaney2007curse} considers an OCP with linear dynamics and a cost function which is the point-wise maximum of quadratic functions. This max-plus approach yields an approximate VF with a converging error bound which holds on $x \in\R^n$, but increases with $|x|$.

%In the context of approximating VFs, we may also consider work which considers the more narrowly focused problem of estimation of the reachable set (Defn.~\ref{defn: reachbale set}); the set of states that can be reached by the solution map from a given set of initial conditions. Specifically, value functions can be used to compute reachable sets. Likewise, candidate VFs can be used to approximate reachable sets. The link between solutions of the HJB PDE and reachable sets was first shown in~\cite{mitchell2005time} (which solved the HJB PDE via discretization).

While all of these numerical methods yield approximate VFs with associated approximation error bounds, the use of these functions for controller synthesis (see \textbf{Q2}) and reachable set estimation has been more limited (the connection between VFs, the HJB and reachable sets was made in~\cite{mitchell2005time}). This is due to the fact that the approximate VFs obtained from such discretization methods are difficult to manipulate and {apart from being close to the true VF,} have relatively few provable properties (such as being {uniformly less than or greater to the true VF ie being} sub or super-VFs). {Being a sub or super-VF is an important property of any approximate VF. As shown in Cor.~\ref{cor: sub value subleve; sets contain reachable sets}, sub/super-VFs can yield outer bounds on reachable sets that can be used to certify that the underlying system does not transition into regions of the state space deemed unsafe; a useful tool in the safety analysis of dynamical systems.}

%We note that none of the aforementioned state of the art numerical methods approximate the VF of a given OCP by a single polynomial sub-VF, unlike the methods later proposed in this paper. A major advantage of having a polynomial approximation, $P$, of the VF for some given OCP, is that $\nabla_x P$ can be efficiently computed (a useful property for solving the controller synthesis Eq.~\eqref{intro: opt controller}). Moreover, the advantage of having a sub-VF approximation is that the sublevel set of any sub-VF is guaranteed to contain the sublevel set of the true VF (see Cor.~\ref{cor: sub value subleve; sets contain reachable sets}), leading to outer approximations of reachable sets (a useful property for safety analysis).
%	

%In each of~\cite{kamoutsi2017infinite,Pakniyat_2019,korda2016controller} the OCP is first lifted to an infinite dimensional space of measures to formulate an LP for which Lasserre's hierarchy of primal-dual moment-sum-of-squares semidefinite relaxations that can be used to approximate the VF associated with infinite time horizon OCP's.
  % We note none of \cite{kamoutsi2017infinite} \cite{Pakniyat_2019} \cite{korda2016controller} study finite-horizon OCP's of Form \eqref{intro: optimal control probelm}.
%\cite{beuchat2019accelerated}

To address these issues, in this paper we focus on obtaining approximate VFs which are both polynomial and sub-VFs. Specifically, the use of polynomials ensures that the derivative of the approximated VFs can be efficiently computed (a useful property for solving the controller synthesis Eq.~\eqref{intro: opt controller}), while the use of sub-VFs ensures that sublevel sets of the VF are guaranteed to contain the sublevel set of the true VF (see Cor.~\ref{cor: sub value subleve; sets contain reachable sets}), and hence provide provable guarantees on the boundary of the reachable set (a useful property for safety analysis).

Substantial work on SOS relaxations of the HJB PDE for reachable set estimation {and safety set analysis} includes the carefully constructed optimization problems in~\cite{summers2013quantitative, yin2018reachability,xue2019inner,zhang2019set} and includes, of course, our work in~\cite{jones2019relaxing,jones2019using}. {Such SOS relaxations of the HJB PDE can yield approximate VFs}. However, there seems to be no prior work on using approximation theory to prove bounds on the sub-optimality of either controllers (see \textbf{Q2}) or corresponding reachable sets {constructed from such approximated VFs}. We note, however, that~\cite{xue2019inner} did establish {the} \textit{existence} of a polynomial sub-solution to the HJB {arbitrarily close to the true solution of the HJB} in the framework of reachable sets. Treatments of the moment-based alternatives to the SOS approach includes~\cite{kamoutsi2017infinite,Pakniyat_2019,korda2016controller,zhao2017control}. Another duality-based approach, found in~\cite{chen2019duality}, considers a density-based dual to the VF and uses finite elements method to iteratively approximate the density and VF.

In this paper we answer {\bfseries Q1} by considering ``sub-solutions'' to the HJB PDE~\eqref{intro: general HJB PDE}. Specifically, a ``sub-solution'', $\tilde V$, to the HJB PDE~\eqref{intro: general HJB PDE} satisfies the relaxed inequality constraint
\vspace{-0.1cm}
\begin{align} \label{intro: relax HJB}
\nabla_t \tilde V(x,t) +  c(x,u,t) + \nabla_x \tilde V(x,t)^T f(x,u) \ge 0 %\text{ for } (x,u,t)\in \R^n \times U \times [0,T]
\end{align}
for all $u \in U$ and $(x,t)\in \R^n \times [0,T]$, which implies that if $V$ is a VF, $\tilde V(x,t)\le V(x,t)$ - i.e. $\tilde V$ is a sub-VF. Then given an OCP~\eqref{intro: optimal control probelm} and based on this relaxed version of the HJB PDE~\eqref{intro: relax HJB}, we propose a sequence of SOS programming problems, indexed by the degree $d \in \N$ of the polynomial variables, and given in Eq.~\eqref{opt: SOS for sub soln of finite time}. The solution to each instance of the proposed sequence of optimization problems yields a polynomial $P_d$ that is a sub-solution to the HJB PDE~\eqref{intro: general HJB PDE} (or sub-VF). We then show in Prop.~\ref{prop: SOS converges} that for any VF $V$ associated with the given OCP we have, \vspace{-0.2cm}
\[
\lim_{d\rightarrow \infty}\norm{P_d-V}_{L^1}=0.
\]
Furthermore, in Prop.~\ref{prop: SOS converges Dv} we show that this implies that the sublevel sets of $\{P_d\}_{d \in \N }$ converge to the sublevel sets of any VF, $V$, of the OCP (respect to the volume metric).

 Our proposed method of approximately solving the HJB PDE by solving an SOS programming problem is implemented via Semi-Definite Programming (SDP). SDP problems can be solved to arbitrary accuracy in polynomial time using interior point methods~\cite{vandenberghe1996semidefinite}. However, the number of variables in the SDP problem associated with an $n$-dimensional and $d$-degree SOS problem is of the order $n^{d}$~\cite{ahmadi2019dsos}, and therefore exponentially increases as $d \to \infty$. {Fortunately there exist several methods that improve the scalability of SOS~\cite{ahmadi2019dsos,zheng2019block} but we do not discus such methods in this paper.}
\vspace{-0.25cm}
\subsection{\textbf{Q2}: Performance bounds for controllers constructed from approximate VFs}
The use of approximate VFs to construct controllers has been well-treated in the literature, although such controllers often: apply only to OCPs with specific structure (typically dynamics are affine in the input variable, see~\cite{ribeiro2020control} for linearization techniques that approximate non-input affine dynamics by input affine dynamics); do not have associated performance bounds; and/or assume differentiability of the VF. For example, in~\cite{jiang2015global,abu2005nearly,baldi2012scalable,baldi2015piecewise,zhu2017policy} policy iteration methods are proposed that alternate between finding approximations of the VF based on a controller and using the approximate VF to synthesizing controller. Also in~\cite{abu2005nearly} it was shown that the proposed policy iteration method converges under the rather restrictive assumption that the true VF is differentiable.  Alternatively, grid based approaches that synthesize controllers can be found in~\cite{kang2017mitigating, kunisch2004hjb}. However, the method in~\cite{kang2017mitigating} is only shown to yield a function that converges to the VF but no performance bound is given for the controller. In~\cite{kunisch2004hjb}, convergence to the optimal controller is demonstrated numerically in certain cases, but no provable performance bound is given.

There are also results within the SOS framework for optimization of polynomials that use approximate VFs to construct controllers. For example, in~\cite{leong2014optimal} it was shown that the objective value of a specific class of OCP's using a controller constructed from a given approximate VF was bounded from above by the approximated VF. However, this bound was conservative and no method was given for refinement of the bound. In~\cite{jennawasin2011performance} a method for approximating VFs by sub and super-VFs that are also SOS polynomials is given, however, no VF approximation error bounds or resulting controller synthesis performance bound is given. Alternatively~\cite{cunis2020sum} proposes a bilinear SOS optimization framework which iterates between finding a Lyapunov function and finding a controller to maximize the region of attraction. However, this work does not consider OCPs or VFs per se.

Despite this extensive literature, to the best of our knowledge, there exists no way of constructing approximate VFs for which the performance of the associated controller can be proven to be arbitrarily close to optimal (although such bounds exist for discrete time systems over infinite time horizons~\cite{bertsekas1995neuro}). For such a result to exist in continuous-time over finite time horizons, then, we need some way of bounding sub-optimality of the performance of the controller based on distance of the approximated VF to the true VF.

% We do not include such numerical schemes as an answer to \textbf{Q1} as they do not formulate the problem of approximated the VF of an OCP as a convex optimization problem. Nevertheless, these numerical schemes are indirectly relevant to \textbf{Q2} as they provide associated error bounds between the numerical schemes' outputted candidate VF and the true VF of a given OCP, without answering \textbf{Q2} by providing sub-optimality performance bounds of any controller constructed from the outputted candidate VF.

To address this need, in Sec.~\ref{sec:optimal controller approx} we answer {\bfseries Q2} by showing that for any $V$, we can construct a candidate solution to the OCP~\eqref{intro: optimal control probelm}, $\mbf u(t)=k(x(t),t)$, given by the controller defined in Eq.~\eqref{intro: opt controller}. We then show in Thm.~\ref{thm: performance bounds} that the corresponding objective value of the OCP~\eqref{intro: optimal control probelm} evaluated at $\mbf u$ is within $C \norm{V^*-V}_{W^{1,\infty}}$ of the optimal objective, where $V^*$ is the true VF of the OCP and $C>0$ is given in Eq.~\eqref{eqn: C}.{This result implies approximation of value functions in the $W^{1,\infty}$ norm results in feedback controllers with performance that can be made arbitrarily close to optimality}. Note, this result may be of broad interest since it does not require $V$ to be a solution to our proposed SOS Problem~\eqref{opt: SOS for sub soln of finite time} and hence provides a bound on the sub-optimality of controllers constructed from any approximate VF.

\vspace{-0.4cm}
\section{{ Notation}} \label{sec: Notation}
\vspace{-0.4cm}
\subsection{Standard Notation} \label{sec: notataion A}
 We define $\sign: \R \to \{-1,1\}$ and for $A \subset \R^n$ $\mathds{1}_A : \R^n \to \R$ by  $\sign(x) = \begin{cases}
& 1 \text{ if } x\ge 0\\
& -1 \text{ otherwise}
\end{cases}$ and $\mathds{1}_A(x) = \begin{cases}
& 1 \text{ if } x \in A\\
& 0 \text{ otherwise.}
\end{cases}$ For two sets $A,B \subset \R^n$ we denote $A/B=\{x \in A: x \notin B\} $. For $B \subseteq \R^n$,  $\mu(B):=\int_{\R^n} \mathds{1}_B (x) dx$ is the Lebesgue measure of $B$, and for $X \subseteq \R^n$ and a function $f:X \to \R$ we denote the essential infimum by $\essinf_{x \in X}f(x):=\sup\{a \in \R: \mu(\{x \in X: f(x)<a\})= 0\}$. Similarly we denote the essential supremum by $\esssup_{x \in X}f(x):=\inf\{a \in \R: \mu(\{x \in X: f(x)>a\})= 0\}$. For $x \in \R^n$ we denote the Euclidean norm by $||x||_2=\sqrt{\sum_{i=1}^n x_i^2 }$. For $r>0$ and $x \in \R^n$ we denote the ball $B(x,r):=\{y \in \R^n: ||x-y||_2<r\}$. For an open set $\Omega \subset \R^n$ we denote the boundary of the set by $\partial \Omega$ and denote the closure of the set by $\overline{\Omega}$. Let $C(\Omega, \Theta)$ be the set of continuous functions with domain $\Omega \subset \R^n$ and image $\Theta \subset \R^m$. For an open set $\Omega \subset \R^n$ and $p \in [1,\infty)$ we denote the set of $p$-integrable functions by $L^p(\Omega,\R):=\{f:\Omega \to \R \text{ measurable }: \int_{\Omega}|f|^p<\infty    \}$, in the case $p = \infty$ we denote $L^\infty(\Omega, \R):=\{f:\Omega \to \R \text{ measurable }: \esssup_{x \in \Omega}|f(x)| < \infty \}$. For $\alpha \in \N^n$ we denote the partial derivative $D^\alpha f(x):= \Pi_{i=1}^{n} \frac{\partial^{\alpha_i} f}{\partial x_i^{\alpha_i}} (x)$ where by convention if $\alpha=[0,..,0]^T$ we denote $D^\alpha f(x):=f(x)$. {We denote the set of $i$ continuously differentiable functions} by $C^i(\Omega,\Theta):=\{f \in C(\Omega,\Theta): D^\alpha f \in C(\Omega, \Theta) \text{ } \text{ for all } \alpha \in \N^n \text{ such that } \sum_{j=1}^{n} \alpha_j \le i\}$. For $k \in \N$ and $1 \le p \le \infty$ we denote the Sobolev space of functions with weak derivatives (Defn.~\ref{def: weak deriv}) by $W^{k,p}(\Omega,\R):=\{u\in L^p(\Omega,\R): D^\alpha u \in L^p(\Omega,\R) \text{ for all } |\alpha| \le k \}$. For $u \in W^{k,p}(\Omega,\R)$ we denote the Sobolev norm $||u||_{W^{k,p}(\Omega, \R)}:= \begin{cases}
\left( \sum_{|\alpha| \le k} \int_\Omega (D^\alpha u(x))^p dx \right)^{\frac{1}{p}} \text{ if } 1 \le p < \infty\\
\sum_{|\alpha| \le k} \esssup_{ x \in \Omega  } \{|D^\alpha u(x) |\} \text{ if } p= \infty.
\end{cases}$ In the case $k=0$ we have $W^{0,p}(\Omega,\R)=L^p(\Omega,\R)$ and thus we use the notation $|| \cdot ||_{L^p(\Omega,\R)} :=|| \cdot ||_{W^{0,p}(\Omega,\R)} $. We denote the shift operator $\tau_{s}:L^2([0,T], \R^m) \rightarrow L^2([0,T-s], \R^m)$, where $s \in [0,T]$, and defined by $ (\tau_{s} \mbf u)(t) := \mbf u(s+t) \text{ for all } t \in [0,T-s]$.
\vspace{-0.35cm}
\subsection{{Non-Standard Notation}} \label{sec: notation B}
We denote the set of locally and uniformly Lipschitz continuous functions on $\Theta_1 \text{ and }\Theta_2$, Defn.~\ref{defn:lip cts}, by $LocLip(\Theta_1,\Theta_2)$ and $Lip(\Theta_1,\Theta_2)$ respectively. Let us denote bounded subsets of $\R^n$ by $\mcl B:=\{B \subset \R^n: \mu(B)<\infty\}$.  If $M$ is a subspace of a vector space $X$ we denote equivalence relation $\sim_M$ for $x,y \in X$ by $x \sim_M y$ if $x-y \in M$. We denote quotient space by $X \pmod M:=\{ \{y \in X: y \sim_M x \}: x \in X\}$. For an open set $\Omega \subset \R^n$ and $\sigma>0$ we denote $<\Omega>_\sigma:=\{x \in \Omega: B(x,\sigma) \subset \Omega \}$. For $V \in C^1(\R^n \times \R, \R)$ we denote $\nabla_x V:= (\frac{\partial V}{\partial x_1},....,\frac{\partial V}{\partial x_n})^T$ and $\nabla_t V=\frac{\partial V}{\partial x_{n+1}} $. We denote the space of polynomials $p: \Omega \to \Theta$ by $\mcl P(\Omega,\Theta)$ and polynomials with degree at most $d \in \N$ by $\mcl{P}_d(\Omega,\Theta)$. We say $p \in \mcl{P}_d(\R^n,\R)$ is Sum-of-Squares (SOS) if there exists $p_i \in \mcl{P}_d(\R^n,\R)$ such that $p(x) = \sum_{i=1}^{k} (p_i(x))^2$. We denote $\sum_{SOS}^d$ to be the set SOS polynomials of at most degree $d \in \N$ and the set of all SOS polynomials as $\sum_{SOS}$. We denote $Z_d: \R^n \times \R \to \R^{\mcl N_d}$ as the vector of monomials of degree $d \in \N$ or less and of size $\mcl N_d:= {d+n \choose d}$.
\vspace{-0.3cm}
\section{Optimal Control Problems} \label{sec: optimal control}
The nested family of finite-time Optimal Control Problems (OCPs), each initialized by $(x_0,t_0) \in \R^n\times [0,T]$, are defined as:
\vspace{-0.3cm}
\begin{align}  \nonumber
& (\mbf u^*, x^*) \in \arg \inf_{\mbf u, x} \bigg\{  \int_{t_0}^T c(x(t), \mbf u(t) ,t) dt + g(x(T))  \bigg\} \text{ subject to, }\\ \label{opt: optimal control probelm}
&   \dot{x}(t) = f(x(t), \mbf u(t)) \quad  \text{ for all } t \in [t_0,T],\\ \nonumber
&  (x(t), \mbf u(t)) \in \Omega \times U \quad  \text{ for all } t \in [t_0,T], \quad   x(t_0)=x_0,
\end{align}
where $c: \R^n \times \R^m \times \R \to \R$ is referred to as the running cost; $g: \R^n \to \R$ is the terminal cost; $f: \R^n \times \R^m \to \R^n$ is the vector field; $\Omega \subset \R^n$ is the state constraint set; $U \subset \R^m$ is the input constraint set; and $T$ is the final time. For a given family of OCPs of Form \eqref{opt: optimal control probelm} we associate the tuple $\{c,g,f,\Omega,U,T\}$. %Given an OCP~\eqref{opt: optimal control probelm}, we refer to any candidate solution, $\mbf u :[t_0,T] \to U$, as a controller.

In this paper we consider a special class of OCPs of Form~\eqref{opt: optimal control probelm}, where $U$ is compact and $c,g,f$ are locally Lipschitz continuous. We next recall the definition of local Lipschitz continuity.
\begin{defn} \label{defn:lip cts}
	Consider sets $\Theta_1 \subset \R^n$ and $\Theta_2 \subset \R^m$. We say the function $F: \Theta_1 \to \Theta_2$ is \textbf{locally Lipschitz continuous} on $\Theta_1 \text{ and }\Theta_2$, denoted $F \in LocLip(\Theta_1, \Theta_2)$, if for every compact set $X \subseteq \Theta_1$ there exists $K_X>0$ such that for all $x,y \in X$
	\begin{align} \label{lip}
	||F(x) - F(y)||_2 \le K_X ||x - y||_2.
	\end{align}
	If there exists $K>0$ such that Eq.~\eqref{lip} holds for all $x,y \in \Theta_1$ we say $F$ is \textbf{uniformly Lipschitz continuous}, denoted $F \in Lip(\Theta_1,\Theta_2)$.
	%We refer to $K>0$ as the Lipschitz constant of $F$ and throughout the paper denote the Lipschitz constant of a Lipschitz continuous function $F$ by $L_F>0$.
\end{defn}
\vspace{-0.1cm}
\begin{defn} \label{defn: lip optimal control}
	We say the six tuple $\{c,g,f,\Omega,U,T\}$ is a Family of \textit{Lipschitz OCPs} of Form~\eqref{opt: optimal control probelm} or $\{c,g,f,\Omega,U,T\} \in \mcl M_{Lip}$ if:
%	\vspace{-0.4cm}
 %   \begin{multicols}{2}
{	\begin{enumerate}
		\item $c \in  LocLip(\Omega \times U \times [0,T], \R )$.
		\item $g \in LocLip(\Omega,\R)$.
		\item  $f \in LocLip(\Omega \times U, \R )$.
		\item $U \subset \R^m$ is compact.
	\end{enumerate} }
%\end{multicols}
\end{defn}
\vspace{-0.2cm}
For $\{c,g,f,\Omega,U,T\} \in \mcl M_{Lip}$, if $\Omega = \R^n$ we say the family of associated OCPs is \textit{state unconstrained}, and if $\Omega \neq \R^n$ we say the associated family of OCPs is \textit{state constrained}.

%\begin{lem}
%	Suppose $X_0 \subset \R^n$, $\Omega \subset \R^n$, $U \subset \R^m$, $f: \R^n \times \R^m \to \R^n$, and $T \in \R^+$. Then $BR_f(\bar{X}_0,\Omega,U,\{T\})= \overline{BR_f(X_0,\Omega,U,\{T\})}$.
%\end{lem}
%\begin{proof}
%	Suppose $y \in BR_f(\bar{X}_0,\Omega,U,\{T\})$ then there exists $x \in \bar{X}_0$ and $\mbf u \in \mcl U_{\Omega,U,f,T}(y)$ such that $\phi_f(y,T,\mbf u)=x$. If $x \in X_0$ then $y \in BR_f(X_0,\Omega,U,\{T\}) \subseteq \overline{BR_f(X_0,\Omega,U,\{T\})}$. If $x \notin X_0$ then $x \in \delta X_0$.
%\end{proof}
\vspace{-0.2cm}
\section{Properties and Importance of Value Functions} \label{Section: finite time}
%To solve OCP's of Form \eqref{opt: optimal control probelm} we may find a Value Function (VF), as defined in Eq.~\eqref{opt: optimal control general}, and use Theorem \ref{thm: value functions construct optimal controllers} to construct an optimal controller.
We recall several important properties of  Value Functions (VFs) that we use to prove the main result of the paper, given in Sec~\ref{sec: SOS implemtation}. In the following subsections we recall that for every family of Lipschitz OCPs, as defined in Section~\ref{sec: optimal control}, there exists a function, called the Value Function (VF), which:
\begin{enumerate}[label=(\Alph*)]
\item Is determined by the solution map - Eq.~\eqref{opt: optimal control general}.
\item Solves the Hamilton-Jacobi-Bellman (HJB) Partial Differential Equation (PDE) - Eq.~\eqref{eqn: general HJB PDE}.
\item Can be used to construct a solution to the OCP.
%\item Is uniformly lower or upper bounded by functions satisfying the dissipation Inequalities \eqref{ineq: diss ineq for sub sol of HJB} and \eqref{ineq: BC} or \eqref{ineq: super diss ineq for sub sol of HJB} and \eqref{ineq: super BC} respectively.
\end{enumerate}
%In the special case of no input (stability), the value function corresponds to the standard Massera-class converse Lyapunov construction. In the case of reachable sets, it corresponds to . In the case of Dynamic Programming, it corresponds to the value function.
\vspace{-0.2cm}
\subsection{Value Functions Are Determined By The Solution Map}
Consider a nonlinear Ordinary Differential Equation (ODE) of the form
\vspace{-0.55cm}
\begin{align} \label{eqn: ODE}
& \dot{x}(t) = f(x(t), \mbf u(t)), \quad x(0)=x_0,
\end{align}
where $f:\R^n \times \R^{m} \to \R^n$, $ \mbf u: \R \to \R^{m}$, and $x_0 \in \R^n$.
\begin{defn}
We say the function $\phi_f$ is a solution map of the ODE given in Eq.~\eqref{eqn: ODE} on $[0,T] \subset \R$ if for all $ t \in [0,T]$
\begin{align*}
& \frac{\partial \phi_f(x_0, t,\mbf u)}{\partial t}= f(\phi_f(x_0,t, \mbf u), \mbf u(t)), \text{ and } \phi_f(x_0,0,\mbf u)=x_0.
\end{align*}
\end{defn}

{{\textbf{Definition of Admissible Inputs:}}} Given $\{c,g,f,\Omega,U,T\} \in \mcl M_{Lip}$ and associated family of OCPs of Form~\eqref{opt: optimal control probelm}, we now use the solution map to define the set of admissible input signals for the OCP initialized at $(x_0,t_0) \in \Omega \times [0,T]$. For this we use the shift operator, denoted $\tau_{s}:L^2([0,T], \R^m) \rightarrow L^2([0,T-s], \R^m)$, where $s \in [0,T]$, and defined by \begin{align} \label{defn: shift operator}
(\tau_{s} \mbf u)(t) := \mbf u(s+t) \text{ for all } t \in [0,T-s]. \end{align}
\begin{defn} \label{defn: soln map}
	For any $(x_0,t_0) \in \R^n\times[0,T]$,  we say $\mbf u$ is admissible, denoted $\mbf u \in \mcl U_{\Omega,U,f,T}(x_0,t_0)$, if $\mbf u: [t_0,T] \to U$ and there exists a unique solution map, $\phi_f$, such that
	\begin{align} \nonumber
	%&\phi_f(x_0, t-t_0, \tau_{t_0}\mbf u) \in \Omega,\\ \nonumber
	& \frac{\partial \phi_f(x_0, t-t_0, \tau_{t_0}\mbf u)}{\partial t}= f(\phi_f(x_0,t-t_0, \tau_{t_0}\mbf u), \mbf u(t )) \text{ for } t \in [t_0,T],\\ \label{eqn: soln map}
	&\phi_f(x_0, t-t_0, \tau_{t_0}\mbf u) \in \Omega \text{ for } t \in [t_0,T], \text{ and } \phi_f(x_0,0,\tau_{t_0} \mbf u)=x_0.
	\end{align}
\end{defn}

%\textcolor{red}{If we only consider initial conditions where no solution map can exit $\Omega$ than do we still need to assume $\Omega$ has the inward condition?}

%\textbf{Note:} Only considering optimal control problems with $c,g,f$ polynomial allows for SOS numerical implementation while
%\paragraph{Definition of a Value Function}
For a given family of OCPs of Form~\eqref{opt: optimal control probelm}, we now define the associated VF using the solution map, $\phi_f$. Lemma~\ref{lem: value function is lip} then shows that VFs are locally Lipschitz continuous.

\begin{defn} \label{defn: VF}
	For given $\{c,g,f,\Omega,U,T\} \in \mcl M_{Lip}$ we say $V^*: \R^n \times \R \to \R$ is a \textit{Value Function (VF)} of the associated family of OCPs if for $(x,t) \in \Omega \times [0,T]$, the following holds
{\small	\begin{align}
	V^*(x,t) \label{opt: optimal control general} = & \inf_{\mbf u \in \mcl U_{\Omega,U,f,T}(x,t)} \bigg\{\\ \nonumber
	 & \int_{t}^{T} c(\phi_f(x,s-t, \tau_t \mbf u),\mbf u(s),s) ds  + g(\phi_f(x,T-t, \tau_t \mbf u)) \bigg\},\notag
	\end{align} }
	where $\phi_f$ is as in Eq.~\eqref{eqn: soln map}. By convention if $\mcl U_{\Omega,U,f,T}(x,t) = \emptyset$ then $V^*(x,t)=\infty$.
\end{defn}

\begin{lem}[\cite{bressan2011viscosity}, Local Lipschitz continuity of VF] \label{lem: value function is lip}
	Consider some $\{c,g,f,\R^n,U,T\} \in \mcl M_{Lip}$. Then if $V^*$ satisfies Eq.~\eqref{opt: optimal control general}, we have that $V^* \in LocLip(\R^n \times [0,T], \R )$.
\end{lem}

%\begin{lem}[\cite{bressan2011viscosity}] \label{lem: value function is lip}
%	For given $\{c,g,f,\Omega,U,T\} \in \mcl M$ the associated value function, given by Equation \eqref{opt: optimal control general}, is Lipschitz continuous.
%\end{lem}
\vspace{-0.25cm}
\subsection{{ Value Functions are Solutions to the HJB PDE}} \label{sec: diss ineq arbitrary well}
Consider the family of OCPs associated with $\{c,g,f,\Omega,U,T\} \in \mcl M_{Lip}$. As shown in \cite{bertsekas2005dynamic}, a sufficient condition for a function $V^*$ to be a VF, is for $V^*$ to satisfy the Hamilton Jacobi Bellman (HJB) PDE, given in Eq.~\eqref{eqn: general HJB PDE}. However, for a general family of OCPs  of form $\{c,g,f,\Omega,U,T\} \in \mcl M_{Lip}$, solutions to the HJB PDE may not be differentiable, and hence classical solutions to the HJB PDE may not exist. For this reason, one typically uses a generalized notion of a solution to the HJB PDE called a viscosity solution, which is defined in~\cite{crandall1997viscosity} as follows.

\begin{defn} \label{defn: visocity soln}
	Consider the first order PDE
	\begin{align} \label{PDE}
	F(x,y(x),\nabla y(x)) =0 \quad \text{ for all } x \in \Omega,
	\end{align}
	where $\Omega \subset \R^n$ and $F \in C( \Omega \times \R \times \R^n, \R)$.
	
	We say $y \in C(\Omega)$ is a \textbf{viscosity sub-solution} of \eqref{PDE} if
	\begin{align*}
	F(x,y(x),p) \le 0 \quad \text{ for all } x \in \Omega \text{ and } p \in {D^+y(x)},
	\end{align*}
	where $D^+y(x):=\{p \in \R: \exists \Phi \in C^1(\Omega,\R) \text{ such that } \nabla \Phi(x)=p \text{ and } y-\Phi \text{ attains a local max at x}  \}$.
	
	Similarly, $y \in C(\Omega)$ is a \textbf{viscosity super-solution} of \eqref{PDE} if
	\begin{align*}
	F(x,y(x),p) \ge 0 \quad \text{ for all } x \in \Omega \text{ and } p \in {D^-y(x)}
	\end{align*}
	where $D^-y(x):=\{p \in \R: \exists \Phi \in C^1(\Omega,\R) \text{ such that } \nabla \Phi(x)=p \text{ and } y-\Phi \text{ attains a local min at x}  \}$.
	
	We say $y \in C(\Omega)$ is a \textbf{viscosity solution} of \eqref{PDE} if it is both a viscosity sub and super-solution.
\end{defn}

%In this paper we propose an algorithm to approximate value functions. To prove the convergence of this algorithm we require that the value function is the unique viscosity solution to the HJB PDE and also Lipschitz continuous, defined next..

%\begin{defn}
%	Consider sets $\Theta_1 \subset \R^n$ and $\Theta_2 \subset \R^m$. We say the function $F: \Theta_1 \to \Theta_2$ is Lipschitz continuous if there exists $K>0$ such that for all $x,y \in \Theta_1$
%	\begin{align*}
%	||F(x) - F(y)||_2 \le K ||x - y||_2.
%	\end{align*}
%	We refer to $K>0$ as the Lipschitz constant of $F$ and throughout the paper denote the Lipschitz constant of a Lipschitz continuous function $F$ by $L_F>0$.
%\end{defn}

%\paragraph{Smooth Solutions to the HJB}

\begin{thm}[\cite{bressan2011viscosity}, Uniqueness of VF] \label{thm: HJB value function}
	Consider the family of OCPs associated with the tuple $\{c,g,f,\R^n,U,T\} \in \mcl M_{Lip}$. Any function satisfying Eq.~\eqref{opt: optimal control general} is the unique viscosity solution of the HJB PDE
	\begin{align}  \nonumber
	&\nabla_t V(x,t) + \inf_{u \in U} \left\{ c(x,u,t) + \nabla_x V(x,t)^T f(x,u) \right\} = 0 \\
	& \hspace{3.75cm} \text{ for all } (x,t) \in \R^n \times [0,T] \label{eqn: general HJB PDE}\\ \nonumber
	& V(x,T)= g(x) \quad \text{ for all } x \in \R^n.
	\end{align}
\end{thm}
%As stated next, for $\{c,g,f,\R^n,U,T\} \in \mcl M_{Lip}$ if $V^*$ satisfies Eq.~\eqref{opt: optimal control general} then it is locally Lipschitz continuous and unique.

Note that Lemma~\ref{lem: value function is lip} and Theorem~\ref{thm: HJB value function} are only valid in the absence of state constraints ($\Omega = \R^n$). However, as we will show in Lemma~\ref{lem: value function of unconstrained and constrained problem are equal}, if the state constraints are sufficiently ``loose", then the unconstrained and constrained solutions coincide.

\vspace{-0.3cm}
\subsection{VFs Can Construct Optimal Controllers}
\vspace{-0.1cm} Given an OCP, we next show if a ``classical" differentiable solution to the HJB PDE~\eqref{eqn: general HJB PDE} associated with the OCP is known then a solution to the OCP can be constructed using Eqs.~\eqref{u} and \eqref{eqn:k}. We will refer to any $k:\Omega \times [0,T]\to U$ that satisfies Eqs.~\eqref{u} and \eqref{eqn:k} for some $V$ as a controller and say this is the optimal controller of the OCP if $V$ is the VF of the OCP.
\vspace{-0.2cm}
%If the HJB PDE, given in Eq.\eqref{eqn: general HJB PDE}, admits a ``classical" differentiable solution then it is possible to use this solution to construct an optimal controller.
\begin{thm}[\cite{liberzon2011calculus}] \label{thm: value functions construct optimal controllers}
	Consider the family of OCPs associated with tuple $\{c,g,f,\R^n,U,T\} \in \mcl M_{Lip}$. Suppose $V \in C^1(\R^n \times \R, \R)$ solves the HJB PDE \eqref{eqn: general HJB PDE}. Then $\mbf u^*: [t_0,T] \to U$ solves the OCP associated with $\{c,g,f,\R^n,U,T\}$ initialized at $(x_0,t_0) \in \R^n \times [0,T]$ if and only if
	\vspace{-0.1cm}\begin{align} \label{u}
	 &\mbf u ^*(t)= k(\phi_f(x_0,t,\mbf u^*), t) \text{ for all } t \in [t_0,T],\\ \label{eqn:k}
	& \text{where }	k(x,t) \in \arg \inf_{u \in U} \{ c(x,u,t) + \nabla_x V(x,t)^T f(x,u) \}.
		\end{align}
\end{thm}

If the function $V$ in Eq.~\eqref{eqn:k} is not a VF the resulting controller may no longer construct a solution to the OCP. In Section~\ref{sec:optimal controller approx} we will provide a bound on the performance of a constructed controller from a candidate VF based on how ``close" the candidate VF is to the true VF under the Sobolev norm.

\vspace{-0.2cm}
\section{The Feasibility Problem Of Finding VFs} \label{sec: feasibiity problem} %\vspace{-0.1cm}
Consider a family of OCPs associated with some $\{c,g,f,\Omega,U,T\} \in \mcl M_{Lip}$. Previously it was shown in Theorem~\ref{thm: value functions construct optimal controllers} that if $V \in C^1(\R^n \times \R, \R)$ is a solution to the HJB PDE~\eqref{eqn: general HJB PDE} then $V$ may be used to solve the family of OCPs using Eqs.~\eqref{u} and \eqref{eqn:k}. The question, now, is how to find such a $V$.

Let us consider the problem of finding a value function as an optimization problem subject to constraints imposed by the HJB PDE~\eqref{eqn: general HJB PDE}. This yields the following feasibility problem:
\vspace{-0.55cm}\begin{align} \label{feas: find VF}
& \text{Find } V \in C^1(\R^n \times \R, \R), \\ \nonumber
& \qquad \qquad \text{such that } V \text{ satisfies } \eqref{eqn: general HJB PDE}.
\end{align}
Note that our optimization problem of Form~\eqref{feas: find VF} is non-convex and may not even have a solution with sufficient regularity. For these reasons, we next propose a convex relaxation of Problem~\eqref{feas: find VF}. We first define sub-VFs and super-VFs that uniformly bound VFs either from above or bellow.

\vspace{-0.1cm}	
\begin{defn} \label{defn: sub and super value functions}
	We say the function $J: \R^n \times \R \to \R$ is a sub-VF to the family of OCPs associated with $\{c,g,f,\Omega,U,T\} \in \mcl M_{Lip}$ if
	
	\vspace{-0.8cm}
	\begin{align*}
	J(x,t) \le V^*(x,t) \text{ for all } t \in [0,T] \text{ and } x \in \Omega,
	\end{align*}
	for any $V^*$ satisfying Eq.\eqref{opt: optimal control general}. Moreover if
	\vspace{-0.15cm}\begin{align*}
	J(x,t) \ge V^*(x,t) \text{ for all } t \in [0,T] \text{ and } x \in \Omega,
	\end{align*}
	for any $V^*$ satisfying Eq.~\eqref{opt: optimal control general}, we say $J$ is a super-VF.
\end{defn}
\vspace{-0.6cm}
\subsection{A Sufficient Condition For A Function To Be A Sub-VF}

We now propose ``dissipation" inequalities and show that if a differentiable function satisfies such inequalities then it must be a sub-value function.

%\textcolor{red}{Need to prove prop 1 2 for the constrained problem. Need to assume $\mcl U \ne \emptyset$ in second case.}
\vspace{-0.1cm}
\begin{prop} \label{prop: diss ineq implies lower soln}
	For given $\{c,g,f,\Omega,U,T\} \in \mcl M_{Lip}$ suppose $J \in C^1(\R^n \times \R, \R)$ satisfies $\text{for all } (x,u,t) \in \Omega \times U \times (0,T)$
	\begin{align} \label{ineq: diss ineq for sub sol of HJB}
	& \nabla_t J(x,t) + c(x,u,t) + \nabla_x J(x,t)^T f(x,u) \ge 0, \\ \label{ineq: BC}
	& J(x,T) \le g(x).
	\end{align}
	Then $J$ is a sub-value function of the family of OCPs associated with $\{c,g,f,\Omega,U,T\}$.
\end{prop}
\begin{proof}
	Suppose $J \in C^1(\R^n \times \R, \R)$ satisfies Eqs.~\eqref{ineq: diss ineq for sub sol of HJB} and \eqref{ineq: BC}. Consider an arbitrary $(x_0,t_0) \in \Omega \times [0,T]$. If $\mcl U_{\Omega,U,f,T}(x_0,t_0) = \emptyset$ then $V^*(x_0,t_0) = \infty$. Clearly in this case $J(x_0,t_0) < V^*(x_0,t_0)$ as $J$ is continuous and therefore is finite over the compact region ${\Omega} \times [0,T]$. Alternatively if $\mcl U_{\Omega,U,f,T}(x_0,t_0) \ne \emptyset$, then for any $\mbf{\tilde{u}} \in \mcl U_{\Omega,U,f,T}(x_0,t_0)$, we have the following by Defn.~\ref{defn: soln map}:
	\begin{align*}
	 & \phi_f(x_0,t-t_0, \tau_{t_0} \mbf{\tilde{u}} ) \in \Omega \text{ for all } t \in  [t_0,T],\\
	 & \mbf{\tilde{u}} (t)\in U \text{ for all } t \in   [t_0,T].
	\end{align*}
	Therefore (using the shorthand $\tilde{x}(t):=\phi_f(x_0,t-t_0, \tau_{t_0} \mbf{\tilde{u}} )$), by Eq.~\eqref{ineq: diss ineq for sub sol of HJB} we have $\text{for all } t \in [t_0,T]$
	\begin{align*}
	& \nabla_t J(\tilde{x}(t),t) + c(\tilde{x}(t), \mbf{\tilde{u}} (t),t) + \nabla_x J(\tilde{x}(t),t)^T f(\tilde{x}(t), \mbf{\tilde{u}}(t) ) \ge 0.
	\end{align*}
Now, using the chain rule we deduce
		\begin{align*}
	& \frac{d}{dt} J(\tilde{x}(t),t) + c(\tilde{x}(t), \mbf{\tilde{u}} (t),t)  \ge 0 \text{ for all } t \in [t_0,T].
	\end{align*}
	Then, integrating over $t \in [t_0,T]$, and since $J(\tilde{x}(T),T) \le g(\tilde{x}(T))$ by Eq.~\eqref{ineq: BC}, we have
	\begin{align} \label{above ineq}
	J(x_0,t_0) \le \int_{t_0}^{T} & c(\tilde{x}(t),\mbf{\tilde{u}}(t),t) dt + g(\tilde{x}(T)).
	\end{align}
	Since Eq.~\eqref{above ineq} holds for all $\mbf{\tilde{u}} \in \mcl U_{\Omega,U,f,T}(x_0,t_0)$, we may take the infimum over $\mcl U_{\Omega,U,f,T}(x_0,t_0)$ to show that $J(x_0,t_0) \le V^*(x_0,t_0)$. As this argument can be used for any $(x_0,t_0) \in \Omega \times [0,T]$ it follows that $J$ is a sub-value function.
\end{proof}

\begin{defn} \label{defn: dissipative}
 For given $\{c,g,f,\Omega,U,T\} \in \mcl M_{Lip}$ we say a function $J \in C^1(\R^n \times \R, \R)$ is \textit{dissipative} if it satisfies Inequalities \eqref{ineq: diss ineq for sub sol of HJB} and \eqref{ineq: BC}.
\end{defn}
Dissipative functions are viscosity sub-solutions (as per Defn.~\ref{defn: visocity soln}) to the HJB PDE~\eqref{eqn: general HJB PDE}. Moreover, by Prop.~\ref{prop: diss ineq implies lower soln} a dissipative function is a sub-VF. However, a sub-VF need not be dissipative or a viscosity sub-solution to the HJB PDE.

\vspace{-0.25cm}
\subsection{A Convex Relaxation Of The Problem Of Finding VFs}

The set of functions satisfying Eqs.~\eqref{ineq: diss ineq for sub sol of HJB} and \eqref{ineq: BC} is convex as Eqs.~\eqref{ineq: diss ineq for sub sol of HJB} and \eqref{ineq: BC} are linear in terms of the unknown variable/function $J$. Furthermore, for given $\{c,g,f,\Omega,U,T\} \in \mcl M_{Lip}$, any function which satisfies the HJB PDE \eqref{eqn: general HJB PDE} also satisfies Eqs.~\eqref{ineq: diss ineq for sub sol of HJB} and \eqref{ineq: BC}. This allows us to propose the following convex relaxation of the problem of finding a VF (Problem~\eqref{feas: find VF}):

\vspace{-0.6cm}
\begin{align} \label{feas: convex  find VF}
& \text{Find } J \in C^1(\R^n \times \R, \R), \\ \nonumber
& \qquad \qquad \text{such that } J \text{ satisfies } \eqref{ineq: diss ineq for sub sol of HJB} \text{ and } \eqref{ineq: BC}.
\end{align}
%
%, and thus is a VF, then it is also clearly dissipative, as per Definition \ref{defn: dissipative}.

\vspace{-0.5cm}
\subsection{A Polynomial Tightening Of The Problem Of Finding VFs} \vspace{-0.05cm}
Problem~\eqref{feas: convex  find VF} is convex. However, a function $J$, feasible for Problem~\eqref{feas: convex  find VF} (and hence dissipative), may be arbitrarily far from the VF. For instance, in the case $c(x,u,t) \ge 0$ and  $0 \le g(x)  <M$, the constant function $J(x,t)\equiv - C$ is dissipative for any $C>M$. Thus, by selecting sufficiently large enough $C>M$, we can make $||J -V||$ arbitrary large, regardless of the chosen norm,  $|| \cdot ||$. %This implies that controllers designed using solutions to Problem~\eqref{feas: convex  find VF} (as opposed to the true VF) will perform poorly.

To address this issue, we propose a modification of Problem~\eqref{feas: convex  find VF}, wherein we include an objective of Form~$ \int_{\Lambda \times [0,T]} w(x,t) J(x,t) dx dt$, parameterized by a compact domain of interest $\Lambda \subset \R^n$ and weight $w \in L^1 (\Lambda \times [0,T], \R^+)$ (we use the weight, $w$, in Prop.~\ref{prop: SOS converges Dv}). Specifically, for given $\{c,g,f,\Omega,U,T\} \in \mcl M_{Lip}$ and $d \in \N$, consider the optimization problem:

\vspace{-0.8cm}
\begin{align} \label{opt: convex L1 norm}
&J_d \in \arg \max_{J \in \mcl P_d(\R^n \times \R, \R)} \int_{\Lambda \times [0,T]} w(x,t) J (x,t) dx dt\\ \nonumber
& \text{subject to: } 	 \nabla_t J(x,t) + c(x,u,t) + \nabla_x J(x,t)^T f(x,u) > 0 \\ \nonumber
& \hspace{2.75cm}  \text{ for all } x \in \Omega, t \in (0,T), u \in U, \\ \nonumber
& \hspace{1.65cm} J (x,T) < g(x) \text{ for all } x \in \Omega.
\end{align}

\vspace{-0.1cm}
{Maximizing} $ \int_{\Lambda \times [0,T]} w(x,t) J(x,t) dx dt$ minimizes the weighted $L^1$ norm $ \int_{\Lambda \times [0,T]} w(x,t) |V(x,t)-J(x,t)| dx dt$. The restriction to polynomial solutions $J \in \mcl P_d(\R^n \times \R, \R)$ makes the problem finite-dimensional.% (discretization is an alternative approach for tightening optimization problems not pursued in this paper).

%

%
%Optimization problem \eqref{opt: convex L1 norm} is a LP problem, and thus convex, as the objective function and the set of functions that satisfy the constraints are linear in the decision variable $J$.

\vspace{-0.4cm}

\section{A Sequence Of Dissipative Polynomials That Converge To The VF In Sobolev Space} \label{Section: mollification and polynomial approx}
For a given $\{c,g,f,\Omega,U,T\} \in \mcl M_{Lip}$, in Eq.~\eqref{opt: convex L1 norm}, we proposed a sequence of optimization problems, indexed by $d \in \N$, each instance of which yields a dissipative function $J_d \in \mcl P_d(\R^n \times \R, \R)$. In this section, we prove that {$\lim_{d\rightarrow \infty}\norm{J_d-V}_{L^1(\Lambda \times (0,T), \R)}\rightarrow 0$} where $V$ is the VF associated with the OCP $\{c,g,f,\Omega,U,T\} \in \mcl M_{Lip}$. To accomplish this proof, we divide the section into three subsections, wherein we find the following.
\begin{enumerate}[label=(\Alph*)]
	\item {In Prop.~\ref{prop:sontag smooth approx} we show that for any $V \in Lip(\Omega \times [0,T],\R)$ that satisfies the dissipation-type inequality in Eq.~\eqref{diss ineq ae} and any $\eps>0$ there exists a dissipative function $J_\eps \in C^\infty(\Omega \times [0,T], \R)$  such that $|| J_\eps-V||_{W^{1,p}(\Omega \times [0,T],\R)}< \eps$.}
	\item In Theorem~\ref{thm: approx true value functions} we show that for every $\eps>0$, there exists $d \in \N$ and dissipative $P_\eps \in \mcl P_d(\R^n \times \R, \R)$ such that $||P_\eps- V||_{W^{1,p}(\Omega \times [0,T],\R)} < \eps$, for any value function, $V$, associated with $\{c,g,f,\Omega,U,T\} \in \mcl M_{Lip}$.
	\item For any positive weight $w$, Prop.~\ref{prop: convergence of inf L1} shows that if $J_d$ solves \eqref{opt: convex L1 norm} for $d \in \N$, then {$\lim_{d \to \infty} {||w(J_d - V)||_{L^1(\Lambda \times (0,T), \R)}}=0$} for any VF, $V$, associated with $\{c,g,f,\Omega,U,T\} \in \mcl M_{Lip}$.
\end{enumerate}
 %we show that VF's can be approximated arbitrarily well by polynomial functions that satisfy the dissipation inequalities \eqref{ineq: diss ineq for sub sol of HJB} and \eqref{ineq: BC}. Our proof takes two steps. First we approximate the VF by an infinitely differentiable function by taking the mollification of the VF and then we approximate the mollification by a polynomial.
 \vspace{-0.4cm}
\subsection{Existence Of Smooth Dissipative Functions That Approximate The VF Arbitrarily Well Under The $W^{1,p}$ Norm}
In this section we create a sequence of smooth (elements of $C^\infty(\R^n \times \R, \R)$) functions that converges, with respect to the $W^{1,p}$ norm, to any Lipschitz function, $V$, satisfying the dissipation-type inequality in Eq.~\eqref{diss ineq ae}. This subsection uses some aspects of mollification theory. For an overview of this field, we refer to~\cite{evans2010partial}.

\paragraph{Mollifiers}
The standard mollifier, $\eta \in C^\infty(\R^n \times \R, \R)$ is defined as

\vspace{-0.7cm}
\begin{align} \label{fun: mollifier}
\eta(x,t):=\begin{cases}
C \exp\left(\frac{1}{||(x,t)||_2^2 -1}\right) \quad \text{when } ||(x,t)||_2<1,\\
0 \quad \text{when } ||(x,t)||_2 \ge 1,
\end{cases}
\end{align}
where $C>0$ is chosen such that $\int_{\R^n \times \R} \eta(x,t) dx dt =1$.

For $\sigma>0$ we denote the scaled standard mollifier by $\eta_\sigma \in C^\infty(\R^n \times \R, \R)$ such that

\vspace{-0.55cm}
\begin{equation*}
\eta_\sigma(x,t):= \frac{1}{\sigma^{n+1}} \eta\left(\frac{x}{\sigma}, \frac{t}{\sigma}\right).
\end{equation*} Note, clearly $\eta_\sigma(x,t)=0$ for all $(x,t) \notin B(0,\sigma)$.

\paragraph{Mollification of a Function (Smooth Approximation)}
Recall from Section \ref{sec: notation B} that for open sets $\Omega \subset \R^n$, $(0,T) \subset \R$, and $\sigma>0$ we denote $<\Omega \times (0,T)>_\sigma:=\{x \in \Omega \times (0,T): B(x,\sigma) \subset \Omega \times (0,T) \}$. Now, for each $\sigma>0$ and function $V \in L^1(\Omega \times (0,T), \R)$ we denote the \textit{$\sigma$-mollification} of $V$ by $[V]_\sigma: <\Omega \times (0,T)>_\sigma \to \R$, where
\vspace{-0.1cm}
\begin{align}
[V]_\sigma(x,t) & := \int_{\R^n \times \R} \eta_\sigma(x-z_1,t-z_2)V(z_1,z_2)dz_1 dz_2\\ \nonumber
& =\int_{B(0,\sigma)} \eta_\sigma(z_1,z_2) V(x-z_1,t-z_2) dz_1 dz_2   .
\end{align}
To calculate the derivative of a mollification we next introduce the concept of weak derivatives.
\begin{defn} \label{def: weak deriv}
	For $\Omega \subset \R^n$ and $F \in L^1(\Omega, \R)$ we say any $H \in L^1(\Omega, \R)$ is the weak $i \in \{1,..,n\}$-partial derivative of $F$ if
	
	\vspace{-0.55cm}
	\begin{align*}
	\int_{\Omega} F(x) \frac{\partial}{\partial x_i} \alpha(x) dx = - \int_{\Omega} H(x) \alpha(x) dx,  \text{ for } \hspace{-0.05cm} \alpha \in C^\infty(\R^n,\R).
	\end{align*}
\end{defn}
Weak derivatives are ``essentially unique". That is if $H_1$ and $H_2$ are both weak derivatives of a function $F$ then the set of points where $H_1(x) \ne H_2(x)$ has measure zero. If a function is differentiable then its weak derivative is equal to its derivative in the ``classical" sense. We will use the same notation for the derivative in the ``classical" sense and in the weak sense.

%We now define the space of functions that have weak derivatives, called the Sobolev space. For $k \in \N$ and $1 \le p \le \infty$ we denote
%\begin{align*}
%W^{k,p}(\Omega,\R):=\{u\in L^p(\Omega,\R): D^\alpha u \in L^p(\Omega,\R) \forall |\alpha| \le k \}.
%\end{align*}
%The Sobolev norm for $u \in W^{k,p}(\Omega,\R)$ is defined as
%\begin{align*}
%||u||_{W^{k,p}(\Omega, \R)}:= \begin{cases}
%\left( \sum_{|\alpha| \le k} \int_\Omega (D^\alpha u(x))^p dx \right)^{\frac{1}{p}} \text{ if } 1 \le p < \infty\\
%\sum_{|\alpha| \le k} \esssup_{ x \in \Omega  } \{|D^\alpha u(x) |\} \text{ if } p= \infty.
%\end{cases}
%\end{align*}
%In the case $k=0$ the $W^{0,p}(\Omega,\R)$-norm is the $L^p(\Omega,\R)$-norm. Thus we use the notation $|| \cdot ||_{L^p(\Omega,\R)} :=|| \cdot ||_{W^{0,p}(\Omega,\R)} $.

In the next proposition we state some useful properties about Sobolev spaces and mollifications taken from \cite{evans2010partial}.

\begin{prop}[\cite{evans2010partial}] \label{prop:mollification}
	For $1 \le p < \infty$ and $k \in \N/\{0\}$ we consider $V \in W^{k,p}(E, \R)$, where $E\subset \R^{n+1}$ is an open bounded set, and its $\sigma$-mollification $[V]_\sigma$. Recalling from Section \ref{sec: notation B} that for an open set $\Omega \subset \R^n$ and $\sigma>0$ we denote $<\Omega>_\sigma:=\{x \in \Omega: B(x,\sigma) \subset \Omega \}$, the following holds:
	\begin{enumerate}
		%	\item For all $1 \le  p < \infty$ we have $W^{1,p}(\Omega \times (0,T), \R) \subset W^{1, \infty}(\Omega \times (0,T), \R)$.
		\item For all $\sigma>0$ we have $[V]_\sigma \in C^\infty( <E>_\sigma, \R )$.
		%We need the approximation to be continuous over the closure of the set because we have a boundary condition at T to satisfy.
		\item For all $\sigma>0$ we have $\nabla_t [V]_\sigma(x,t)= [ \nabla_t  V]_\sigma(x,t)$ and $\nabla_x [V]_\sigma(x,t)= [ \nabla_x  V]_\sigma(x,t)$ for $(x,t) \in {<E>_\sigma}$, where  $\nabla_t  V$ and $\nabla_x V$ are weak derivatives.
		\item  If $V \in C(E,\R)$ then for any compact set $K \subset E$ we have $\lim_{\sigma \to 0}\sup_{(x,t) \in K}| V(x,t)- [V]_\sigma(x,t)|=0$.
		\item (Meyers-Serrin Local Approximation) For any compact set $K \subset E$ we have $\lim_{\sigma \to 0}\norm{ [V]_\sigma-V}_{W^{k,p}(K,\R)}=0$.
	\end{enumerate}
\end{prop}

%\begin{rem}
%	By the Sobolev Embedding Theorem differentiability, in terms of $k \in \N$, can be ``traded" for integrability, in terms of $p \in [1, \infty]$. Therefore, for $k \ge 1$ convergence in $W^{k,p}$ for all $1 \le p < \infty$ implies convergence in $L^q$ for all $1 \le q \le \infty$.
%\end{rem}

%It can be shown that a $\sigma$-mollification of a sobolev function tends to that function under the induced sobolev norm as $\sigma \to 0$. Next we will state the stronger result from \cite{evans2010partial} that in the special case when the function is Lipschitz continuous the mollification converges not just in the Sobolev induced norm but uniformly. Moreover, the Sobolev space $W^{1,\infty}(\Omega,\R)$ is exactly the space of Lipschitz continuous functions when $\Omega$ is of class $C^1$ (defined previously in Definition \ref{defn: C1 set}).
%\begin{prop}
%Let $\Omega \subset \R^n$ be an open bounded set such that $\Omega \times (0,T)$ is of class $C^1$ and $V:\Omega \times (0,T) \to \R$ be Lipschitz continuous. Then the following holds
%\begin{itemize}
%	\item $V \in W^{1, \infty}(\Omega \times (0,T), \R)$.
%	\item For all $\eps>0$ there exists $\delta>0$ such that for all $0 \le \sigma< \delta$ we have
%			\begin{equation*}
%			\sup_{(x,t) \in {\Omega} \times (0,T)} \bigg|V(x,t) - [V]_\sigma(x,t) \bigg|<\eps.
%			\end{equation*}
%\end{itemize}
%\end{prop}
%\textcolor{red}{When do we need the set bounded?}

\paragraph{Approximation of Lipschitz functions satisfying a dissipation-type inequality} We now show that for any Lipschitz function, $V$, satisfying the dissipation-type inequality in Eq.~\eqref{diss ineq ae}, $V$ can be approximated arbitrarily well by a smooth function, $J_\eps$, that also satisfies the dissipation-type inequality in Eq.~\eqref{diss ineq ae}. {We use a similar proof strategy first appearing in~\cite{kurzwel1963inversion} and also later appearing in~\cite{wilson1969smoothing,lin1996smooth,teel2000smooth}}.
\begin{lem} \label{lem: sontag smooth approx}
	Let $E \subset \R^{n+1}$ be an open bounded set, $ \Omega \subset \R^n $ be such that $\Omega \times (0,T) \subseteq E$, where $T>0$, $U \subset \R^m$ be a compact set, $f \in Lip( \Omega \times U, \R^n)$, $c \in Lip(\Omega\times U \times [0,T], \R)$, and $V \in Lip(E,\R)$ such that	
	\vspace{-0.2cm}\begin{align} \label{diss ineq ae}
	 \essinf_{(x,t) \in \Omega \times (0,T)} \hspace{-0.12cm} \{ \hspace{-0.0cm} \nabla_t V(x,t) + & \nabla_xV(x,t)^T \hspace{-0.07cm} f(x,u) + c(x,u,t) \hspace{-0.0cm} \} \hspace{-0.05cm} \ge \hspace{-0.05cm} 0,
	%||\nabla_t V(x,t) + & \nabla_xV(x,t)^Tf(x,u) + c(x,u,t) ||_{L^\infty(\Omega \times (0,T), \R)} \ge 0,
	\end{align}
	where the derivatives, $\nabla_t V$ and $\nabla_x V$, are weak derivatives.%, and known to exist by Rademacher's Theorem (Theorem \ref{thm: Rademacher theorem}) using the Lipschitz continuity of $V(x,t)$.
	
	Then for any compact set $K \subset E$, $1 \le p <\infty$ and for all $\eps>0$ there exists $J_\eps \in C^{\infty}(K, \R)$ such that 	
	\vspace{-0.05cm}\begin{align} \label{compact close smooth}
	&  ||V - J_\eps||_{W^{1,p}(K,\R)}<\eps \text{ and } \hspace{-0.05cm} \sup_{(x,t) \in K} \hspace{-0.05cm} |V(x,t) - J_\eps(x,t) |<\eps,
	\end{align}
	
	\vspace{-0.25cm}
	
	and for all $(x,t)\in K \cap (\Omega \times (0,T))$ and $ u \in U$
	\vspace{-0.05cm}\begin{align} \label{compact diss ineq}
	\nabla_t & J_\eps(x,t) + \nabla_x J_\eps(x,t)^Tf(x,u) + c(x,u,t) \ge -\eps.
	\end{align}
\end{lem}
\vspace{-0.1cm}
\begin{proof}
	Suppose $V$ satisfies Eq.~\eqref{diss ineq ae}, $K \subset E$ is a compact set, $1 \le p < \infty$, and $\eps>0$. By Rademacher's Theorem (Theorem~\ref{thm: Rademacher theorem}) $V$ is weakly differentiable with essentially bounded derivative. Therefore $V \in W^{1,\infty}(E, \R)$ and hence $V \in W^{1,p}(E, \R)$. Now Prop.~\ref{prop:mollification} (Statements 3 and 4) can be used to show there exists $\sigma_1 >0$ such that for any $0 \le \sigma<\sigma_1$ we have
	\begin{align} \label{compact close smooth proof}
	 ||V - [V]_{\sigma_1} ||_{W^{1,p}(K,\R)}<\eps \text{ and } \hspace{-0.05cm} \sup_{(x,t) \in K} \hspace{-0.1cm}|V(x,t) - [V]_{\sigma_1}(x,t) |<\eps.
	%   \sup_{(x,t) \in K}|V - [V]_{\sigma_1} |<\eps
	\end{align}
	{\normalsize Select $\sigma_2>0$ small enough so $K \subset <E>_{\sigma_2}$ (which can be done as $E$ is open). Select $0<\sigma_3<\frac{\eps}{L_V L_f+2L_c}$, where $L_V,L_f,L_c>0$ are the Lipschitz constant of the functions $V$, $f$, and $c$ respectively. We now have the following for all $\sigma_4 < \min\{\sigma_3,\sigma_2\}$, $u \in U$ and $(x,t) \in K \cap (\Omega \times (0,T))$, }
	\vspace{-0.2cm}
	{\small \begin{align}\label{compact diss ineq proof}
		& \nabla_t [V]_{\sigma_4}(x,t) + \nabla_x [V]_{\sigma_4}(x,t)^T f(x,u) + c(x,u,t)\\ \nonumber
		& =  [ \nabla_t V]_{\sigma_4}(x,t) +  [ \nabla_x V]_{\sigma_4}(x,t)^T f(x,u)  + c(x,u,t)\\ \nonumber
		& = \int_{B(0,\sigma_4)} \eta_{\sigma_4}(z_1,z_2)\bigg(\nabla_t V(x-z_1,t-z_2) \\ \nonumber
		& \qquad + \nabla_x V(x-z_1,t-z_2)^Tf(x-z_1,u) + c(x-z_1,u,t-z_2)\bigg)dz_1 dz_2  \\ \nonumber
		& \qquad - \int_{B(0,\sigma_4)} \eta_{\sigma_4}(z_1,z_2) \nabla_x V(x-z_1,t-z_2)^T \\ \nonumber
		& \hspace{4cm} \bigg(  f(x-z_1,u) - f(x,u) \bigg) dz_1 dz_2\\ \nonumber
		& \quad  - \int_{B(0,\sigma_4)} \eta_{\sigma_4}(z_1,z_2)\bigg( c(x-z_1,u,t-z_2)  - c(x,u,t) \bigg) dz_1 dz_2\\ \nonumber
		& \ge \essinf_{(z_1,z_2)\in B(0,\sigma_4)} \bigg\{ \nabla_t V(x-z_1,t-z_2) \\ \nonumber
		& \hspace{1cm} + \nabla_x V(x-z_1,t-z_2)^Tf(x-z_1,u) + c(x-z_1,u,t-z_2)\bigg\}\\ \nonumber
		& - \esssup_{(z_1,z_2)\in B(0,\sigma_4)} \bigg\{||\nabla_x V(x-z_1,t-z_2) ||_2\bigg \}\\ \nonumber  & \hspace{2cm} \esssup_{(z_1,z_2)\in B(0,\sigma_4)} \bigg\{ ||f(x-z_1,u) - f(x,u)||_2 \bigg\}\\ \nonumber
		& -\esssup_{(z_1,z_2)\in B(0,\sigma_4)} \bigg\{  |c(x-z_1,u,t-z_2)  - c(x,u,t)| \bigg\}\\ \nonumber
		& \ge -L_V\esssup_{(z_1,z_2)\in B(0,\sigma_4)} \bigg\{ ||f(x-z_1,u) - f(x,u)||_2 \bigg\}\\ \nonumber
		& \qquad  -\esssup_{(z_1,z_2)\in B(0,\sigma_4)} \bigg\{  |c(x-z_1,u,t-z_2)  - c(x,u,t)| \bigg\}\\ \nonumber
		& \ge -L_V L_f\esssup_{(z_1,z_2)\in B(0,\sigma_4)}\bigg\{ ||z_1||_2 \bigg\}   -L_c\esssup_{(z_1,z_2)\in B(0,\sigma_4)} \bigg\{  ||z_1||_2 + |z_2| \bigg\}\\ \nonumber
		&= -(L_V L_f +2L_c) \sigma_4 \ge -\eps.
		\end{align} }
	
	\vspace{-0.2cm}
	
\noindent	The first equality of Eq.~\eqref{compact diss ineq proof} follows since $\nabla_t [V]_{\sigma_4}(x,t) = [\nabla_t V]_{\sigma_4}(x,t)$ and $\nabla_x [V]_{\sigma_4}(x,t) = [\nabla_x V]_{\sigma_4}(x,t)$ for all $(x,t) \in K \subset <E>_{\sigma_4}$ by Prop.~\ref{prop:mollification} (Statement 2). The first inequality follows by the monotonicity property of integration and the Cauchy Swartz inequality. Since $V$ is Lipschitz $\esssup_{(x,t) \in E} ||\nabla_x V(x,t)||_2< L_V$ by Rademacher's Theorem (Theorem \ref{thm: Rademacher theorem}). Now the second inequality follows by using \eqref{diss ineq ae} together with $\esssup_{(x,t) \in E} ||\nabla_x V(x,t)||_2< L_V$. The third inequality follows by the Lipschitz continuity of $f$ and $c$. Finally the fourth inequality follows by the fact $\sigma_4< \sigma_3< \frac{\eps}{L_VL_f +L_c }$.
	
	Now define $J_\eps(x,t):= [V]_\sigma(x.t)$ where $0<\sigma< \min\{\sigma_1,\sigma_4\}$. It follows that $J_\eps \in C^{\infty}(K, \R)$ by Prop.~\ref{prop:mollification} (Statement 1). Moreover $J_\eps$ satisfies Eqs.~\eqref{compact close smooth} and \eqref{compact diss ineq} by Eqs.~\eqref{compact close smooth proof} and \eqref{compact diss ineq proof}.
\end{proof}
In Lemma~\ref{lem: sontag smooth approx} we showed that for any given function, $V \in Lip(E, \R)$, any compact subsets $K \subset E$, any $\eps>0$, and any $1 \le p <\infty$, there exists a smooth function, $J_\eps$, satisfying Eq.~\eqref{compact diss ineq}, such that $||V-J_\eps||_{W^{1,p}(K, \R)}<\eps$. We next show this ``local" result over compact subsets, $K$, can be extended to a ``global" results over the entire domain, $E$. To do this we use Theorem~\ref{thm: partition of unity}, stated in Section \ref{sec: appendix 3}. Given an open cover of $E$, Theorem~\ref{thm: partition of unity} states that there exists a family of functions, called a partition of unity. In the next proposition we use partitions of unity together with the ``local" approximates of the Lipschitz function, $V$, to construct a smooth ``global" approximation of $V$ over the entire domain $E$.
\begin{prop} \label{prop:sontag smooth approx}
		Let $E \subset \R^{n+1}$ be an open bounded set, $ \Omega \subset \R^n $ be such that $\Omega \times (0,T) \subseteq E$, where $T>0$, $U \subset \R^m$ be a compact set, $f \in Lip( \Omega \times U, \R^n)$, $c \in Lip(\Omega\times U \times [0,T], \R)$, and $V \in Lip(E,\R)$ satisfies Eq.~\eqref{diss ineq ae}.
%		
%		such that	}
%	\begin{align} \label{diss ineq ae 2}
%	\essinf_{(x,t) \in \Omega \times (0,T)} \bigg\{\nabla_t V(x,t) + & \nabla_xV(x,t)^Tf(x,u) + c(x,u,t) \bigg\} \ge 0,
%	\end{align}
%	where the derivatives, $\nabla_t V$ and $\nabla_x V$, are understood in the weak sense.
%	
	Then for all $1 \le p <\infty$ and $\eps>0$ there exists $J \in C^{\infty}( E , \R)$ such that
	\begin{align} \label{global smooth approx value function}
||V - J||_{W^{1,p}(E,\R)}<\eps \text{ and } \sup_{(x,t) \in E}|V(x,t) - J(x,t) |<\eps,
	\end{align}
	and for all $(x,u,t)\in \Omega \times U
	\times (0,T)$  % $u \in U$
	\begin{align} \label{global diss ineq}
	\nabla_t & J(x,t) + \nabla_x J(x,t)^Tf(x,u) + c(x,u,t) \ge -\eps.
	\end{align}
\end{prop}
\begin{proof}
	Let us consider the family of sets $E_i=\{x \in E: \sup_{y \in {\partial E }}||x - y||_2<\frac{1}{i} \}$ for $i \in \N$. It follows $\{E_i\}_{i=1}^\infty$ is an open cover (Defn.~\ref{defn: open cover}) for $E$ and thus by Theorem~\ref{thm: partition of unity} there exists a smooth partition of unity, $\{\psi_i \}_{i=1}^\infty \subset C^\infty(E,\R)$, that satisfies Statements~1 to 4 of Theorem~\ref{thm: partition of unity}.
	
	For $\eps>0$ Lemma \ref{lem: sontag smooth approx} shows that for each $i \in \N$ there exists a function $J_i \in C^\infty( {\overline{ E_i } }, \R)$ such that
	\begin{align} \label{global approx}
	& \sup_{(x,t) \in E_i} |V(x,t) - J_i(x,t)|<\frac{\eps}{2^{i+1}(1+\tau_i+ \theta_i)},\\
%	&||V - J_\eps||_{L^\infty(\overline{ E_i },\R)}<\frac{\eps}{2^{i+1}(1+\tau_i)} \\ \label{global approx W norm}
	& \label{global approx W norm} ||V - J_i||_{W^{1,p}(\overline{ E_i },\R)}<\frac{\eps}{2^{i+1}(1+\tau_i + \theta_i)}, \\ \nonumber
	\nabla_t & J_i(x,t) + \nabla_x J_i(x,t)^Tf(x,u) + c(x,u,t) \ge -\frac{\eps}{2^{i+1}(1+\tau_i+ \theta_i)}\\
	& \qquad \text{for all } (x,t) \in \overline{ E_i} \cap (\Omega \times (0,T)), u \in U, \label{global diss approx}
	\end{align}
	where we denote $\tau_i := \sup_{(x,u,t) \in \Omega \times U \times (0,T)}\{ |\nabla_t \psi_i(x,t) + \nabla_x \psi_i(x,t)^Tf(x,u)| \} \ge 0$ and $\theta_i:=\left( \max_{|\alpha| \le 1} \sup_{(x,t) \in E} {|D^\alpha \psi_i(x,t)|^p} \right)^p \ge 0$; which is well defined and finite as $\Omega \times U \times (0,T)$ is bounded and $\psi_i$ is smooth.
	
	Now, let us define $J(x,t):= \sum_{i=1}^\infty \psi_i(x,t) J_i(x,t)$, we will show $J \in C^\infty(E, \R)$ and that $J$ satisfies Eqs.~\eqref{global smooth approx value function} and \eqref{global diss ineq}.
	
	It follows $J \in C^\infty(E, \R)$ by Theorem~\ref{thm: partition of unity}. To see this we note for each $i \in \N$ we have $ \psi_i \in C^\infty(E, \R)$ and $\psi_i(x,t)=0$ outside $E_i$ implying $\psi_i J_i \in C^\infty(E, \R)$. Moreover, for each $(x,t) \in E$ there exists an open set, $S \subseteq E$, where only a finite number of $\psi_i$ are nonzero. Therefore it follows that the function $J$ is a finite sum of infinitely differentiable functions and thus $J$ is also infinitely differentiable.
	
	We now show $J$ satisfies Eq.~\eqref{global smooth approx value function}. We first show $\norm{V - J}_{W^{1,p}(E,\R)}<\eps$:
		\begin{align} \label{similar argument}
&	\norm{V - J}_{W^{1,p}(E,\R)} = \norm{V - \sum_{i=1}^\infty \psi_i J_i}_{W^{1,p}(E,\R)}\\ \nonumber
& = \norm{\sum_{i=1}^\infty \psi_i(V - J_i)}_{W^{1,p}(E,\R)} \le \sum_{i=1}^\infty \norm{\psi_i(V -  J_i)}_{W^{1,p}(E,\R)}\\ \nonumber
		&= \sum_{i=1}^\infty \norm{\psi_i(V -  J_i)}_{W^{1,p}(\bar{E}_i,\R)} \le  \sum_{i=1}^\infty \theta_i \norm{V -  J_i}_{W^{1,p}(\bar{E}_i,\R)}\\ \nonumber
	& < \sum_{i=1}^\infty \bigg(\frac{\eps + \theta_i}{2^{i+1}(1+\tau_i + \theta_i)} \bigg) < \eps.
	\end{align}
	The second equality of Eq.~\eqref{similar argument} follows since partitions of unity satisfy $\sum_{i=1}^\infty \psi_i(x,t) \equiv 1$ by Theorem~\ref{thm: partition of unity}. The first inequality follows by the triangle inequality. The third equality follows since partitions of unity satisfy $\psi_i(x,t) =0$ outside of $E_i$ for all $i \in \N$ by Theorem~\ref{thm: partition of unity}. The third inequality follows by Eq.~\eqref{global approx W norm}. The fourth inequality follows as $\sum_{i=1}^\infty \frac{1}{2^i}=1$. Now, by a similar augment to Eq.~\eqref{similar argument}, using Eq.~\eqref{global approx} rather than Eq.~\eqref{global approx W norm}, it also follows $\sup_{(x,t) \in E}|V(x,t) - J(x,t)|<\eps$ and thus $J$ satisfies Eq.~\eqref{global smooth approx value function}.
	
%	For $(x,t) \in E$ we have
%	\begin{align*}
%	|V(x,t) - J(x,t)| & = \bigg|V(x,t) - \sum_{i=1}^\infty \psi_i(x,t)J_i(x,t) \bigg|\\
%	& = \bigg|\sum_{i=1}^\infty \psi_i(x,t)V(x,t) - \sum_{i=1}^\infty \psi_i(x,t)J_i(x,t) \bigg|\\
%	& \le \sum_{i=1}^\infty \psi_i(x,t) |V(x,t) - J_i(x,t)|\\
%	& < \sum_{i=1}^\infty \psi_i(x,t) \bigg(\frac{\eps}{2^{i+1}(1+\tau_i)} \bigg) < \eps,
%	\end{align*}
%	where the second equality follows as $\sum_{i=1}^\infty \psi_i(x,t) \equiv 1$ (by Theorem \ref{thm: partition of unity}), the first inequality follows by the triangle inequality, the second inequality follows by \eqref{global approx}, the third inequality follows as $\sum_{i=1}^\infty \frac{1}{2^i}=1$.
	
	Next we will show $J$ satisfies Eq.~\eqref{global diss ineq}. Before doing this we first prove a preliminary identity. Specifically,
	\begin{align} \label{identity psi}
	\sum_{i=1}^\infty \bigg(\nabla_t \psi_i(x,t) + \nabla_x \psi_i(x,t)^Tf(x,u) \bigg)=0,
	\end{align}
	for all $(x,t)\in \Omega \times (0,T) \subseteq E$ and $u \in U$. This follows {because only finitely many $\psi_i$'s are non-zero for each $(x,t) \in E$ and thus it follows $\sum_{i=1}^\infty \psi_i(x,t)$ is a finite sum of infitely differentiable functions. Therefore, we can interchange derivatives and summations, thus since $\sum_{i=1}^\infty \psi_i(x,t) \equiv 1$ it follows that $\nabla_t \bigg(\sum_{i=1}^\infty \psi_i(x,t) \bigg) =\sum_{i=1}^\infty \nabla_t \psi_i(x,t) = 0 $. Similarly for each $j \in \{1,...,n\}$ we have $\sum_{i=1}^\infty \frac{ \partial \psi_i(x,t)}{\partial x_j} = 0$ which implies $\sum_{i=1}^\infty \nabla_x \psi_i(x,t) = 0 \in \R^n $.}
	
	Now, it follows $J$ satisfies Eq.~\eqref{global diss ineq} since
	\begin{align} \label{52}
	& \nabla_t  J(x,t) + \nabla_x J(x,t)^Tf(x,u) + c(x,u,t)\\ \nonumber
	& = \sum_{i=1}^\infty \bigg( \psi_i(x,t)  ( \nabla_t J_i(x,t) + \nabla_x J_i(x,t)^T f(x,u) + c(x,u,t) ) \bigg) \\ \nonumber
	& \quad + \sum_{i=1}^\infty \bigg( J_i(x,t)(\nabla_t \psi_i(x,t) + \nabla_x \psi_i(x,t)^Tf(x,u) ) \bigg) \\ \nonumber
	& \ge \frac{-\eps}{2} + \sum_{i=1}^\infty  (J_i(x,t)- V(x,t))(\nabla_t \psi_i(x,t) + \nabla_x \psi_i(x,t)^Tf(x,u) ) \\ \nonumber
	& \ge - \eps,
	\end{align}
	for all $(x,t)\in \Omega \times (0,T) \subseteq E$ and $u \in U$. The first equality of Eq.~\eqref{52} follows by the chain rule and the fact $\sum_{i=1}^\infty \psi_i(x,t) \equiv 1$. The first inequality follows by Eqs.~\eqref{global diss approx} and \eqref{identity psi}. The second inequality follows by Eq.~\eqref{global approx} and $\sum_{i=1}^\infty \frac{1}{2^i}=1$.
\end{proof}
\subsection{Existence Of Dissipative Polynomials That Can Approximate The VF Arbitrarily well Under The $W^{1,p}$ Norm} \label{subsec: polynomials can approx value functions}
%he following theorem provides the main result of the paper, which says that for suitably smooth problem data, with associated value function $V$, there exists a polynomial subvalue function, $V_l$, such that $\norm{V-V_l}\le \epsilon$ and $\norm{\nabla(V-V_l)}\le \epsilon$.
%
Previously, in Prop.~\ref{prop:sontag smooth approx}, we showed for any $V \in Lip(\Omega \times [0,T],\R)$ satisfying Eq.~\eqref{diss ineq ae} there exists a smooth function $J$ that also satisfies Eq.~\eqref{diss ineq ae} and approximates $V$ with arbitrary accuracy under the Sobolev norm. We now use this result to show for any VF, associated with some family OCPs $\{c,g,f,\Omega,U,T\} \in \mcl M_{Lip}$, there exists a dissipative polynomial, $V_l$, that approximates the VF arbitrarily well with respect to the Sobolev norm. Our proof uses Theorem~\ref{thm:Nachbin}, found in Appendix~\ref{sec: appendix 3}, that shows differentiable functions, such as $J$, can be approximated up to their first order derivatives over compact sets arbitrarily well by polynomials. Prop.~\ref{prop:sontag smooth approx} only gives the existence of a smooth approximation, $J$, when the VF is Lipschitz continuous. Lemma~\ref{lem: value function is lip} shows the VF, associated with a family of OCPs, is locally Lipschitz when $\Omega= \R^n$ (which is not a compact set). Unfortunately, Theorem~\ref{thm:Nachbin} can only be used for polynomial approximation over compact sets. Thus, before proceeding we first give a sufficient condition for a VF, associated with a family of OCPs with compact state constraints, to be Lipschitz continuous over some set $\Lambda \subset \Omega$.

 % Clearly in order to use Prop.~\ref{prop:sontag smooth approx} the VF must be Lipshitz continuous which is true by Lemma \ref{lem: value function is lip} when $\Omega = \R^n$.

%  However, in order to achieve this we consider the family of state constrained OCP's $\{c,g,f,\Omega,U,T\} \in \mcl M_{Lip}$ such that $\Omega \subset \R^n$ is compact.

%\textbf{MOOOOOOOOOOOOOOOOOOOOOOOOOOOOOOOOOOOVE}

\paragraph{Lipschitz continuity of VFs associated with a family of state constrained OCPs} Consider the family of OCPs $\{c,g,f,\Omega,U,T\} \in \mcl M_{Lip}$. If the state is constrained ($\Omega \ne \R^n$), the associated VF can be discontinuous and is no longer uniquely defined as the viscosity solution of the HJB PDE. Next, in Lemma \ref{lem: value function of unconstrained and constrained problem are equal}, we give a sufficient condition that when satisfied implies VFs, associated with a family of state constrained OCPs, are equal to the unique locally Lipschitz continuous VF of the state unconstrained OCP over some subset $\Lambda \subseteq \Omega$,  {and hence are Lipschitz continuous over $\Lambda$}. To state Lemma~\ref{lem: value function of unconstrained and constrained problem are equal} we first define the forward reachable set.

%We first define forward reachable sets as the set of points the solution map can reach in some finite time and the set of initial conditions such that the solution map can reach some target set in a finite time; we formally define these two sets next.
\begin{defn} \label{defn: reachbale set}
	For $X_0 \subset \R^n$, $\Omega \subseteq \R^n$, $U \subset \R^m$, $f: \R^n \times \R^m \to \R^n$ and $S \subset \R^+$, define
	{ \begin{align*}
		FR_f  (X_0,\Omega,U, & S)  := \bigg\{  y \in \R^n \;: \; \hspace{-0.075cm} \text{there exists } \hspace{-0.075cm} x \in X_0,  T \in S, \text{and} \\ & \mbf u \in \mcl U_{\Omega,U,f,T}(x,0)
		\text{ such that } \phi_f(x,T,\mbf u)=y  \bigg\}.
		\end{align*} } \normalsize
\end{defn}

\begin{lem} \label{lem: value function of unconstrained and constrained problem are equal}
	Consider $\{c,g,f,\Omega,U,T\} \in \mcl M_{Lip}$ and any function $V_1: \Omega \times [0,T] \to \R$ that satisfies Eq.~\eqref{opt: optimal control general}. Let $V_2: \R^n \times [0,T] \to \R$ be the VF for the unconstrained problem  $\{c,g,f,\R^n,U,T\}$. If $\Lambda \subseteq \Omega$ is such that
	\begin{align} \label{reachable set condition}
	FR_f  (\Lambda,\R^n,U, [0,T]) \subseteq \Omega,
	\end{align}
	then $V_1(x,t)= V_2(x,t)$ for all $(x,t) \in \Lambda \times [0,T]$.
\end{lem}
\begin{proof}
	To show $V_1(x,t)= V_2(x,t)$ for all $(x,t) \in \Lambda \times [0,T]$ we must prove $\mcl U_{\Omega,U,f,T}(x,t)= \mcl U_{\R^n,U,f,T}(x,t)$ for all $(x,t) \in \Lambda \times [0,T]$.
	
	For any $(x,t) \in \Lambda \times [0,T]$ if $\mbf u \in \mcl U_{\Omega,U,f,T}(x,t)$ then clearly $\mbf u \in \mcl U_{\R^n,U,f,T}(x,t)$, thus $\mcl U_{\Omega,U,f,T}(x,t) \subseteq \mcl U_{\R^n,U,f,T}(x,t)$. On the other hand if $\mbf u \in \mcl U_{\R^n,U,f,T}(x,t)$ then by Defn.~\ref{defn: soln map} it follows $\mbf u (s) \in U$ for all $s \in [t, T]$ and that there exists a unique map, denoted by $\phi_f(x,s,\mbf u)$, that satisfies the following for all $s \in [t,T]$
	\begin{align*}
	& \frac{\partial \phi_f(x, s-t,\tau_t \mbf u)}{\partial s}= f(\phi_f(x,s-t, \tau_t \mbf u), \mbf u(s)), \text{ } \phi_f(x,0,\tau_t \mbf u)=x.
	\end{align*}
	To show $\mbf u \in \mcl U_{\Omega,U,f,T}(x,t)$ we need $\phi_f(x,s-t,\tau_t \mbf u) \in \Omega \text{ for all } s \in [t,T],$ which is equivalent to
	\begin{align}\label{99} \phi_f(x,s, \tilde{\mbf u}) \in \Omega \text{ for all } s \in [0,T-t],\end{align} where $\tilde{\mbf u}= \tau_t \mbf u \in \mcl U_{\Omega,U,f,T-t}(x,0)$. Eq.~\eqref{99} then follows trivially by Eq.~\eqref{reachable set condition}.
\end{proof}
Alternative sufficient conditions that imply a VF, associated with some family of state constrained OCPs, is Lipschitz continuous and the unique viscosity solution of the HJB PDE include: the Inward Pointing Constraint Qualification (IPCQ) \cite{soner1986optimal} \cite{frankowska2013discontinuous}, the Outward Pointing Constraint Qualification (OPCQ) \cite{frankowska2000existence}, and epigraph characterization of VF's \cite{altarovici2013general}.

%\textbf{END MOOOOOOOOOOOOOOOOOOOOOOOOOOOOOOOOOOOVE}
\paragraph{Approximation of VFs by dissipative polynomials}
Considering a family of OCPs $\{c,g,f,\Omega,U,T\} \in \mcl M_{Lip}$, and assuming there exists a set $\Lambda \subseteq \Omega$ that satisfies Eq.~\eqref{reachable set condition}, we now prove the existence of dissipative polynomial functions that can approximate the any VF of $\{c,g,f,\Omega,U,T\} \in \mcl M_{Lip}$ arbitrarily well under the Sobolev norm.
\begin{thm} \label{thm: approx true value functions}
	For given $\{c,g,f,\Omega,U,T\} \in \mcl M_{Lip}$ suppose $\Lambda \subseteq \Omega$ is a bounded set that satisfies \eqref{reachable set condition}, then for any function $V$ satisfying Eq.~\eqref{opt: optimal control general}, $1 \le p < \infty$, and $\eps>0$ there exists $V_l \in \mcl P( \R^n \times \R, \R)$ such that
	\begin{align}
	\label{V_l close to V}
	&  \norm{V - V_l}_{W^{1,p}(\Lambda \times [0,T],\R)} < \eps,\\ \label{V_l close to V infinty norm}
	& \sup_{(x,t) \in \Lambda \times [0,T]}|V(x,t) - V_l(x,t) |<\eps,\\
	\label{V_l is a subvalue function}
	& V_l(x,t) \le V(x,t)    \text{ for all } t \in [0,T] \text{ and } x \in \Omega, \\ \label{V_l diss ineq}
	& \nabla_t V_l(x,t) + c(x,u,t) + \nabla_x V_l(x,t)^T f(x,u) > 0 \\ \nonumber
	& \hspace{3.25cm}  \text{ for all } x \in \Omega, t \in (0,T), u \in U, \\ \label{V_l boundary ineq}
	& V_l (x,T) < g(x) \text{ for all } x \in \Omega.
	\end{align}
\end{thm}

\begin{proof}
		 %We therefore only need to show there exists a polynomial satisfying \eqref{V_l close to V}, \eqref{V_l close to V infinty norm}, \eqref{V_l diss ineq}, and \eqref{V_l boundary ineq}.
	Let $\eps>0$. Suppose $V$ satisfies Eq.~\eqref{opt: optimal control general}. Rather than approximating $V$, defined for a family of OCPs on the compact set $\Omega$, we instead approximate the unique VF, denoted by $V^*$, associated with the family of OCPs where $\Omega=\R^n$. It is easier to approximate $V^*$ compared to $V$ as $V^*$ has the following useful properties: By Lemma~\ref{lem: value function is lip}, $V^*$ is locally Lipschitz continuous; and by Theorem~\ref{thm: HJB value function}, $V^*$ is the unique viscosity solution of the HJB PDE~\eqref{eqn: general HJB PDE}. Furthermore, as $\Lambda$ satisfies Eq.~\eqref{reachable set condition}, Lemma~\ref{lem: value function of unconstrained and constrained problem are equal} implies
	\begin{align} \label{V star = V}
V^*(x,t)= V(x,t) \text{ for all } (x,t) \in \Lambda \times [0,T].
	\end{align}
	
This proof is structured as follows. We first use Prop.~\ref{prop:sontag smooth approx} to approximate $V^*$ by an infinitely differentiable function denoted as $J_\delta$. Then using Theorem~\ref{thm:Nachbin}, found in Appendix~\ref{sec: appendix 3}, we approximate $J_\delta$ by a polynomial $P_\delta$. Finally, to ensure Inequalities~\eqref{V_l diss ineq} and \eqref{V_l boundary ineq} are satisfied, a correction term $\rho$ is subtracted from $P_\delta$, creating the function $V_l(x,t):=P_\delta(x,t)-\rho(t) $ that we show satisfies Eqs.~\eqref{V_l close to V} to~\eqref{V_l boundary ineq}.

	%Let $\eps>0$. We will now construct a polynomial, $V_l(x,t)$, that satisfies \eqref{V_l close to V}, \eqref{V_l diss ineq}, and \eqref{V_l boundary ineq}. To do this we first approximate the value function, $V^*(x,t)$, by a smooth function.
	
%	Let $V^*$ denote the VF of the state unconstrained problem $\{c,g,f,\R^n,U,T\}$. By Theorem \ref{thm: HJB value function}, $V^*$ is the unique viscosity solution of the HJB PDE \eqref{eqn: general HJB PDE}. Furthermore, by Lemma \ref{lem: value function is lip}, $V^*$ is uniformly Lipschitz continuous. Moreover, since $\Lambda$ satisfies Eq.~\eqref{reachable set condition}, Lemma \ref{lem: value function of unconstrained and constrained problem are equal} implies that if $V$ satisfies \eqref{opt: optimal control general}, then $V^*(x,t)= V(x,t)$ for all $(x,t) \in \Lambda \times [0,T]$.

%
%Therefore, if there exists some $ V_l \in \mcl P(\R^n \times \R, \R)$ such that $\norm{V^* - V_l}_{W^{1,p}(\Lambda \times [0,T],\R)} < \eps$ and $\sup_{(x,t) \in \Lambda \times [0,T]}|V^*(x,t) - V_l(x,t) |<\eps$, then we have that $V_l$ Eqs.~\eqref{V_l close to V} and \eqref{V_l close to V infinty norm} imply that there exists . Here we will only show $\norm{V^* - V_l}_{W^{1,p}(\Lambda \times [0,T],\R)} < \eps$ as $\sup_{(x,t) \in \Lambda \times [0,T]}|V^*(x,t) - V_l(x,t) |<\eps$ will follow by a similar argument, exchanging the Sobolev norm with the uniform norm.

Since $\Omega$ is compact, there exists some open bounded set $E \subset \R^{n+1}$ of finite measure which contains $\overline{\Omega \times (0,T) }$. Since ${V}^* \in LocLip(\R^n \times \R, \R)$ (by Lemma \ref{lem: value function is lip}) and $E \subset \R^n$ is bounded it follows $V^* \in Lip(E \times [0,T], \R)$. Then by Rademacher's theorem (See Theorem~\ref{thm: Rademacher theorem} in Section \ref{sec: appendix 3}), $V^*$ is differentiable almost everywhere in $E$. Moreover, as $V^*$ is the unique viscosity solution to the HJB PDE, the following holds for all $u \in U$ and almost everywhere in $(x,t) \in \Omega \times (0,T) \subset E$.
	\begin{align*}
	& \nabla_t {V}^*(x,t) + c(x,u,t) + \nabla_x {V}^* (x,t)^T f(x,u) \\
	& \ge \nabla_t {V}^*(x,t) + \inf_{u \in U}\{ c(x,u,t) + \nabla_x {V}^* (x,t)^T f(x,u) \} =0
	\end{align*}
	This implies that the following holds for all $u \in U$	
	\begin{align*}
	\essinf_{(x,t) \in \Omega \times (0,T)} \bigg\{\nabla_t {V}^* (x,t) + & \nabla_x {V}^* (x,t)^Tf(x,u) + c(x,u,t) \bigg\} \ge 0.
	\end{align*}
	%$0< \delta< \frac{\eps}{4 + M T}$, where $M:=\sup_{(x,u) \in \Omega \times U} ||f(x,u)||_2$ is well defined and finite by the continuity of $f(x,u)$,
	
	Therefore, we conclude that $V^* $ satisfies Eq.~\eqref{diss ineq ae}. Thus, by Prop.~\ref{prop:sontag smooth approx}, for any $\delta>0$ there exists $J_\delta \in C^{\infty}(E, \R)$ such that
	{ {	\begin{align} \label{J close to V}
			& \norm{{V}^*  - J_\delta}_{W^{1,p}(E,\R)}<\delta , \\ \label{J dissipition ineq} %\text{ and} \sup_{(x,t) \in \Lambda \times [0,T]}|V^*(x,t) - J_\delta(x,t) |<\eps
			& \nabla_t J_\delta(x,t) + \nabla_x J_\delta(x,t)^Tf(x,u) + c(x,u,t) \ge -\delta\\ \nonumber
			& \hspace{4cm} \text{for all} (x,t) \in \Omega \times (0,T) .
			\end{align} }}
	In particular, let us choose $\delta>0$ such that
	\begin{align} \label{eps delta}
	\delta< \frac{\eps}{2 + (2 + 4T+2MT)(T\mu(\Lambda))^{\frac{1}{p}}},
	\end{align}	
where $M:=\sup_{(x,u) \in \Omega \times U} ||f(x,u)||_2<\infty$ and $\mu(\Lambda)< \infty$ is the Lebesgue measure of $\Lambda$.

	We now approximate $J_\delta \in C^\infty(E,\R)$ by a polynomial function. Theorem~\ref{thm:Nachbin}, found in Section~\ref{sec: appendix 3}, shows there exists $P_\delta \in \mcl P(\R^n \times \R, \R)$ such that {$\text{for all } (x,t) \in  E $}
	\begin{align} \label{J close to p}
	& |J_\delta(x,t)-P_\delta(x,t)| < \delta. \\ \label{J nabla t close to P}
	& |\nabla_t J_\delta(x,t)- \nabla_t P_\delta(x,t)|< \delta. \\ \label{J nabala x close to P}
	& ||\nabla_x J_\delta(x,t)- \nabla_x P_\delta(x,t)||_2< \delta.\\ \label{J close to P W norm}
	& \norm{ J_\delta - P_\delta}_{W^{1,p}(E,\R)}< \delta.
	\end{align}
%	\vspace{-0.5cm}
%	Now,
	\begin{align} \nonumber
	&\text{Now, } \norm{V^*    - P_\delta}_{W^{1,p}(E,\R)} = \norm{V^* - J_\delta + J_\delta - P_\delta}_{W^{1,p}(E,\R)} \\
	& \le \norm{V^* - J_\delta}_{W^{1,p}(E,\R)} + \norm{ J_\delta - P_\delta}_{W^{1,p}(E,\R)} < 2 \delta   \label{P close to V},
	\end{align}
	where the first inequality follows by the triangle inequality, and the second inequality follows from Eq.~\eqref{J close to V} and Eq.~\eqref{J close to P W norm}.
	
	By a similar argument to Inequality~\eqref{P close to V} we deduce,
	 	\begin{align}
	&\sup_{(x,t) \in E}|V^*(x,t)    - P_\delta(x,t)| < 2 \delta   \label{P close to V uniformly}.
	\end{align}
	Furthermore,
	\begin{align} \nonumber
	&  \nabla_t P_\delta(x,t) + \nabla_x P_\delta(x,t)^Tf(x,u) + c(x,u,t)  \\ \nonumber
	& \ge \bigg(\nabla_t P_\delta(x,t) + \nabla_x P_\delta(x,t)^Tf(x,u) + c(x,u,t) \bigg)\\ \nonumber
	& \qquad -\delta - \bigg( \nabla_t J_\delta(x,t) + \nabla_x J_\delta(x,t)^Tf(x,u) + c(x,u,t) \bigg) \\ \nonumber
	& = -\delta + \bigg(\nabla_t P_\delta(x,t) - \nabla_t J_\delta(x,t) \bigg) \\ \nonumber
	& \qquad - \bigg( \nabla_x J_\delta(x,t) - \nabla_x P_\delta(x,t) \bigg)^T f(x,u)\\ \nonumber
	& > -\delta - \delta - ||\nabla_x J_\delta(x,t) - \nabla_x P_\delta(x,t)||_2||f(x,u)||_2\\ \label{P diss ineq}
	& > -(2+M) \delta  \text{ for all } (x,t) \in \Omega \times (0,T),
	\end{align}
	The first inequality of Eq.~\eqref{P diss ineq} follows by Eq.~\eqref{J dissipition ineq}. The second inequality follows by Eq.~\eqref{J nabla t close to P} and the Cauchy Schwarz inequality. The third inequality follows by Eq.~\eqref{J nabala x close to P}. %Finally we recall $M:=\sup_{(x,u) \in \Omega \times U} ||f(x,u)||_2$.
	
	%As shown in \eqref{P close to V} the polynomial, $P_\delta(x,t)$, approximates the value, $V^*(x,t)$, over the open set $\Omega \times (0,T)$. However, our approximation of $V^*(x,t)$ is also required to satisfy \eqref{V_l boundary ineq}, a condition at the boundary of this domain at $t=T$. We next derive an inequality at the boundary $t=T$ using the Lipschitz properties of $P_\delta(x,t)$ and $V^*(x,t)$.
	%
	%For all $x \in \Omega$ and $0< \theta< \frac{\delta}{L_{P_\delta}+ L_{V^*}}$, where $L_{P_\delta}$ and $L_{V^*}$ are the Lipschitz constants of $P_\delta(x,t)$ and $V^*(x,t)$ over $\Omega \times [0,T]$ respectively, we have
	%	\begin{align} \nonumber
	%&	P_\delta(x,T)  = P_\delta(x,T) -P_\delta(x,T- \theta) +P_\delta(x,T- \theta) \\ \nonumber
	%	& \qquad  -V^*(x,T- \theta) + V^*(x,T- \theta) - V^*(x,T) + V^*(x,T)\\ \nonumber
	%	& \le |P_\delta(x,T) - P_\delta(x,T-\theta)| + |P_\delta(x,T- \theta) -V^*(x,T- \theta)| \\ \nonumber
	%	& \qquad + |V^*(x,T- \theta) - V^*(x,T)| + V^*(x,T)\\
	%	& \le g(x)+ (L_{P_\delta} + L_V)\theta + 2\delta \le g(x) + 3 \delta, \label{P boundary condition}
	%	\end{align}
	%where the above inequality follows from the Lipschitz continuity of $P_\delta(x,t)$ and $V^*(x,t)$, the fact $V^*$ satisfies the boundary condition in \eqref{eqn: general HJB PDE}, and \eqref{J close to V}.

	Moreover, we have that
	\begin{align} \nonumber
	P_\delta(x,T) & = P_\delta(x,T) - V^*(x,T) + V^*(x,T)\\
	& < g(x)+  2 \delta \text{ for all } x \in \Omega. \label{P boundary condition}
	\end{align}
	This inequality follows from the fact that $V^*(x,T)=g(x)$ since $V^*$ satisfies the boundary condition in the HJB PDE~\eqref{eqn: general HJB PDE}, and Eq.~\eqref{P close to V uniformly}.
	
	We now construct $V_l$ from $P_\delta$. Let us denote the correction function $\rho(t):= (2+M)  (T-t)\delta + 2 \delta$, where $M=\sup_{(x,u) \in \Omega \times U} ||f(x,u)||_2$. We define $V_l$ as
	\begin{equation} \label{V_l}
	V_l(x,t):= P_\delta(x,t) - \rho(t).
	\end{equation}
	We now find that $V_l$ satisfies Inequality~\eqref{V_l diss ineq} since we have
	\begin{align*}
	& \nabla_t V_l(x,t) + c(x,u,t) + \nabla_x V_l(x,t)^T f(x,u)\\
	& =  \bigg(\nabla_t P_\delta(x,t) + \nabla_x P_\delta(x,t)^Tf(x,u) + c(x,u,t) \bigg) + (2+M) \delta\\
	& > 0, \qquad \text{for all }(x,t) \in \Omega \times (0,T) ,
	\end{align*}
where the above inequality follows from Eq.~\eqref{P diss ineq}.
	
	We next show $V_l$ satisfies Inequality~\eqref{V_l boundary ineq}:
	\begin{align*}
	V_l(x,T)= P_\delta(x,T) - 2\delta < g(x) \text{ for all } x \in \Omega,
	\end{align*}
	where the above inequality follows by Eq.~\eqref{P boundary condition}.
	
	Now, since $V_l$ satisfies Eqs.~\eqref{V_l diss ineq} and~\eqref{V_l boundary ineq} it follows $V_l$ satisfies Eq.~\eqref{V_l is a subvalue function} by Prop.~\ref{prop: diss ineq implies lower soln}.
	
%	We next show \eqref{V_l is a subvalue function}. As $V_l(x,t)$ satisfies \eqref{V_l diss ineq} and \eqref{V_l boundary ineq} it follows by Proposition \ref{prop: diss ineq implies lower soln} that $V_l(x,t)$ is a sub value function, that is $V(x,t) - V_l(x,t) \ge 0$ for all $\forall t \in [0,T] \text{ and } x \in \Omega$.

	To show that $V_l$ satisfies Inequality~\eqref{V_l close to V}, we first we derive a bound on the norm of the correction function $\rho$.
	
	\vspace{-0.4cm}
{ \small	\begin{align*}
	& \norm{\rho }_{W^{1,p}(\Lambda \times [0,T],\R)}= \left( \int_{\Lambda \times [0,T]} |(2+M)  (T-t)\delta
	 + 2 \delta|^p dx dt \right)^{\frac{1}{p}} \\
	 & + \left( \int_{\Lambda \times [0,T]} |(2+M) \delta|^p dx dt \right)^{\frac{1}{p}} \le (2 + 4T+2MT)(T\mu(\Lambda))^{\frac{1}{p}}\delta.
	\end{align*} }

\vspace{-0.2cm}
\normalsize
Now, by~Eqs.~\eqref{V star = V}, \eqref{eps delta} and \eqref{P close to V},
	\begin{align} \label{V close to V_l sobolev}
	& \norm{V - V_l}_{W^{1,p}(\Lambda \times [0,T],\R)} =  \norm{V^* - V_l}_{W^{1,p}(\Lambda \times [0,T],\R)} \\ \nonumber
	&  = \norm{V^* - P_\delta - \eta}_{W^{1,p}(\Lambda \times [0,T],\R)}  \\ \nonumber
	& \le \norm{V^* - P_\delta }_{W^{1,p}(E,\R)} + \norm{\eta }_{W^{1,p}(\Lambda \times [0,T],\R)} \\ \nonumber
	& \le 2 \delta + (2 + 4T+2MT)(T\mu(\Lambda))^{\frac{1}{p}} \delta < \eps.
	\end{align}
	By a similar argument to Eq.~\eqref{V close to V_l sobolev} we deduce $V_l$ satisfies Eq.~\eqref{V_l close to V infinty norm}
	
We conclude that $V_l$, defined in Eq.~\eqref{V_l}, satisfies Eqs.~\eqref{V_l close to V infinty norm}, \eqref{V_l is a subvalue function}, \eqref{V_l diss ineq}, and \eqref{V_l boundary ineq} thus completing the proof.
\end{proof}
\vspace{-0.4cm}
\subsection{{Our Family Of Optimization Problems Yield A Sequence Of Polynomials That Converge To A VF Under The $L^1$ Norm}}\label{sec: opt for value function approx}
Consider some $\{c,g,f,\Omega,U,T\} \in \mcl M_{Lip}$ and suppose the sequence $\{J_d\}_{d \in \N}$ solves each instance of the optimization problem given in Eq.~\eqref{opt: convex L1 norm} for $d \in \N$. We next use Theorem~\ref{thm: approx true value functions} to show that the sequence, $\{J_d\}_{d \in \N}$ converges to any VF associated with the family of OCP's $\{c,g,f,\Omega,U,T\} \in \mcl M_{Lip}$ with respect to the weighted $L^1$ norm as $d \to \infty$.
\begin{prop} \label{prop: convergence of inf L1}
For given $\{c,g,f,\Omega,U,T\} \in \mcl M_{Lip}$ and positive integrable function $w\in L^1(\Omega \times [0,T],\R^+)$ suppose $\Lambda \subseteq \Omega$ satisfies Eq.~\eqref{reachable set condition} then
\begin{align} \label{convergence in L1}
\lim_{d \to \infty} \int_{\Lambda \times [0,T]} w(x,t) |V(x,t) -J_d(x,t)| dx dt  = 0,
\end{align}
where $V$ is any function satisfying Eq.~\eqref{opt: optimal control general}, and $J_d \in \mcl P_{d}(\R^n \times \R, \R)$ is any solution to Optimization Problem~\eqref{opt: convex L1 norm} for $d \in \N$.
\end{prop}
\begin{proof} Suppose $V$ satisfies the theorem statement.
To show Eq.~\eqref{convergence in L1} we must show that for any $\eps>0$ there exists $N \in \N$ such that
	\begin{align*}
\int_{\Lambda \times [0,T]} w(x,t) |V(x,t) -J_d(x,t)| dx dt  < \eps \text{ for all } d \ge N.
	\end{align*}
	
	Since by assumption $\Lambda$ satisfies Eq.~\eqref{reachable set condition}, we can use Theorem~\ref{thm: approx true value functions} (from Section \ref{subsec: polynomials can approx value functions}) to show that for any $\delta>0$ there exists dissipative $V_l \in \mcl P(\R^n \times \R, \R)$ feasible to Optimization Problem~\eqref{opt: convex L1 norm} and is such that
\begin{align} \nonumber
\esssup_{(x,t) \in \Lambda \times [0,T]} |V(x,t) -V_l(x,t)|  < \delta.
\end{align}
For our given $\epsilon>0$, by selecting $\delta<{\eps} / { \int_{\Lambda \times [0,T]} w(x,t) dx dt}$ (Note if $\int_{\Lambda \times [0,T]} w(x,t) dx dt = 0$, Eq.~\eqref{convergence in L1} already holds and the proof is complete) we have a $V_l$ such that
	\begin{align} \label{eps}
& \int_{\Lambda \times [0,T]} w(x,t) |V(x,t) - V_l(x,t)| dx dt  \\ \nonumber
&\le \int_{\Lambda \times [0,T]} w(x,t) dx dt \esssup_{(x,t) \in \Lambda \times [0,T]} |V(x,t) -V_l(x,t)|\\ \nonumber
&<\delta \int_{\Lambda \times [0,T]} w(x,t) dx dt  <\eps  .
\end{align}

Now define $N:=\deg(V_l)$ and denote the solution to Problem~\eqref{opt: convex L1 norm} for $d \ge N$ as $J_d \in \mcl P_N(\R^n \times \R, \R)$. As $V_l$ is feasible to Problem~\eqref{opt: convex L1 norm} for all $d \ge N$, it follows the objective function evaluated at $J_d$ is greater than or equal to the objective function evaluated at $V_l$; that is
\begin{align}\label{j}
\int_{\Lambda \times [0,T]} \hspace{-0.75cm} w(x,t) J_d (x,t) dx dt \ge \int_{\Lambda \times [0,T]} \hspace{-0.75cm} w(x,t) V_l(x,t) dx dt \hspace{-0.05cm} \text{ for } d \ge N.
\end{align}
%Now,
\begin{align} \label{80}
&\text{Now, } \int_{\Lambda \times [0,T]}  w(x,t) |V(x,t) -J_d(x,t)| dx dt \\ \nonumber
& \quad  = \int_{\Lambda \times [0,T]} w(x,t) V(x,t) - w(x,t) J_d(x,t) dx dt \\ \nonumber
& \quad \le \int_{\Lambda \times [0,T]}w(x,t)| V(x,t) -V_l(x,t)| dx dt < \eps \text{ for all } d \ge N.
\end{align}
The equality in Eq.~\eqref{80} follows since $J_d(x,t) \le V(x,t)$ for all $(x,t) \in \Omega \times[0,T]$ (Prop.~\ref{prop: diss ineq implies lower soln}). The first inequality follows by a combination of Eq.~\eqref{j} and the inequality $V_l(x,t) \le V(x,t)$ for all $(x,t) \in \Omega \times[0,T]$. Finally, the second inequality follows by Eq.~\eqref{eps}.
\end{proof}
%We note that selecting $w(x,t) \equiv 1$
\vspace{-0.6cm}
%\section{{Solutions To A Family Of SOS Problems Can Construct A Sequence Of Polynomials That Converge To The VF}} \label{sec: SOS implemtation}
\section{{A Family Of SOS Problems That Yield Polynomials That Converge To The VF}} \label{sec: SOS implemtation}
%To solve the Optimization Problem \eqref{opt: convex L1 norm} we must find a function, $V_l$, that satisfies several inequality constraints. Determining whether a polynomial satisfies an inequality constraint over a compact set, $(x,u,t) \in \Omega \times U \times [0,T]$, has the same difficulties as proving a polynomial is globally positive ($f(x)>0$ $\forall x \in \R^n$); \cite{Blum_1998} has shown this problem to be NP-hard. However it can be shown that testing whether a polynomial is Sum-of-Squares (SOS), and hence positive, is equivalent to solving a semidefinite program (SDP). Although not all positive polynomials are SOS, this gap can be made arbitrarily small, see \cite{Hilbert_1888}.
\vspace{-0.1cm}
Consider some $\{c,g,f,\Omega,U,T\} \in \mcl M_{Lip}$ and denote $\{J_d\}_{d \in \N}$ as the sequence of solutions to the optimization problem found in Eq.~\eqref{opt: convex L1 norm}. We have shown in {Prop.~\ref{prop: convergence of inf L1}} that the sequence of functions, $\{J_d\}_{d \in \N}$, converge to any VF associated with the family of OCPs $\{c,g,f,\Omega,U,T\} \in \mcl M_{Lip}$ with respect to the $L^1$ norm. The indexed polynomial optimization problems in Eq.~\eqref{opt: convex L1 norm} may now be readily tightened to more tractable SOS optimization problems.

%We have shown the sequence of functions $\{J_d\}_{d \in \N}$, where each $J_d$ solves the Optimization Problem~\eqref{opt: convex L1 norm} for  $d \in \N$, converges to the VF with respect to the $L^1$ norm. The indexed polynomial optimization problems in Eq.~\eqref{opt: convex L1 norm} may now be readily tightened to more tractable SOS optimization problems.

%
%Unfortunately, the Optimization Problem~\eqref{opt: convex L1 norm} is intractable as determining whether a polynomial satisfies an inequality constraint over a compact set, $(x,u,t) \in \Omega \times U \times [0,T]$, is NP-hard \cite{Blum_1998}.
%
Specifically, for each $d \in \N$, we tighten the polynomial optimization problem in Eq.~\eqref{opt: convex L1 norm} to the SOS optimization problem given in Eq.~\eqref{opt: SOS for sub soln of finite time}. We {later} show in Prop.~\ref{prop: SOS converges} that the sequence of solutions to the SOS problem given in Eq.~\eqref{opt: SOS for sub soln of finite time} yield polynomials, $\{P_d\}_{d \in \N}$, indexed by degree $d \in \N$, that converge to the VF (with respect to the $L^1$ norm) as $d \to \infty$.

For our SOS implementation we consider a special class of OCPs, given next in Defn.~\ref{defn: polynomial optimal control}. %This class has the property that functions $c,g,f$ are polynomial, and sets $\Omega$ and $U$ are semi-algebraic.

\begin{defn} \label{defn: polynomial optimal control}
	We say the six tuple $\{c,g,f,\Omega,U,T\}$ is a \textit{polynomial optimal control problem} or $\{c,g,f,\Omega,U,T\} \in \mcl M_{Poly}$ if the following holds
	\begin{enumerate}
		\item $c \in \mcl P(\R^n\times \R^m \times \R, \R )$ and $g \in \mcl P(\R^n,\R)$.
		\item $f \in \mcl P(\R^n\times \R^m, \R^n )$.
		%	\item The initial condition set is semialgebriac, that is $X_0=\{x \in \R^n: h_X(x)<1\}$ where $h_X \in \mcl P(\R^n, \R)$.
		\item There exists $h_\Omega \in \mcl P(\R^n, \R)$ such that $\Omega=\{x \in \R^n: h_\Omega(x) \ge 0\}$.
		\item There exists $h_U \in \mcl P(\R^m, \R)$ such that $U=\{u \in \R^m: h_U(u) \ge 0\}$.
	\end{enumerate}
\end{defn}
\vspace{-0.1cm}
Note polynomials are locally Lipschitz continuous, that is $\mcl P(\R^n \times \R, \R) \subset LocLip(\R^n \times \R, \R)$. Therefore $\mcl M_{Poly} \subset \mcl M_{Lip}$.

%To solve the Optimization Problem \eqref{opt: convex L1 norm} we must find a function, $V_l$, that satisfies several inequality constraints. Determining whether a polynomial satisfies an inequality constraint over a compact set, $(x,u,t) \in \Omega \times U \times [0,T]$, has the same difficulties as proving a polynomial is globally positive ($f(x)>0$ $\forall x \in \R^n$); \cite{Blum_1998} has shown this problem to be NP-hard. However it can be shown that testing whether a polynomial is Sum-of-Squares (SOS), and hence positive, is equivalent to solving a semidefinite program (SDP). Although not all positive polynomials are SOS, this gap can be made arbitrarily small, see \cite{Hilbert_1888}.

For given $\{c,g,f,\Omega,U,T\} \in \mcl M_{Poly}$, $d \in \N$, $\Lambda \subset \R^n$ and  $w \in L^1(\Lambda \times [0,T],\R^+)$, we thus propose an SOS tightening of Optimization Problem~\eqref{opt: convex L1 norm} as follows:
%To avoid cumbersome notation we will not state the SDP resulting from the SOS tightening explicitly. The constraints we give can be enforced using software such as SOSTOOLS, found in \cite{Prajna_2002}, that will reformulate the problem as an SDP. Using efficient primal-dual interior point methods for SDP's we are able to solve such proposed problems, see \cite{Nesterov_1994}. We now propose the following SOS programming problem.
\begin{align} \label{opt: SOS for sub soln of finite time}
& P_d \in \arg \max_{P \in \mcl P_d(\R^n \times \R, \R)} c^T \alpha\\ \nonumber
& \text{subject to: } k_{0},k_{1} \in \sum_{SOS}^d, \quad s_{i} \in \sum_{SOS}^d \text{ for } i=0,1,2,3, \\ \nonumber
& P (x,t) = c^T Z_d(x,t), \quad \\ \nonumber
&k_{0}(x) = g(x)-{P(x,T)} - s_{0}(x)h_\Omega (x),\\ \nonumber
& k_{1}(x,u,t) =  \nabla_t P(x,t) +c(x,u,t) + \nabla_x P(x,t)^T f(x,u)   \\ \nonumber
& \quad -s_{1}(x,u,t)h_\Omega(x) -s_{2}(x,u,t)h_U(u) - s_{3}(x,u,t) \cdot(Tt-t^2),
\end{align}
where $\alpha_i=\int_{\Lambda \times [0,T]} w(x,t) Z_{d,i}(x,t) dx dt$, and recalling $Z_d: \R^n \times \R \to \R^{ \mcl N_d}$ is the vector of monomials of degree $d \in \N$. {Note that solutions to Opt.~\eqref{opt: SOS for sub soln of finite time} may not be feasible to Opt.~\eqref{opt: convex L1 norm} due to the strict inequalities of the latter problem.} %and $ \mcl N_d= { d+ n \choose d}$.

%Note that the $\alpha_i$ can be calculated a priori using the SOSTOOLS \cite{Prajna_2002} ``int" function when $\Lambda$ is a hyper-rectangle and $w(x,t)$ is polynomial.

%\textbf{Implementation details of objective function:}
%If $V_l  \in \mcl{P}_d(\R^n \times \R,\R)$ then there exists a coefficient vector $c \in \R^N$ such that $V_l(x,t) = c^T z_d(x,t)$, where $z_d: \R^n \times \R \to \R^N$ is the monomial vector of degree $d \in \N$. Therefore the objective function of \eqref{opt: SOS for sub soln of finite time} can be expressed as
%\begin{align*}
%& \max_{V_l \in \mcl{P}_d(\R^n \times \R,\R)}  \left\{ \int_{\Lambda \times [0,T]}  V_l(x,t) dx dt  \right\} \\
%& = \max_{c \in \R^N} \left\{ c^T \int_{\Lambda \times [0,T]} z_d(x,t) dx dt  \right\}= \max_{c \in \R^N} \left\{ \sum_{i=1}^N c_i \alpha_i \right\},
%\end{align*}
%where $\alpha_i \in \R$ is the i'th component of $\int_{\Lambda \times [0,T]} z_d(x,t) dx dt$. Each $\alpha_i$ can be calculated before solving the optimization problem. Matlab's ``polyint" function can be used to efficiently calculate each $\alpha_i$ when $\Lambda$ is a rectangular region of form $[a_1,b_1] \times.... \times [a_n,b_n]$, where $a_i, b_i \in \R$ for all $i \in \{1,...,n\}$.
\vspace{-0.5cm}
\subsection{{We Can Numerically Construct A Sequence Of Polynomials That Converge The VF}}
{For a given family of OCPs, we now show that the sequence of solutions to the SOS Opts.~\eqref{opt: SOS for sub soln of finite time} converges locally to the VF of the associated OCPs with respect to the $L^1$ norm.}
\begin{prop} \label{prop: SOS converges}
For given $\{c,g,f,\Omega,U,T\} \in \mcl M_{Poly}$ and positive integrable function $w \in L^1(\Omega \times [0,T],\R^+)$ suppose $\Lambda \subseteq \Omega$ satisfies Eq.~\eqref{reachable set condition} then
\begin{align} \label{convergence in L1 for SOS}
\lim_{d \to \infty} \int_{\Lambda \times [0,T]} w(x,t) |V(x,t) -P_d(x,t)| dx dt  = 0,
\end{align}
where $V$ is any function satisfying Eq.~\eqref{opt: optimal control general} and $P_d \in \mcl P_d(\R^n \times \R, \R)$ is any solution to Problem~\eqref{opt: SOS for sub soln of finite time} for $d \in \N$.
\end{prop}

\begin{proof}
	To show Eq.~\eqref{convergence in L1 for SOS} we show that for any $\eps>0$ there exists $N \in \N$ such that for all $d \ge N$
\begin{align*}
\int_{\Lambda \times [0,T]} w(x,t)|V(x,t) - P_d(x,t) | dx dt  < \eps.
\end{align*}
As it is assumed $\Lambda$ satisfies Eq.~\eqref{reachable set condition} we are able to use Prop.~\ref{prop: convergence of inf L1} that shows for any $\eps>0$ there exists $N_1 \in \N$ such that  for all $d \ge N_1$
\begin{align} \label{need 1}
\int_{\Lambda \times [0,T]} w(x,t)|V(x,t) -J_d(x,t)| dx dt  < \eps,
\end{align}
where $J_d$ is a solution to Optimization Problem \eqref{opt: convex L1 norm} for $d \in \N$.

In particular let us fix some $d_1 \ge N_1$. Since $J_{d_1}$ solves Problem~\eqref{opt: convex L1 norm} it must satisfy the constraints of Problem~\eqref{opt: convex L1 norm}. Thus we have

\vspace{-0.5cm}
\begin{align*}
	&k_0(x):=g(x) - J_{d_1} (x,T)>0 \text{ for all } x \in \Omega,  \\ \nonumber
& k_1(x,u,t):=\nabla_t J_{d_1}(x,t) + c(x,u,t) + \nabla_x J_{d_1}(x,t)^T f(x,u)>0 \\
& \hspace{3cm} \text{ for all } (x,u,t)\in \Omega \times U \times [0,T].
\end{align*}
Since $k_0$ and $k_1$ are strictly positive functions over the compact semialgebriac set  $\Omega \times U \times [0,T] = \{(x,u,t) \in \R^{n+m+1}: {h_\Omega(x) \ge 0}, h_U(u) \ge 0, t(T-t) \ge 0 \}$, Putinar's {Positivstellensatz} (stated in Theorem~\ref{thm: Psatz}, Appendix~\ref{sec: appendix 3}) shows that there exist $s_0,s_1,s_2,s_3,s_4,s_5 \in \sum_{SOS}$ such that

\vspace{-0.5cm}
\begin{align} \label{apply psatz}
k_0- h_\Omega s_0 =s_1,\\ \nonumber
k_1- h_\Omega s_2 - h_U s_3 - h_T s_4 =s_5,
\end{align}
where $h_T(t):=(t)(T-t)$. %Let us denote $N_2:= \max_{i\in\{0,1,2,3,4,5\}}{deg(s_i)}$.

 %Letting $P_d$ be the solution to Problem~\eqref{opt: SOS for sub soln of finite time}.

  {Let $N_2:= \max_{i\in\{0,1,2,3,4,5\}}{deg(s_i)}$}. By Eq.~\eqref{apply psatz} it follows that if $J_{d_1}$ is feasible to Problem~\eqref{opt: SOS for sub soln of finite time} for $d \ge \max\{d_1,N_2\}$. Therefore, for $d \ge \max\{d_1,N_2\}$, the objective function evaluated at the solution to Problem~\eqref{opt: SOS for sub soln of finite time} must be greater than or equal to objective function evaluated at $J_{d_1}$. That is by writing the solution to Problem~\eqref{opt: SOS for sub soln of finite time} as $P_d(x,t)= c_d^TZ_d(x,t)$ and writing $J_{d_1}$ as $J_{d_1}=b_{d_1}^TZ_{d_1}(x,t)$ we get that for $d \ge \max\{d_1,N_2\}$

  \vspace{-0.5cm}
\begin{align} \label{sum ineq}
 c_{d}^T \alpha \ge  b_{{d_1}}^T \alpha.
\end{align}
Now  using Eqs.~\eqref{sum ineq} and \eqref{need 1} it follows for all $d \ge \max\{d_1,N_2\}$

\vspace{-0.4cm}
\begin{align*}
&  \int_{\Lambda \times [0,T]}   w(x,t) |V(x,t) - P_d(x,t)| dx dt  \\
&  = \int_{\Lambda \times [0,T]} w(x,t)V(x,t) dx dt -  \int_{\Lambda \times [0,T]} w(x,t) P_d(x,t) dx dt\\
& =   \int_{\Lambda \times [0,T]} w(x,t) V(x,t) dx dt - c_{d}^T \alpha\\
& \le \int_{\Lambda \times [0,T]} w(x,t) V(x,t) dx dt - b_{d_1}^T \alpha\\
& =  \int_{\Lambda \times [0,T]} w(x,t) |V(x,t) -J_{d_1}(x,t)| dx dt  < \eps,
\end{align*}
where the above inequality follows using Prop.~\ref{prop: diss ineq implies lower soln}, which shows $P_d(x,t) \le V(x,t)$ and $J_d(x,t) \le V(x,t)$ for all $(x,t) \in \Omega \times [0,T]$, as $P_d$ and $J_d$ satisfy Inequalities~\eqref{ineq: diss ineq for sub sol of HJB} and \eqref{ineq: BC} (since they both satisfy the constraints of Optimization Problem~\eqref{opt: convex L1 norm}).
\end{proof}
Note, the proof of Prop.~\ref{prop: SOS converges}, that shows we can approximate VFs in the $L^1$ norm over compact sets, uses the fact that the VF is Lipschitz continuous when Eq.~\eqref{reachable set condition} is satisfied. For cases where Eq.~\eqref{reachable set condition} is not satisfied it may still be possible to show convergence if the VF of the state constrained OCP $\{c,g,f,\Omega,U,T\} \in \mcl M_{Poly}$ is Lipschitz continuos. 

%\vspace{-0.5cm}
\subsection{We Can Numerically Construct A Sequence Of Sublevel Sets That Converge To The VF's Sublevel Set}
For a given family of OCPs, Prop.~\ref{prop: SOS converges} shows the SOS optimization problem, given in Eq.~\eqref{opt: SOS for sub soln of finite time}, yields a sequence of polynomials, $\{P_d\}_{d \in \N}$, a sequence that converges to the VF (denoted by $V$), where convergence is with respect to the $L^1$ norm, and where the VF is  associated with the given family of OCPs. We next extend this convergence result by showing that, for any $\gamma \in \R$, the sequence $\{P_d\}_{d \in \N}$ yields a sequence of $\gamma$-sublevel sets, where the sequence of $\gamma$-sublevel sets converges to the $\gamma$-sublevel set of the value function, $V$, where convergence is with respect to the volume metric.

For sets $A,B \subset \R^n$, we denote the volume metric as $D_V(A,B)$, where
 \vspace{-0.65cm}
\begin{equation} \label{eqn: volume metric}
\qquad \qquad  D_V(A,B):=\mu( (A/B) \cup (B/A) ),
\end{equation}
\vspace{-0.6cm}

\noindent where we recall $\mu(A):= \int_{\R^n} \mathds{1}_{A}(x) dx$ is the Lebesgue measure. Note that Lemma \ref{lem: Dv is metric} (Appendix~\ref{sec: appendix 2}) shows that $D_V$ is a metric.
\vspace{-0.5cm}
\begin{prop} \label{prop: SOS converges Dv}
	Consider $\{c,g,f,\Omega,U,T\} \in \mcl M_{Poly}$ and $w(x,t)= \delta(t-s)$ where $s \in [0,T]$ and $\delta$ is the Dirac delta function. Suppose $\Lambda \subseteq \Omega$ satisfies Eq.~\eqref{reachable set condition}. Then we have the following for all $\gamma \in \R$:
	
	\vspace{-0.5cm}
	\begin{align} \label{convergence in DV for SOS}
	\lim_{d \to \infty} D_V\bigg(\{ x \in \Lambda : V(x,s) \le \gamma\}, \{x \in \Lambda : P_d(x,s) \le \gamma\} \bigg)  = 0,
	\end{align}
	
	\vspace{-0.3cm}
	\noindent where $V$ is any function satisfying Eq.~\eqref{opt: optimal control general}, and $P_d$ is any solution to Problem~\eqref{opt: SOS for sub soln of finite time} for $d \in \N$.
\end{prop}
\begin{proof}
 To show Eq.~\eqref{convergence in DV for SOS} we use Prop.~\ref{prop: close in L1 implies close in V norm}, found in Appendix~\ref{sec: appendix 2}. Let us consider the family of functions, $ \{P_d \in \mcl P_d(\R^n \times \R, \R)  : d \in \N \}$, where $P_d$ solves the optimization problem given in Eq.~\eqref{opt: SOS for sub soln of finite time} for $d \in \N$ and $w(x,t)= \delta(t-s)$.

From the definition of Problem~\eqref{opt: SOS for sub soln of finite time}, we have that $P_d$ satisfies the inequalities in~\eqref{ineq: diss ineq for sub sol of HJB} and \eqref{ineq: BC}. Therefore, by Prop.~\ref{prop: diss ineq implies lower soln}, we have that  $P_d(x,t) \le V(x,t)$ for all $ (x,t) \in \Omega \times [0,T]$, where $V$ is any function satisfying Eq.~\eqref{opt: optimal control general}. Since $\Lambda \subseteq \Omega$ satisfies Eq.~\eqref{reachable set condition}, and although the Dirac delta function is not a member of $L^1(\Omega \times [0,T], \R)$, a similar argument to Prop.~\ref{prop: SOS converges} implies that

\vspace{-0.4cm}
\begin{align*}
&\lim_{d \to \infty} \int_{\Lambda} | V(x,s) - P_d(x,s)| dx dt \\
& \qquad =  \lim_{d\to \infty} \int_{\Lambda \times [0,T]} \delta(t-s) | V(x,t) - P_d(x,t)| dx dt  =0 .
\end{align*}
We now apply Prop.~\ref{prop: close in L1 implies close in V norm} (Section~\ref{sec: appendix 2}) to deduce Eq.~\eqref{convergence in DV for SOS}.
%We now apply Prop.~\ref{prop: close in L1 implies close in V norm} and deduce Eq.~\eqref{convergence in DV for SOS}.
\end{proof}

\vspace{-0.5cm}
%%%%%%%%%%%%%%%%%%%%%%%%%%%%%%%%%%%%%%%%%%%%%%%%%%%%%%%%%
\section{A Performance Bound On Controllers Constructed Using Approximation to the VF} \label{sec:optimal controller approx}
%In this section we now consider the second question posed in the introduction, how well does a controller perform that is constructed from an approximated VF? We show that if a function approximates the VF well in the $W^{1,\infty}$ norm, then controllers derived from that approximated VF will perform well when applied to the original OCP. Before presenting the result we introduce some notation.
%
%\vspace{-0.05cm}
{ \normalsize Given an OCP, if an associated differentiable VF is known then a solution to the OCP can be constructed using Thm.~\ref{thm: value functions construct optimal controllers}. However, in general, it is challenging to find a VF analytically. Rather than computing a true VF, we consider a candidate VF which is ``close" to a true VF under some norm. This motivates us to ask the question: how well will a controller constructed from a candidate VF perform? To answer this question we next define the loss/performance of an input. For state unconstrained OCPs, $\{c,g,f,\R^n,U,T\} \in \mcl M_{Lip}$, we denote the loss/performance function as, }
%Then Prop.~\ref{prop: SOS converges} shows that such a candidate VF can be made arbitrarily ``close" to the true VF with respect to the $L^1$ norm as $d \to \infty$.

\vspace{-0.4cm}

{\small
	\begin{align}
	& L(x_0,\mbf u)  :=\int_{0}^{T} c(\phi_f(x_0,s  ,\mbf u),  \mbf u (s),s) ds  + g(\phi_f(x_0,T,\mbf u))  \\ \nonumber
	& -\inf_{\mbf u \in \mcl U_{\R^n,U,f,T}(x_0,0)} \bigg\{  \int_{0}^{T} c(\phi_f(x_0,s,\mbf u),\mbf u(s),s) ds  + g(\phi_f(x_0,T,\mbf u)) \bigg\}.
	\end{align} }
Clearly, $L(x_0,\mbf u ) \ge 0$ for all $(x_0, \mbf u) \in \Omega \times \mcl U_{\R^n,U,f,T}(x_0,0)$.
%\vspace{-0.1cm}
\begin{thm} \label{thm: performance bounds}
	Consider some state unconstrained OCP $\{c,g,f,\R^n ,U,T\} \in \mcl M_{Lip}$ initialized at $x_0 \in \R^n$ with associated VF $V^*$. Suppose there exists an open set $\Omega \in \R^n$ such that $FR_f(\{x_0\},\R^n,U,[0,T]) \subseteq \Omega$. Then for any function $J \in C^2( \R^n \times \R, \R)$ we have that
%\vspace{-0.3cm}
\begin{align} \label{ineq:loss functin}
	& L(x_0, \mbf u _J) \le  C ||J- V^*||_{W^{1,\infty}(\Omega \times [0,T], \R)}, \\ \label{eqn: C}
	& 	\vspace{-0.2cm} \text{where } C:= 2\max \left\{1,T, T\max_{1 \le i \le n } \sup_{(x,t) \in \Omega \times U} |f_i(x,u)|   \right\},
	\\ \label{J controller}
	&\qquad \quad \mbf u_J(t)= k_J(\phi_f(x_0,t,\mbf u_J), t),\\
	\label{J controller K}
	&\text{and }  \quad  k_J(x,t) \in \arg \inf_{u \in U}\{ c(x,u,t) + \nabla_x J(x,t)^T f(x,u)\}.
\end{align} 
\end{thm}
\begin{proof}
	%	Let $V^*(x,t)$ be any function satisfying Equation \eqref{opt: optimal control general}. By Theorem \ref{thm: HJB value function} and Lemma \ref{lem: value function of unconstrained and constrained problem are equal} we have $V^*(x,t)=V(x,t)$ for all $(x,t) \in \Lambda \times [0,T]$ where $V(x,t)$ is the unique viscosity solution to the HJB PDE \eqref{eqn: general HJB PDE}.

%	To aid the proof we introduce some notation. For $F \in Lip( \R^n \times \R , \R)$ we denote

Now, for any $J \in C^2(\R^n \times \R, \R)\subset LocLip( \R^n \times \R , \R)$, we wish to show that Eq.~\eqref{ineq:loss functin} holds. To do this, we will show that $J$ is the true VF for some modified OCP. Before constructing this modified OCP, for any $F \in LocLip( \R^n \times \R , \R)$, let us define	\vspace{-0.15cm}
	\begin{align*}
	H_F(x,t,u) & := \nabla_t F(x,t) + c(x,u,t) + \nabla_x F(x,t)^T f(x,u),\\
	\tilde{H}_F(x,t) & :=\inf_{u \in U} H_F(x,t,u),
	\end{align*}
	where $\nabla_t F$ and $\nabla_x F$ are weak derivatives, known to exist by Rademacher's Theorem (Thm.~\ref{thm: Rademacher theorem}).
	
	Then, by construction, $J$ satisfies the following PDE
	\begin{align} \nonumber
	& \nabla_t J (x,t)+ \inf_{u \in U} \left\{  c(x,u,t) -\tilde{H}_J(x,t) + \nabla_x J(x,t)^T f(x,u) \right\} =0\\  & \hspace{3.5cm} \text{ for all } (x,t) \in \R^n \times  [0,T]. \label{1}
	\end{align}
Eq.~\eqref{1} implies that $J$ satisfies the HJB PDE associated with $\{\tilde{c},\tilde{g},f,\R^n,U,T\}$, where $\tilde{c}(x,u,t):=c(x,u,t)-\tilde{H}_J(x,t)$ and $\tilde{g}(x):=J(x,T)$. Note that since $c \in LocLip(\Omega \times U \times [0,T], \R)$, $f \in LocLip(\Omega \times U, \R)$, and $\frac{\partial}{\partial x_i} J \in LocLip(\Omega \times [0,T], \R)$ for all $i \in \{1,...n+1\}$ (since  $J \in C^2( \R^n \times \R, \R )$) it follows that $H_J \in LocLip(\Omega \times U \times [0,T],\R)$. By Lemma \ref{lem: family of lip is lip} we then deduce $\tilde{H}_J \in LocLip(\Omega \times [0,T],\R)$ and thus $\{\tilde{c},\tilde{g},f,\R^n,U,T\} \in \mcl M_{Lip}$.
	
 Since ${H}_J$ is independent of $u \in U$, we have that
	
	\vspace{-0.4cm}
	\begin{align*}
	\arg \inf_{u \in U}\{ \tilde{c}(x,u,t) &+ \nabla_x J(x,t)^T f(x,u)\} \\ & = \arg \inf_{u \in U}\{ c(x,u,t) + \nabla_x J(x,t)^T f(x,u)\},
	\end{align*}
	and therefore we are able to deduce by Theorem \ref{thm: value functions construct optimal controllers} that $\mbf u_J$ (given in Eq.~\eqref{J controller}) solves the modified OCP associated with $\{\tilde{c},\tilde{g},f,\R^n,U,T\}$ with initial condition $x_0 \in \R^n$. Thus for all $ \mbf u \in \mcl U_{\R^n,U,f,T}(x_0,0)$ we have that
	
	\vspace{-0.4cm}
	
	\begin{align} \label{u_J ineq}
	& \int_{0}^{T} \tilde{c}(  \phi_f(x_0,s  ,\mbf u_J),  \mbf u_J(s),s) ds  + \tilde{g}(\phi_f(x_0,T,  \mbf u_J)) \\ \nonumber
	& \le \int_{0}^{T} \tilde{c}(\phi_f(x_0,s  ,\mbf u),  \mbf u(s),s) ds  + \tilde{g}(\phi_f(x_0,T,\mbf u)).
	\end{align}
	By substituting $\tilde{c}(x,u,t)=c(x,u,t)-\tilde{H}_J(x,t)$ and $\tilde{g}(x)=J(x,T)$  into Inequality~\eqref{u_J ineq} and noting that $V^*(x,T)= g(x)$, we have the following for all $ \mbf u \in \mcl U_{\R^n,U,f,T}(x_0,0)$:
	
	\vspace{-0.4cm}
	\begin{align} \label{difference in cost}
	& \int_{0}^{T} c(  \phi_f(x_0,s  ,\mbf u_J),  \mbf u_J(s),s) ds  + {g}(\phi_f(x_0,T,  \mbf u_J)) \\ \nonumber
	& \qquad - \int_{0}^{T} {c}(\phi_f(x_0,s  ,\mbf u),  \mbf u (s),s) ds  - {g}(\phi_f(x_0,T,\mbf u)) \\ \nonumber
	& \le \int_0^T \tilde{H}_J(\phi_f(x_0,s  ,\mbf u_J),s)  - \tilde{H}_J(\phi_f(x_0,s  ,\mbf u),s)  ds \\\nonumber
	& \qquad + V^*(\phi_f(x_0,T,  \mbf u_J),T) - J(\phi_f(x_0,T,  \mbf u_J),T)\\ \nonumber
	& \qquad + J(\phi_f(x_0,T,  \mbf u ),T) - V^*(\phi_f(x_0,T,  \mbf u ),T) \\ \nonumber
	& < T\esssup_{s \in[0,T]}\{ \tilde{H}_J(\phi_f(x_0,s  ,\mbf u_J),s)    - \tilde{H}_J(\phi_f(x_0,s  ,\mbf u),s)  \}\\ \nonumber
	& \qquad +  2\sup_{y \in \Omega}\{ |V^*(y,T) - J(y,T)| \}\\ \nonumber
	& \le T \bigg( \esssup_{(y,s) \in \Omega \times [0,T]}\{ \tilde{H}_J(y,s) \} -  \essinf_{(y,s) \in \Omega \times [0,T]}\{ \tilde{H}_J(y,s) \} \bigg) \\ \nonumber
	& \qquad +  2\esssup_{(y,s) \in \Omega \times [0,T]}\{ |V^*(y,s) - J(y,s)| \}.
	\end{align}
	The second and third inequalities of Eq.~\eqref{difference in cost} follow because $\phi_f(x_0,t,\mbf u) \in \Omega$ for all $(t, \mbf u) \in [0,T] \times \in \mcl U_{\R^n,U,f,T}(x_0,0)$ (since it is assumed $FR_f(\{x_0\},\R^n,U,[0,T]) \subseteq \Omega$), and because $\sup_{y \in \Omega}\{ |V^*(y,T) - J(y,T)| \}= \esssup_{y \in \Omega}\{ |V^*(y,T) - J(y,T)| \}$ holds by Lemma~\ref{lem: esssup=sup} (since $V^*$ and $J$ are both continuous, and $\Omega$ is open).
	
	We now split the remainder of the proof into three parts. In Part~1, we derive an upper bound for $\esssup_{(y,s) \in \Omega \times [0,T]}\{ \tilde{H}_J(y,s) \}$. In Part~2, we find a lower bound for $\essinf_{(y,s) \in \Omega \times [0,T]}\{ \tilde{H}_J(y,s) \}$. In Part~3 we use these two bounds, combined with Inequality~\eqref{difference in cost} to verify Eq.~\eqref{ineq:loss functin} and complete the proof.
	
Before proceeding with Parts 1 to 3 we introduce some notation for the set of points where the VF is differentiable,
	\begin{align*}
	S_{V^*}:=\{(x,t) \in \Omega \times [0,T]: V^* \text{ is differentiable at } (x,t)\}.
	\end{align*}
Lemma~\ref{lem: value function is lip} shows that $V^* \in Lip(\Omega \times [0,T], \R) \subset LocLip(\R^n \times \R, \R)$ and Rademacher's Theorem (Thm. \ref{thm: Rademacher theorem}) states that Lipschitz functions are differentiable almost everywhere. It follows, therefore, that $\mu((\Omega \times [0,T]) / S_{V^*})=0$, where $\mu$ is the Lebesgue measure.
	
\underline{\textbf{Part 1 of Proof:}}
For each $(y,s) \in S_{V^*}$ let us consider some family of points $k^*_{y,s} \in U$ such that
\begin{align*}
	k^*_{y,s} \in \arg \inf_{u \in U} \bigg\{ c(y,u,s) + \nabla_x V^*(y,s)^Tf(y,u)  \bigg\}.
\end{align*}
Note, $k^*_{y,s}$ exists for each fixed $(y,s) \in S_{V^*}$ by the extreme value theorem since $U \subset \R^m$ is compact, $c,f$ are continuous, and $\nabla_x V^*$ is independent of $u \in U$ and bounded by Rademacher's Theorem (Thm. \ref{thm: Rademacher theorem}).
	
Now for all $ (y,s) \in S_{V^*}$ it follows that
\begin{align} \label{11}
	\tilde{H}_J(y,s) & = \inf_{u \in U} H_J(y,s,u) \le H_J(y,s,k^*_{y,s} ).
\end{align}

	Moreover, since $V^*$ is the viscosity solution to the HJB PDE by Theorem~\ref{thm: HJB value function}, we have that
	\begin{align} \label{12}
	H_{V^*}(y,s,k^*_{y,s})=0\quad  \text{for all } (y,s) \in S_{V^*}.
	\end{align}

	Combing Eqs.~\eqref{11} and \eqref{12} it follows that
	\begin{align} \nonumber
	& \tilde{H}_J(y,s)  \le H_J(y,s,k^*_{y,s}) - H_{V^*}(y,s,k^*_{y,s})\\ \nonumber
	& = \nabla_t J(y,s) - \nabla_t V^*(y,s)  + ( \nabla_x J(y,s) - \nabla_x V^*(y,s))^T f(y,k^*_{y,s})\\ \nonumber
	& \le |\nabla_t J(y,s) - \nabla_t V^*(y,s)| \\ \label{42}
	& \quad + \max_{1 \le i \le n }|f_i(y,k^*_{y,s})| \sum_{i=1}^n \bigg| \frac{\partial }{\partial x_i} (J(y,s)  -V^*(y,s) ) \bigg|
	\end{align}
for all $ (y,s) \in S_{V^*}$.	As Eq.~\eqref{42} is satisfied for all $(y,s) \in S_{V^*}$ and $\mu((\Omega \times [0,T]) / S_{V^*})=0$ it follows Eq.~\eqref{42} holds almost everywhere. Therefore
	\begin{align} \label{31}
	& \esssup_{ (y,s) \in \Omega \times [0,T] } \tilde{H}_J(y,s) \\ \nonumber
	& \qquad \le \max \bigg\{ 1, \max_{1 \le i \le n } \sup_{(x,t) \in \Omega \times U} |f_i(x,u)| \bigg\} || V^* - J||_{W^{1, \infty}(\Omega \times [0,T])}.
	\end{align}

	\underline{\textbf{Part 2 of Proof:}} If $k_J$ satisfies Eq.~\eqref{J controller K}, then
	\begin{align} \label{21}
	\tilde{H}_J(y,s)= \inf_{u \in U} H_J(y,s,u) = H_J(y,s,k_J(y,s)) \text{ for all } (y,s) \in S_{V^*}.
	\end{align}
%for all $(y,s) \in \R^n \times \R$.
	Moreover, since $V^*$ is a viscosity solution to the HJB PDE~\eqref{eqn: general HJB PDE}, we have by Theorem~\ref{thm: HJB value function} that
	\begin{align} \label{22}
	H_{V^*}(y,s,k_J(y,s))\ge \inf_{u \in U} H_{V^*}(y,s,u) =  0 \text{ for all } (y,s) \in S_{V^*}.
	\end{align}
	Combining Eqn's~\eqref{21} and \eqref{22} it follows that
	\begin{align} \nonumber
	& \tilde{H}_J(y,s)  \ge    H_J(y,s,k_J(y,s)) - H_{V^*}(y,s,k_J(y,s)) \\ \nonumber
	& = \nabla_t J(y,s) - \nabla_t V^*(y,s)  + ( \nabla_x J(y,s) - \nabla_x V^*(y,s))^T f(y,k_J(y,s))\\ \nonumber
	& \ge -|\nabla_t J(y,s) - \nabla_t V^*(y,s)| \\ \label{41}
	& \qquad - \max_{1 \le i \le n }|f_i(y,k_J(y,s))| \sum_{i=1}^n \bigg| \frac{\partial }{\partial x_i} (J(y,s)  -V^*(y,s) ) \bigg|
	\end{align}
for all $ (y,s) \in S_{V^*}$. Therefore, since $\mu((\Omega \times [0,T]) / S_{V^*})=0$, we have that Inequality~\eqref{41} holds almost everywhere. Thus
	\begin{align} \label{32}
	& \essinf_{ (y,s) \in \Omega \times [0,T] } \tilde{H}_J(y,s) \\ \nonumber
	&  \ge -\max \bigg\{ 1, \max_{1 \le i \le n } \sup_{(x,t) \in \Omega \times U} |f_i(x,u)| \bigg\} || V^* - J||_{W^{1, \infty}(\Omega \times [0,T])}.
	\end{align}
	
	\underline{\textbf{Part 3 of Proof:}}

Combining Inequalities~\eqref{difference in cost}, \eqref{31} and \eqref{32} it follows that
	\begin{align}  \label{8}
	& \int_{0}^{T} c(  \phi_f(x_0,s  ,\mbf u_J),  \mbf u_J(s),s) ds  + {g}(\phi_f(x_0,T,  \mbf u_J)) \\ \nonumber
	& \qquad - \int_{0}^{T} {c}(\phi_f(x_0,s  ,\mbf u),  \mbf u(s),s) ds  - {g}(\phi_f(x_0,T,\mbf u)) \\ \nonumber
	& < C ||J- V^*||_{W^{1,\infty}} \text{ for all } \mbf u \in \mcl U_{\Omega,U,f,T}(x_0,0),
	\end{align}
	where $C:= 2\max \bigg\{1,T, T\max_{1 \le i \le n } \sup_{(x,t) \in \Omega \times U} |f_i(x,u)|   \bigg\}$. Now as Inequality~\eqref{8} holds for all $ \mbf u \in \mcl U_{\R^n,U,f,T}(x_0,0)$ we can take the infimum and deduce Inequality~\eqref{ineq:loss functin}.
\end{proof}
Note, the condition in Thm.~\ref{thm: performance bounds} that $FR_f(\{x_0\},\R^n,U,[0,T]) \subseteq \Omega$ is trivially satisfied by $\Omega=\R^n$. However, since our SOS method only approximates VFs over compact sets (Prop.~\ref{prop: SOS converges}) we would need this condition to be satisfied by a compact $\Omega$ if we were to use Thm.~\ref{thm: performance bounds} to derive performance bounds for controllers synthesized from our SOS derived approximate VFs.
%We will next show there exists a polynomial approximation of the value function that can construct a controller arbitrarily close to optimality.
%\begin{cor}
%For given $\eps>0$ and $\{c,g,f,\Omega,U,T\} \in \mcl M$ suppose $\Lambda \subseteq \Omega$ satisfies \eqref{reachable set condition}. Then there exists $J \in \mcl P( \R^n \times \R, \R)$ that satisfies the Dissipation Inequities \eqref{ineq: diss ineq for sub sol of HJB} and \eqref{ineq: BC} and if $x_0 \in \Lambda$, then
%\begin{align*}
%& 0 \le L(x_0,\mbf u_J)  < \eps,
%	\end{align*}
%	where $\mbf u_J$ is as in \eqref{J controller} and \eqref{J controller K}.
%\end{cor}

\vspace{-0.2cm}
\section{Numerical Examples: Using Our SOS Algorithm To Approximate VFs}
%\subsection{Numerical Example: Approximating VF's}
In this section we use the SOS programming problem as defined in Eq.~\eqref{opt: SOS for sub soln of finite time} to numerically approximate the VFs associated with several different OCPs. We first approximate a known VF. Then, in Subsection~\ref{subsec: approx optimal control with poly VF's}, we approximate an unknown VF and use this approximation to construct a controller and analyze its performance. Then, in Subsection~\ref{subsec:reachable sets}, we approximate another unknown VF for reachable set estimation.
\vspace{-0.2cm}
\begin{ex} \label{ex:lib}%{\cite{liberzon2011calculus}}
	Let us consider the tuple $\{c,g,f,\Omega,U,T\} \in \mcl M_{Poly}$, where $c(x,u,t) \equiv 0$, $g(x)=x$, $f(x,u)=xu$, $\Omega=(-R,R)=\{x \in \R: x^2<R^2\}$, $U=(-1,1)=\{u \in \R: u^2<1\}$, and $T=1$. It was shown in \cite{liberzon2011calculus} that the VF associated with $\{c,g,f,\R^n,U,T\}$ can be analytically found as
	\vspace{-0.1cm}
	\begin{equation} \label{analytucal soln}
	V^*(x,t)= \begin{cases} \exp(t-1)x \text{ if } x>0,\\
	\exp(1-t)x \text{ if } x<0,\\
	0 \text{ if } x=0. \end{cases}
	\end{equation}
	We note that $V^*$ is not differentiable at $x=0$. However, $V^*$ satisfies the HJB PDE away from $x=0$. This problem shows that the VF can be non-smooth even for simple OCP's with polynomial vector field and cost functions.
	
	{In Fig.~\ref{fig:lib} we have plotted the point wise error, $e(x,t):=V^*(x,t) - P_d(x,t)$, where $P_d$ is the solution to the SOS Optimization Problem~\eqref{opt: SOS for sub soln of finite time} for $d=16$, $T=1$, $\Lambda=[-0.5,0.5]$, $w(x,t) \equiv 1$, $h_\Omega(x)=2.4^2 -x^2 $ and $h_U(u)=1 - u^2$. The figure shows $e(x,t) \ge 0$ for all $(x,t) \in [-0.5,0.5] \times [0,1]$ verifying that, as expected by Prop.~\ref{prop: diss ineq implies lower soln}, $P_d$ is a sub-VF. Moreover, $0<e(x,t)<0.1125$ for all $(x,t) \in [-0.5,0.5] \times [0,1]$ implying $||V^* - P_d||_{\infty}<0.1125$, showing that we get a tight VF approximation in the $L^\infty$ norm (even though we optimize for the $L^1$ norm). }
	
	In Fig.~\ref{fig: Liberzon d=1:5} we have plotted the function $F(d):=||V^*-P_d||_{L^1(\Lambda \times [0,T], \R)}$ where $V^*$ is given in Eq.~\eqref{analytucal soln} and $P_d$ is the solution to the SOS Optimization Problem~\eqref{opt: SOS for sub soln of finite time} for $d=4$ to $24$, where $\Lambda=[-0.5,0.5]$, $w(x,t) \equiv 1$, $h_\Omega(x)=2.4^2 -x^2 $ and $h_U(u)=1 - u^2$. All solutions, $V_d$, of Problem~\eqref{opt: SOS for sub soln of finite time} were sub-value functions as expected. Moreover, the figure shows by increasing the degree $d \in \N$ the resulting sub-VF, $V_d$, better approximates $V^*$, {however convergence does slow after $d=5$}.%; where the $L^1$ error appears to decay at a faster rate than ${1}/{d^{1.63}}$.
	
		\begin{figure}
		\includegraphics[scale=0.6]{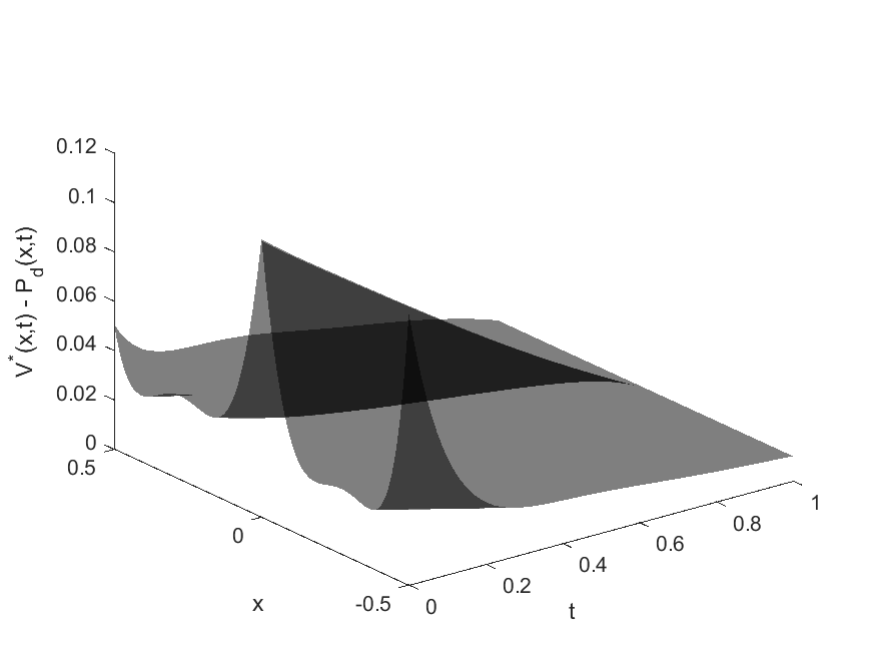}
		\vspace{-20pt}
		\caption{{Plot associated with Example~\ref{ex:lib} showing point wise error, $e(x,t):=V^*(x,t) - P_d(x,t)$ where $V^*$ is given in Eq.~\eqref{analytucal soln} and $P_d$ solves the SOS Problem~\eqref{opt: SOS for sub soln of finite time} for $d=16$.}}
		\label{fig:lib}
		\vspace{-10pt}
	\end{figure}

	\begin{figure}
		\includegraphics[scale=0.6]{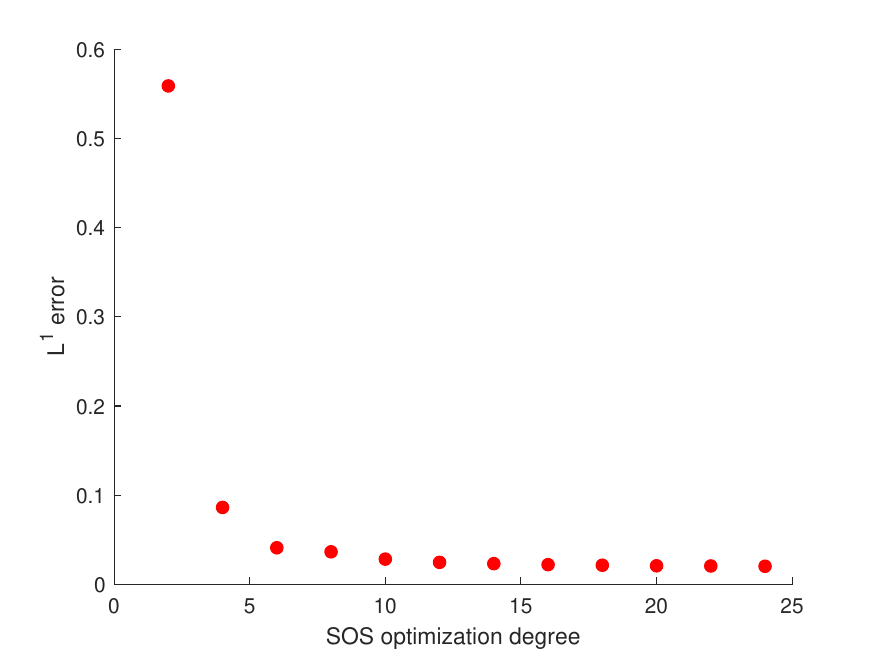}
		\vspace{-20pt}
		\caption{{Scatter plot associated with Example~\ref{ex:lib} showing the $L^1$ norm error: $||V^*-P_d||_{L^1(\Lambda \times [0,T], \R)}$, where $V^*$ is given in Eq.~\eqref{analytucal soln} and $P_d$ solves the SOS Problem~\eqref{opt: SOS for sub soln of finite time} for $d=4$ to $24$. The smallest $L^1$ norm error occur ed at $d=24$ with a value of $0.020316$.}}
		\label{fig: Liberzon d=1:5}
		\vspace{-15pt}
	\end{figure}
\end{ex}
\subsection{Using SOS Programming To Construct Polynomial Sub-Value Functions For Controller Synthesis} \label{subsec: approx optimal control with poly VF's}
Given an OCP, in Theorem~\ref{thm: performance bounds} we showed that the performance of a controller constructed from a candidate VF is bounded by the $W^{1,\infty}$ norm between the true VF of the OCP and the candidate VF. We next demonstrate through numerical examples that the performance of a controller constructed from a typical solutions to the SOS Problem~\eqref{opt: SOS for sub soln of finite time} is significantly higher than that predicted by this bound. % We will construct such a candidate VF by solving the SOS Problem~\eqref{opt: SOS for sub soln of finite time}. %Prop.~\ref{prop: SOS converges} shows a function can be constructed that is arbitrarily close to the VF under the $L^1$ norm by solving the SOS Problem~\eqref{opt: SOS for sub soln of finite time} for sufficiently large $d \in \N$.

%We now construct controllers based on the solution to the SOS programing problem found in . Although, Proposition \ref{prop: SOS converges} only shows solutions to \eqref{opt: SOS for sub soln of finite time} converge to the true VF in the $L^1$-norm, a weaker form of convergence than $W^{1, \infty}$ convergence, we next present numerical examples where our derived controllers outperform alternative controllers found in the literature.

Consider tuple $\{c,g,f,\R^n,U,T\} \in \mcl M_{Poly}$, where the cost function is of the form $c(x,u,t)= c_0(x,t) + \sum_{i=1}^m c_i(x,t)u_i$, the dynamics are of the form $f(x,u)=f_0(x) + \sum_{i=1}^m f_i(x)u_i$, and the input constraints are of the form $U=[a_1,b_1] \times ... \times [a_m,b_m]$. Since any rectangular set can be represented as $U=[-1,1]^m$ using the substitution $\tilde{u}_i = \frac{2u_i - 2b_i}{b_i - a_i}$ for $i \in \{1,...,m\}$, without loss of generality we assume $U=[-1,1]^m$. Now, given an OCP associated with $\{c,g,f,\R^n,U,T\} \in \mcl M_{Poly}$ suppose $V \in C^1(\R^n \times \R, \R)$ solves the HJB PDE~\eqref{eqn: general HJB PDE}, then by Theorem~\ref{thm: value functions construct optimal controllers} a solution to the OCP initialized at $x_0 \in \R^n$ can be found as
{ \begin{align} \label{optimal controller 2}
&\mbf u ^*(t):= k(\phi_f(x_0,t,\mbf u^*), t), \text{ where}\\ \label{optimal controller 3}
&	k(x,t) \in \arg \inf_{u \in [-1,1]^m}\bigg\{ \sum_{i=1}^m c_i(x,t)u_i + \nabla_x V(x,t)^T f_i(x)u_i \bigg\}.
	\end{align} }
Since the objective function in Eq.~\eqref{optimal controller 3} is linear in the decision variables $u \in \R^m$, and since the constraints have the form $u_i \in [-1,-1]$, it follows that Eqs.~\eqref{optimal controller 2} and \eqref{optimal controller 3} can be reformulated as
{ \begin{align} \label{optimal controller new}
&\mbf u ^*(t):= k(\phi_f(x_0,t,\mbf u^*), t), \text{ where}\\ \label{optimal controller new 2}
&k_i(x,t)= - \sign(c_i(x,t)+ \nabla_x V(x,t)^T f_i(x)).
%&	k(x,t) \in \arg \inf_{u \in [-1,1]^m}\bigg\{ \sum_{i=1}^m c_i(x,t)u_i + \nabla_x V(x,t)^T f_i(x)u_i \bigg\}.
	\end{align} }
%
%the solution to  lies at the boundary $u_i\in\{-1,1\}^m$. Thus for each $i \in \{1,..,m\}$ we may obtain an expression for $k_i(x,t)$ as
%\begin{align} \label{optimal controller}
%k_i(x,t)= - \sign(c_i(x,t)+ \nabla_x V(x,t)^T f_i(x)).
%\end{align}
In the following numerical examples we approximately solve {OCPs of this form (with cost functions and dynamics affine in  the input variable)} by constructing a controller from the solution $P$ to the SOS Problem~\eqref{opt: SOS for sub soln of finite time}. We construct such controllers by replacing $V$ with $P$ in Eqs.~\eqref{optimal controller new} and \eqref{optimal controller new 2}. We will consider OCPs with no state constraints and initial conditions inside some set $\Lambda \subseteq \R^n$. We select $\Omega = B(0,R)$ with $R>0$ sufficiently large enough so Eq.~\eqref{reachable set condition} is satisfied. That is, no matter what control we use, the solution map starting from any $x_0 \in \Lambda$ will not be able to leave the state constraint set $\Omega$. In this case the solution to the state constrained problem, $\{c,g,f,\Omega,U,T\}$, is equivalent to the solution of the state unconstrained problem, $\{c,g,f,\R^n,U,T\}$.

{
To evaluate the performance of our constructed controller, $\mbf{{u}}$, we approximate the objective/cost function of the OCP evaluated at $\mbf{{u}}$ (ie the cost associated with $\mbf{{u}}$) using the Riemann sum:
\begin{align} \label{R sum}
\int_0^T c(\phi_{{f}}(x_0,t,& \mbf u),t) dt  + g(\phi_{{f}}(x_0,T,\mbf u )) \\ \nonumber &  \approx  \sum_{i=1}^{N-1}  c(\phi_{{f}}(x_0,t_i,\mbf u ),t_i)\Delta t_i
 +g(\phi_{{f}}(x_0,t_N,\mbf u )) ,
\end{align}
where $0=t_0<...<t_N=T$, $\Delta t_i = t_{i+1}-t_i$ for all $i \in \{1,...,N-1\}$, and $\{\phi_{{f}}(x_0,t_i,\mbf u )\}_{i=0}^N$ can be found using Matlab's \texttt{ode45} function.
}

%{We next approximately solve OCPs with cost functions and dynamics that are affine in the input variable. Examples~\ref{ex: Jacobson problem 4} and~\ref{ex:jacboson} apear in~\cite{jacobson1970computation} and have quadratic cost  }

\begin{ex} \label{ex: Jacobson problem 4}
	Let us consider the following OCP from~\cite{jacobson1970computation}:
	\begin{align} \label{example 2 opt}
	& \min_{\mbf u} \int_0^5 x_1(t) dt \\ \nonumber
	& \text{subject to:} \begin{bmatrix}
	\dot{x_1}(t)\\
	\dot{x_2}(t)
	\end{bmatrix}= \begin{bmatrix}
	x_2(t) \\ \mbf u(t)
	\end{bmatrix}, \text{ } \mbf u(t) \in [-1,1] \text{ for all } t \in [0,5].
	\end{align}
	We associate this problem with the tuple $\{c,g,f,\Omega,U,T\} \in \mcl M_{Poly}$ where $c(x,t)= x_1$, $g(x) \equiv 0$, $f(x,u)= [x_2,u]^T$, $U=[-1,1]$, and $T=5$. By solving the SOS Optimization Problem \eqref{opt: SOS for sub soln of finite time} for $d=3$, $\Lambda=[-0.6,0.6] \times [-1,1]$, $w(x,t) \equiv 1$, $h_\Omega(x)=10^2 -x_1^2 - x_2^2 $ and $h_U(u)=1 - u^2$, it is possible to obtain a polynomial sub-value function $P$. By replacing $V$ with $P$ in Eqs.~\eqref{optimal controller new} and \eqref{optimal controller new 2} it is then possible to construct a controller, $k_P$, that yields a candidate solution to the OCP as $\mbf{\tilde{u}}(t)=k_P(x(t),t)$.
	
	%Matlab's function, \texttt{ode45}, can be used to approximately solve the ODE $\dot{x}(t)=f(x(t), \mbf u(t))$ for each given $\mbf u$. We denote the solution map (Defn.~\ref{defn: soln map}) of this ODE by $\phi_{{f}}$. 
	
	{For initial condition $x_0 =[ 0 ,1]^T$ we use Matlab's \texttt{ode45} to find the set $\{ \phi_{\tilde{f}}(x_0,t,\mbf{\tilde{u}} ) \in \R^2: t \in [0,T] \}$ (recalling $\phi_{{f}}$ denotes the solution map (Defn.~\ref{defn: soln map})), which is shown in the phase plot in Figure~\ref{fig: problem 4}. }
	%Close to the point $(0,0)$ the trajectory becomes ``hairy" implying rapid switching between $u=-1$ and $u=1$. This is expected as the controller attempts to stall the solution map around the line $x_1=0$ in order to minimize the cost $c(x,t)=x_1$.
	 For $N=10^8$ Eq.~\eqref{R sum} was used to find the cost associated with a fixed input, $\mbf u(t) \equiv 1$, as $354.17$, whereas the cost of using $\mbf u(t) \equiv -1$ was found to be $41.67$. The cost of using our derived input $\mbf{\tilde{u}}$ was found to be $0.2721$, an improvement when compared to the cost $0.2771$ found in~\cite{jacobson1970computation}. {Note,
	 	 it may be possible that the results of the algorithm proposed in~\cite{jacobson1970computation} may be improved by selecting different tunning parameters of the algorithm. We have assumed that the authors of~\cite{jacobson1970computation} have selected the tunning parameters for which their algorithm performs best.}

\end{ex}

\begin{figure}
	\includegraphics[scale=0.6]{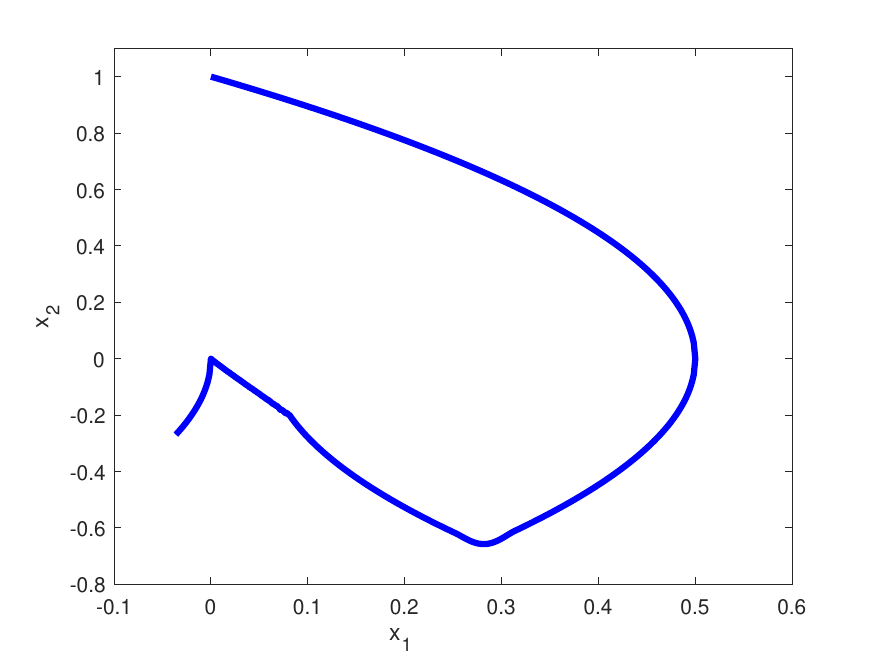}
	\vspace{-25pt}
	\caption{The phase plot of Example~\ref{ex: Jacobson problem 4} found by constructing the controller given in Eq.~\eqref{optimal controller new} using the solution to the SOS Problem~\eqref{opt: SOS for sub soln of finite time}.}
	\label{fig: problem 4}
	\vspace{-15pt}
\end{figure}
\vspace{-0.25cm}
\begin{ex} \label{ex:jacboson}
	Consider an OCP found in~\cite{jacobson1970computation} and \cite{dadebo1998computation} which has the same dynamics as Eq.~\eqref{example 2 opt} but a different cost function. The associated tuple is $\{c,g,f,\Omega,U,T\} \in \mcl M_{Poly}$ where $c(x,t)= x_1^2 +x_2^2$, $g(x) \equiv 0$, $f(x,u)= [x_2,u]^T$, $U=[-1,1]$, and $T=5$. By solving the SOS Optimization Problem \eqref{opt: SOS for sub soln of finite time} for $d=4$, $\Lambda=[-0.5,1.1]\times[-1.1,.5]$, $w(x,t) \equiv 1$, $h_\Omega(x)=10^2 -x_1^2 - x_2^2 $ and $h_U(u)=1 - u^2$, we obtain the polynomial sub-VF $P$. Similarly to Example~\ref{ex: Jacobson problem 4} we construct a controller from the polynomial sub-VF $P$ using Eqs.~\eqref{optimal controller new} and \eqref{optimal controller new 2}. Using Eq.~\eqref{R sum}, the fixed input $\mbf u(t) \equiv +1$ was found to have cost $446.03$. The fixed input $\mbf u(t) \equiv -1$ cost was found to be $67.48$. The controller derived from $P$ was found to have cost $0.7255$, an improvement compared to a cost of $0.75041$ found in \cite{dadebo1998computation} and $0.8285$ found in \cite{jacobson1970computation}. {Note that in this numerical example the dynamics are linear and the cost function is quadratic. However, due to the input constraints this is not a classical LQR problem and hence cannot be trivially solved.} {Also note, it may be possible that the results of the algorithms proposed in~\cite{jacobson1970computation,dadebo1998computation} may be improved by selecting different tunning parameters of the algorithms. We have assumed that the authors of~\cite{jacobson1970computation,dadebo1998computation} have selected the tunning parameters for which their algorithm performs best.}
\end{ex}

\begin{ex} \label{ex: van}
{	As in~\cite{9480331} let us consider the (scaled) Van der Pol oscillator:
	%-1*[-2*y(2); (0.8*y(1)+(-10*y(2)*(0.21-(1.2^2)*y(1)^2)) +w)]
	\begin{align}  \label{ODE: van}
	&\dot{x}_1(t)  = 2x_2(t), \\ \nonumber
	& \dot{x}_2(t)  = 10x_2(t)(0.21-1.2^2x_1(t)) -0.8x_1(t) + u(t), 
	\end{align}
	where $u(t) \in [-1,1]$. Let us consider OCPs of Form~\eqref{opt: optimal control probelm} governed by the dynamics given in Eq.~\eqref{ODE: van} with $\Omega =\R^n$, $U=[-1,1]$ and cost functions of the form $c(x,u,t)= ||x-q||_2^2$ and $g(x)=||x-q||_2^2$, where $q=[-0.4,0]$ or $q =[0;0]$. Clearly any solution to the OCP is an input $\mbf u$ that forces the systems trajectories towards the point $q\in \R^2$.}
	
{By solving the SOS Optimization Problem~\eqref{opt: SOS for sub soln of finite time} twice for $q=[-0.4,0]$ and $q =[0;0]$ with $d=14$, $T=10$, $f,c,g$ as defined previously, $\Lambda=[-1,1]^2$, $w(x,t) \equiv 1$, $h_\Omega(x)=2.1 -x_1^2 - x_2^2 $, and $h_U(u)=1 - u^2$ we obtain polynomial sub-value functions $P_1$ and $P_2$ respectively. By replacing $V$ with $P_i$, for $i \in \{1,2\}$, in Eqs.~\eqref{optimal controller new} and~\eqref{optimal controller new 2} we then construct controllers, $k_{P_i}$, that yield candidate solution to the OCPs, $\mbf{\tilde{u_i}}(t)=k_{P_i}(x(t),t)$ $i \in \{1,2\}$. }

	{For initial condition $x_0 =[ 0.75 ,0.75]^T$ and terminal time $T=10$ we use Matlab's \texttt{ode45} to find the curves $\{ \phi_{\tilde{f}}(x_0,t,\mbf{\tilde{u_i}} ) \in \R^2: t \in [0,T] \}$ for $i=1,2$ (recalling $\phi_{{f}}$ denotes the solution map (Defn.~\ref{defn: soln map})), which is shown as the blue and red curves respectively in the phase plot in Figure~\ref{fig: van}. Moreover, for comparison we have also plotted the solution trajectory under the fixed input $\mbf u(t) \equiv 0$ as the green curve, which demonstrates the shape of the Van-der-Pol limit cycle. As expected the input $\mbf u_1$ drives the system to the point $q =[-0.4;0]$ with terminal state, shown as the black dot in Figure~\ref{fig: van}, as $[-0.430;0.112]$. Moreover, the input $\mbf u_2$ drives the system to the point $q =[0;0]$ with terminal state $[-0.012;0.007]$.  }
	
	{Table~\ref{tab:cost of van} shows the $T=10$ cost of using various inputs when $q=[-0.4,0]$ or $q =[0;0]$. All costs were calculated using Eq.~\eqref{R sum} for initiali condition $[0.75;0.75]$. The costs of using $\mbf u_1$ and $\mbf u_2$ are shown in the $\mbf u_{SOS}$ row under columns $q=[-0.4,0]$ or $q =[0;0]$ respectively. As expected the inputs derived using SOS out perform (have lower cost) compared to constant inputs.}
\end{ex} 
\vspace{-0.5cm}
\begin{figure}
	\includegraphics[scale=0.6]{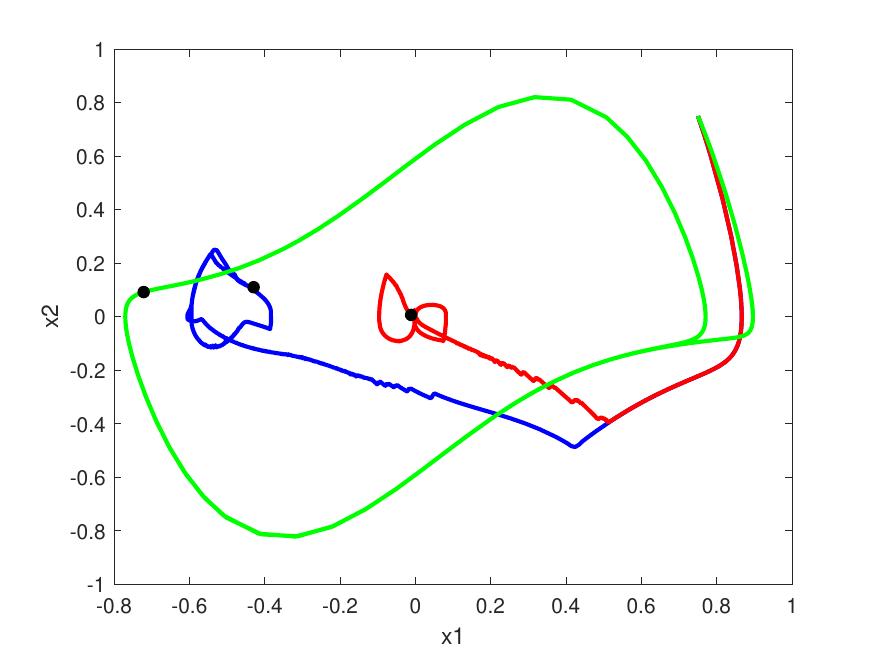}
	\vspace{-25pt}
	\caption{{Graph showing the phase plot of Example~\ref{ex: van} found by constructing controllers given by Eq.~\eqref{optimal controller new} using the solution to the SOS Problem~\eqref{opt: SOS for sub soln of finite time}. The blue curve shows the $T=10$ solution trajectory initialized at $[0.75;0.75]$ of the ODE~\eqref{ODE: van} driven by the controller found by considering costs $c(x,u,t)= ||x-q||_2^2$ and $g(x)=||x-q||_2^2$, where $q=[-0.4,0]$. The red curve shows the $T=10$ solution trajectory initialized at $[0.75;0.75]$ of the ODE~\eqref{ODE: van} driven by the controller found by considering the same costs but with $q =[0;0]$. The green curve is the $T=10$ solution trajectory initialized at $[0.75;0.75]$ of the ODE~\eqref{ODE: van} under the input $\mbf u(t) \equiv 0$. Terminal states for each trajectory are given by the black dots. Costs associated with each trajectory can be found in Table~\ref{tab:cost of van}. } }
	\label{fig: van}
	\vspace{-10pt}
\end{figure}

	\begin{table}
	\centering
	\caption{{This table shows the corresponding costs of various inputs for the OCPs of Form~\eqref{opt: optimal control probelm} given in Example~\ref{ex: van}.}} 
			%The cost functions for these OCPs are of the form $c(x,u,t)= ||x-q||_2^2$ and $g(x)=||x-q||_2^2$, where $q=[-0.4,0]$ or $q =[0;0]$. }}
			\label{tab:cost of van}
	\begin{tabular}{|l|l|l|}
		\cline{1-3}
		Input $\mbf u$  & Cost for $q=[-0.4;0]$      & Cost for $q=[0;0]$    \\ \cline{1-3}
		$\mbf u_{SOS}$  & $0.21473$ &  $0.078919$        \\ \cline{1-3}
		$\mbf u(t) \equiv 0$    & $0.84466$      & $1.0037$    \\ \cline{1-3}
		$\mbf u(t) \equiv +1 $ & $1.1824$        & $2.444$     \\ \cline{1-3}
		$\mbf u(t) \equiv -1 $  & $4.5615$ & $2.4681$        \\ \cline{1-3} 
		% 				$(h,0,-h)$    & h/2        &  &  &  \\ \cline{1-2}
		%   				$(h,0,0)$ & 0        &  &  &  \\ \cline{1-2}
		%   				$(h,-h,0)$  & -h &  &  &  \\ \cline{1-2}
		%   				$(h,-h,h)$    & -(3/2)h      &  &  &  \\ \cline{1-2}
	\end{tabular}
\end{table} 
\subsection{Using SOS Programming to Construct Polynomial Sub-Value Functions For Reachable Sets Estimation} \label{subsec:reachable sets}
Appendix~\ref{app: reachable sets} shows the sublevel sets of VFs characterize reachable sets. We now numerically solve the SOS programming problem in Eq.~\eqref{opt: SOS for sub soln of finite time} obtaining an approximate VF that can be used to estimate the reachable set of the Lorenz system. The problem of estimating the Lorenz attractor has previously been studied in \cite{jones2018using,jones2019using,Li_2004,Wang_2012,Goluskin_2018}.

\begin{ex} \label{ex: Lorrenz}
Let us consider the Lorenz system defined by the three dimensional second order nonlinear ODE:
\vspace{-0.1cm}\begin{align} \nonumber
\dot{x}_1(t)  = \sigma(x_2(t) & - x_1(t)), \text{ } \dot{x}_2(t)  = x_1(t)(\rho - x_3(t)) - x_2(t), \\ \label{ODE: Lorenz}
\dot{x}_3(t) & = x_1(t) x_2(t) - \beta x_3(t),
\end{align}
	where $\sigma=10$, $\beta= 8/3$, $\rho=28$. We make a coordinate change so the Lorenz attractor is located in a unit box by defining
	\vspace{-0.1cm} \begin{align} \label{eqn: Lorentz coordinate change}
	\bar{x}_1  :=  50 x_1, \quad \bar{x}_2   := 50 x_2, \quad \bar{x}_3   := 50 x_3 + 25  .
	\end{align}
	The ODE~\eqref{ODE: Lorenz} can then be written in the form $\dot{x}(t) =\tilde{f}(x(t),\mbf u(t))$ using $\tilde{f}(x)= [50\sigma(x_2 - x_1), 50x_1(\rho - 50x_3- 50(25)) - 50x_2,50^2 x_1 x_2 - 50\beta x_3- 25 \beta ]^T$. Note, as $\tilde{f}$ is independent of any input $u \in U$ without loss of generality we will set $U= \emptyset$. The problem of estimating the Lorenz attractor is then equivalent to the problem of estimating $FR_{\tilde{f}}(\R^n,\R^n,U,\{\infty\})$. In this section we estimate $FR_{\tilde{f}}(\R^n,\R^n,U,\{\infty\})$ by estimating $FR_{\tilde{f}}(X_0,\Lambda,U,\{T\})$ for some $T< \infty$, $\Lambda \subset \R^3$, $X_0:=\{x \in \R^3: g(x)<0\}$, and $g \in \mcl P(\R^n, \R)$.
	
	Figure~\ref{fig: lorrenz} shows the set $\{x \in \R^3: P(x,0) < 0\}$ where $P$ is the solution to the SOS Optimization Problem~\eqref{opt: SOS for sub soln of finite time} for $d=10$, $T=0.5$, {$f(x)=-\tilde{f}(x)$ for all $x \in \Omega:=\{x \in \R^n: h_\Omega(x) \ge 0\}$ and $f(x)=0$ for all $x \in \partial \Omega$ (freezing the dynamics on $\partial \Omega$ helps to ensure Eq.~\eqref{reachable set condition} is satisfied, improving numerical performance)}, $h_U \equiv 0$, $h_\Omega(x)=2^2 - x_1^2 - x_2^2 - x_3^2$, $c \equiv 0$, $g(x)=   (x_1+0.6)^2 + (x_2-0.6)^2 + (x_3-0.2)^2 - 0.1^2$, $\Lambda=[-0.4, 0.4] \times [ -0.5, 0.5] \times  [-0.4, 0.6]$, and $w(x,t)=\delta(t)$ where $\delta$ is the Dirac delta function. Prop.~\ref{prop: diss ineq implies lower soln} shows $P$ is a sub-VF. Then Cor.~\ref{cor: sub value subleve; sets contain reachable sets} shows $BR_f(X_0,\Lambda,U,\{T\}) \subseteq \{x \in \R^3: P(x,0) < 0\}$ and hence $FR_{\tilde{f}}(X_0,\Lambda,U,\{T\})=BR_f(X_0,\Lambda,U,\{T\})\subseteq \{x \in \R^3: P(x,0) < 0\}$ by Lem.~\ref{lem: backward and forward reach sets}. Thus the $0$-sublevel set of $P$ contains the forward reachable set. Moreover, Figure~\ref{fig: lorrenz} provides numerical evidence that the $0$-sublevel set of $P$ approximates the Lorenz attractor accurately.
	
	%By taking $10^6$ randomly sampled points in $[-1,1]^3$ and counting the ratio of points that are inside to outside the $0$-sublevel set of $V_l(x,0)$ the volume of the sublevel set, shown in Figure \ref{fig: lorrenz}, was found to approximately be 0.7649.
\end{ex}
\vspace{-0.2cm}
Note, given an OCP with VF denoted by $V^*$, Prop.~\ref{prop: SOS converges} shows that the sequence of polynomial solutions to the SOS Problem~\eqref{opt: SOS for sub soln of finite time}, indexed by $d \in \N$, converges to $V^*$ with respect to the $L^1$ norm as $d\to \infty$. Moreover,  Prop.~\ref{prop: SOS converges Dv} shows that this sequence of polynomial solutions yields a sequence of sublevel sets that converges to $\{x \in \R^n : V^*(x,0) \le 0 \}$ with respect to the volume metric as $d \to \infty$. However, Theorem.~\ref{thm: HJB to characterize reachable sets} shows reachable sets are characterized by the ``strict" sublevel sets of VFs, $\{x \in \R^n : V^*(x,0) < 0 \}$. Counterexample~\ref{cex: strict dublevel sets do not approx} (Appendix~\ref{sec: appendix 2}) shows that a sequence of functions that converges to some function $V$ with respect to the $L^1$ norm may not yield a sequence of ``strict" sublevel sets that converges to the ``strict" sublevel set of $V$. Therefore we conclude that the sequence of ``strict" sublevel sets obtained by solving the SOS Problem~\eqref{opt: SOS for sub soln of finite time} may in general not converge to the desired reachable set. However, in practice there is often little difference between the sets $\{x \in \R^n : V^*(x,0) \le 0 \}$ and $\{x \in \R^n : V^*(x,0) < 0 \}$. Example~\ref{ex: Lorrenz} shows how accurate estimates of reachable sets can be obtained by solving the SOS Problem~\eqref{opt: SOS for sub soln of finite time}. Moreover, these reachable set estimations are guaranteed to contain the true reachable set by Cor.~\ref{cor: sub value subleve; sets contain reachable sets}, a property useful in safety analysis~\cite{yin2018reachability}.
\begin{figure}
	\includegraphics[scale=0.6]{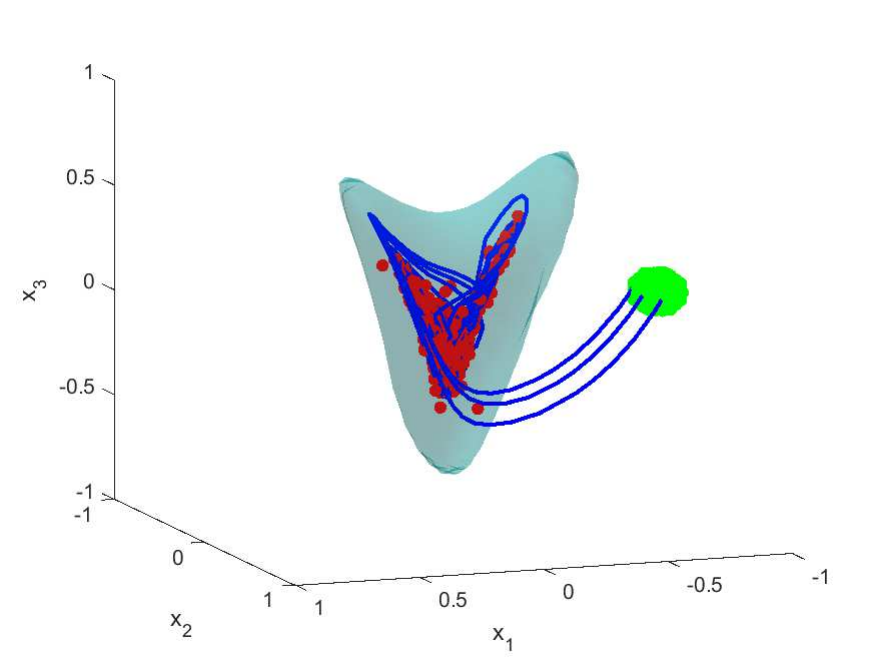}
	\vspace{-28pt}
	\caption{Forward reachable set estimation from Example~\ref{ex: Lorrenz}. The transparent cyan set represents the $0$-sublevel set of the solution to the SOS Problem~\eqref{opt: SOS for sub soln of finite time}, the $20^3$ green points represent initial conditions, the $20^3$ {red} points represent where initial conditions transition to after $t=0.5$ under scaled dynamics from the ODE \eqref{ODE: Lorenz} (found using Matlab's \texttt{ODE45} function), and the three blue curves represents three sample trajectories terminated at $t=0.5$ and initialized at three randomly selected green initial conditions.}
	\label{fig: lorrenz}
	\vspace{-15pt}
\end{figure}

\vspace{-0.6cm}
\section{Conclusion} \label{sec: conclusion}
%We have proposed a feasibility problem solved by the VF of an OCP. We have relaxed and then tightened the feasibility problem for finding VF's to a family of SOS problems, indexed by their degree $d \in \N$. We have shown that a sequence of dissipative polynomials can be constructed, by solving the proposed SOS problems for each $d \in \N$, that converge to the VF as $d \to \infty$ in the $L^1$ norm. We have bounded the performance of a controller constructed from a candidate VF by the $W^{1, \infty}$ Sobolev norm between the candidate VF and true VF. Although $L^1$ approximation isn't as strong as $W^{1, \infty}$ we have demonstrated in several numerical examples that $L^1$ VF approximation is sufficient in practice for controller design.
For a given optimal control problem, we have proposed a sequence of SOS programming problems, each instance of which yields a polynomial, and where the polynomials become increasingly tight approximations to the true value function of the optimal control problem respect to the $L^1$ norm. Moreover, the sublevel sets of these polynomials become increasingly tight approximations to the sublevel sets of the true value function with respect to the volume metric. Furthermore, we have also shown that a controller can be constructed from a candidate value function that performs arbitrarily close to optimality when the candidate value function approximates the true value function arbitrarily well with respect to the $W^{1, \infty}$ norm. We would like to emphasize that our performance bound, for controllers constructed from candidate value functions, can be applied independently of our proposed SOS algorithm for value function approximation, and therefore is maybe of broader interest. %Moreover, the bound was derived in a constructive manner, giving

%Although $L^1$ approximation isn't as strong as $W^{1, \infty}$ we have demonstrated in several numerical examples that $L^1$ value function approximation is sufficient in practice for controller design.
\vspace{-0.45cm}

%shown that the performance of a controller constructed by an approximate value function is bounded by the $W^{1, \infty}$ norm between the approximated value function and true value function. We have proposed an SOS programing problem whose solution can construct a $d$-degree polynomial approximate of a value function. We have shown that as the degree tends to infinity our approximation tends to the true value function with respect to the $L^1$ norm and its sublevel sets tend to the true value functions sublevel sets with respect to the volume metric. Although $L^1$ approximation isn't as strong as $W^{1, \infty}$ we have presented several numerical examples showing $L^1$ approximation is sufficient in practice for controller design.

\bibliographystyle{ieeetr}
\bibliography{BIB_Reach_journal}
%
%\enlargethispage{0cm}
%
\vspace{-1.5cm}
%\newpage
%
%\IEEEtriggercmd{\enlargethispage{-5.35in}}
%\vskip -4pt plus -1fil
\begin{IEEEbiography}[{\includegraphics[width=1in,height=1.25in,clip,keepaspectratio]{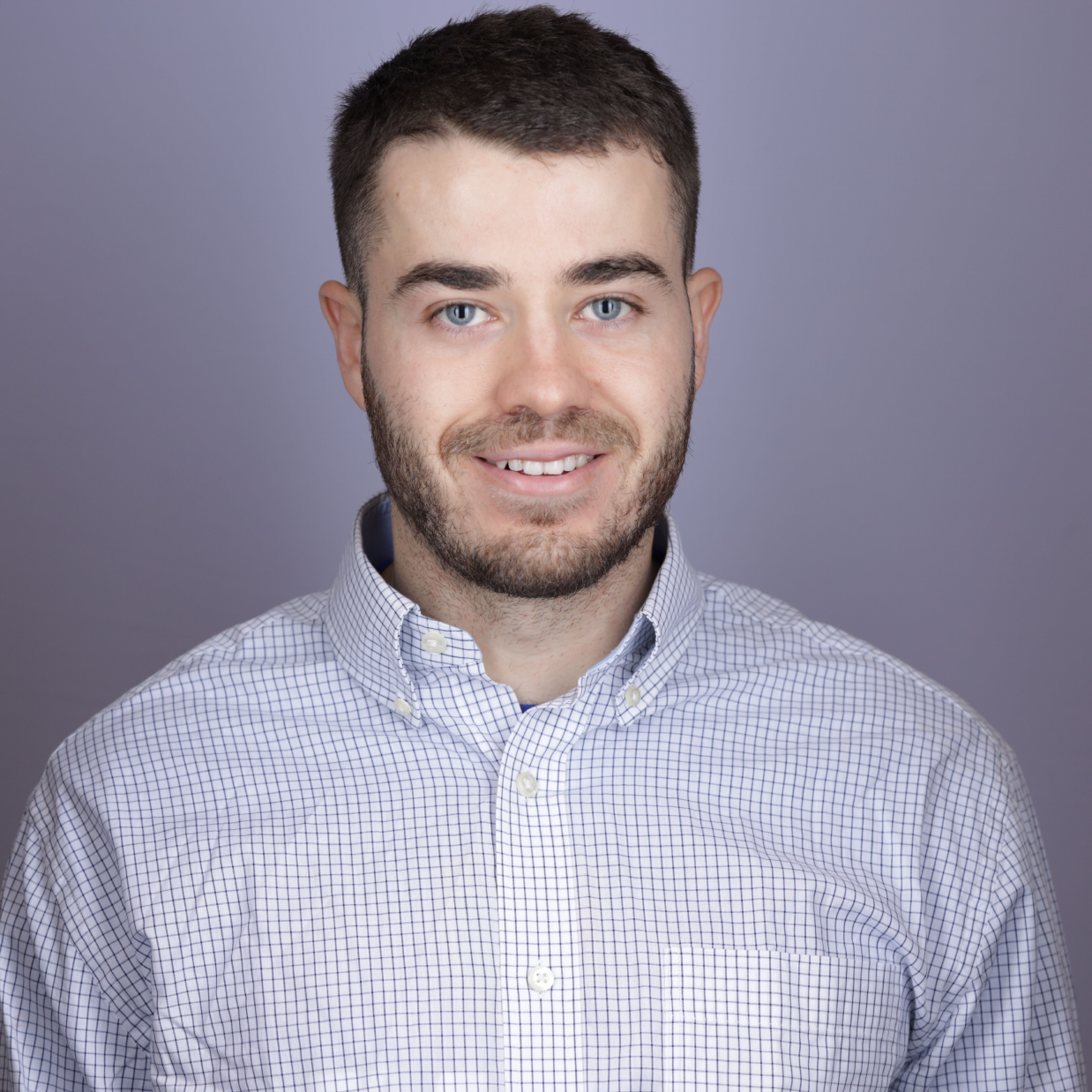}}]{Morgan Jones} \small
	received the Mmath degree in
	mathematics from The University of Oxford, England in 2016 and PhD degree from Arizona State University, USA in 2021.
	Since 2022 he has been a lecturer in the department of Automatic Control and Systems Engineering at the University of Sheffield. His research primarily
	focuses on the estimation of reachable sets, attractors and regions of attraction for nonlinear ODEs. Furthermore,
	he has an interest in extensions of the dynamic programing framework to non-separable cost functions.
\end{IEEEbiography}
%\vskip -2\baselineskip plus -1fil
%\newpage
\vspace{-1cm}
%\vspace{-20cm}
\begin{IEEEbiography}[{\includegraphics[width=1in,height=1.25in,clip,keepaspectratio]{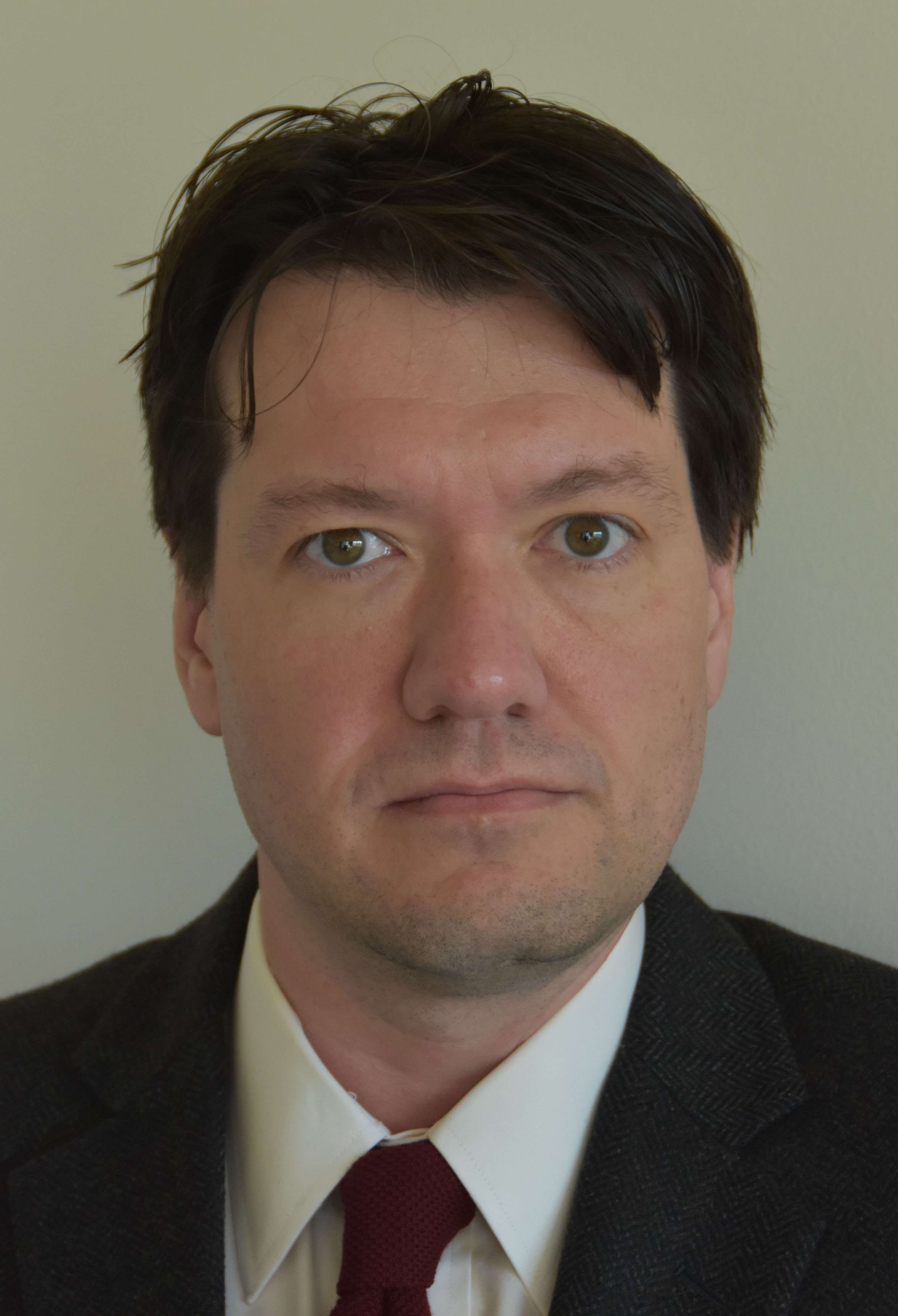}}]{Matthew M. Peet} \small
	received the B.S. degree in
	physics and in aerospace engineering from the University
	of Texas, Austin, TX, USA, in 1999 and
	the M.S. and Ph.D. degrees in aeronautics and astronautics
	from Stanford University, Stanford, CA,
	in 2001 and 2006, respectively. He was a Postdoctoral
	Fellow at INRIA, Paris, France from 2006 to
	2008. He was an Assistant Professor of Aerospace
	Engineering at the Illinois Institute of Technology,
	Chicago, IL, USA, from 2008 to 2012. Currently, he
	is an Associate Professor of Aerospace Engineering
	at Arizona State University, Tempe, AZ, USA. Dr. Peet received a National
	Science Foundation CAREER award in 2011.
\end{IEEEbiography}
%
%
%
%%\enlargethispage{-5cm}
%
%%\vspace{2cm}
%
%%\vspace{-20cm}
%%\newpage
%%\vspace{-40cm}
%%\clearpage
%%\enlargethispage{5cm}
%
\section{Appendix A: sublevel set approximation} \label{sec: appendix 2}
In this appendix we show that the volume metric ($D_V$ in Eq.~\eqref{eqn: volume metric}) is indeed a metric. Moreover, in Prop.~\ref{prop: close in L1 implies close in V norm} we show that if $\lim_{d \to \infty} ||J_d - V||_{L^1}=0$ then for any $\gamma \in \R$ we have $\lim_{d \to \infty}	D_V (\{x \in \Lambda : V(x) \le \gamma\}, \{x \in \Lambda : J_d(x) \le \gamma\} ) =0$. The sublevel approximation results presented in this appendix are required in the proof of Prop.~\ref{prop: SOS converges Dv}.
\begin{defn}
	$D: X \times X \to \R$ is a \textit{metric} if the following is satisfied for all $x,y \in X$,
	\vspace{-0.4cm}
	\setlength{\columnsep}{-0.75in}
	\begin{multicols}{2}
		\begin{itemize}
	\item $D(x,y) \ge 0$,
\item $D(x,y)=0$ iff $x=y$,
\item $D(x,y)=D(y,x)$,
\item $D(x,z) \le D(x,y) + D(y,z)$.
		\end{itemize}
	\end{multicols}
%	\begin{AutoMultiColItemize}
%		\item $D(x,y) \ge 0$,
%		\item $D(x,y)=0$ iff $x=y$,
%		\item $D(x,y)=D(y,x)$,
%		\item $D(x,z) \le D(x,y) + D(y,z)$.
%	\end{AutoMultiColItemize}
%	\begin{itemize}
%		\item $D(x,y) \ge 0$,
%		\item $D(x,y)=0$ iff $x=y$,
%		\item $D(x,y)=D(y,x)$,
%		\item $D(x,z) \le D(x,y) + D(y,z)$.
%	\end{itemize}
\end{defn}
%We now prove $D_V$ is indeed a metric. \textcolor{red}{For unbounded sets is this still a metric?}
\begin{lem}[\cite{jones2018using}] \label{lem: Dv is metric}
	Consider the quotient space,
{	\[
	X:= \mcl B \pmod {\{X \subset \R^n : X \ne \emptyset, \mu(X) =0 \}},
	\] } recalling $\mcl B:= \{B \subset \R^n: \mu(B)<\infty\}$ is the set of all bounded sets. Then $D_V: X \times X \to \R$, defined in Eq.~\eqref{eqn: volume metric}, is a metric.
\end{lem}

%We next show that when $B \subseteq A$ the metric $D_V$ greatly simplifies.

\begin{lem}[\cite{jones2018using}] \label{lem: D_V is related to vol}
	If $A,B \in \mcl B$ and  $B \subseteq A$ then
	\begin{align*}
D_V(A,B)& =\mu(A/B)= \mu(A)- \mu (B).
	\end{align*}
\end{lem}
%\begin{proof}
%	\begin{align*}
%	D_V(X,Y) & = \mu (X/Y \cup Y/X ) = \mu (X/Y) + \mu (Y/X)\\
%	&= \mu (X/Y) = \mu (X) - \mu (Y),
%	\end{align*}	
%	where the second equality is because $X/Y$ and $Y/X$ are disjoint sets; the third equality is because $Y/X= \emptyset$ as $Y \subseteq X$; the fourth equality is because $Y \subseteq X$ so $Y \cap X= Y$ and thus $\mathds{1}_{X/Y}(x)= \mathds{1}_{X}(x)- \mathds{1}_{X \cap Y}(x) =  \mathds{1}_{X}(x)- \mathds{1}_{Y }(x)$ for all $x \in \R^n$, which implies $\mu(X/Y)=\mu(X)-\mu(Y)$.
%\end{proof}
Inspired by an argument used in \cite{lasserre2015tractable} we now show if two functions are close in the $L^1$ norm then it follows their sublevel sets are close with respect to the volume metric.

%\textcolor{red}{Do we need V J measurable or integrable? Do we need their sublevel sets to be bounded? Lemma 4 and Prop 6. Be careful because everything we change down here we might need to impose these conditions on $\Omega$ in the main approx theorem}
\begin{prop} \label{prop: close in L1 implies close in V norm}
	Consider a set $\Lambda \in \mcl B$, a function $V \in L^1(\Lambda, \R)$, and a family of functions $\{J_d \in L^1(\Lambda, \R): d \in \N\}$ that satisfies the following properties:
	\begin{enumerate}
		\item For any $ d \in \N$ we have $J_d(x) \le V(x)$ for all $x \in \Lambda$.
		\item $\lim_{d \to \infty} ||V -J_d||_{L^1(\Lambda, \R)} =0$.
	\end{enumerate}
Then for all $\gamma \in \R$
	\begin{align} \label{sublevel sets close}
\lim_{d \to \infty}	D_V \bigg(\{x \in \Lambda : V(x) \le \gamma\}, \{x \in \Lambda : J_d(x) \le \gamma\} \bigg) =0.
	\end{align}
\end{prop}
\begin{proof}
	To prove Eq.~\eqref{sublevel sets close} we show for all $\eps>0$ there exists $N \in \N$ such that for all $d \ge N$
		\begin{align} \label{sublevel sets close 2}
	D_V \bigg(\{x \in \Lambda : V(x) \le \gamma\}, \{x \in \Lambda : J_d(x) \le \gamma\} \bigg) < \eps.
	\end{align}
In order to do this we first denote the following family of sets for each $n \in \N$
\vspace{-0.3cm} \begin{align*}
A_n:=\bigg\{x \in \Lambda : V(x) \le \gamma + \frac{1}{n} \bigg\}.
%B_n:=\bigg\{x \in \R^n: V(x) \le \gamma - \frac{1}{n} \bigg\}.
\end{align*}
Since $J_d(x) \le V(x)$ for all $x \in \Lambda$ and $ d \in \N$ we have
%We note that for all $d \in \N$ that
\begin{align} \label{J contains V}
\{x \in \Lambda : V(x) \le \gamma\} \subseteq \{x \in \Lambda : J_d(x) \le \gamma\} \text{ for all } d \in N.
\end{align}
Moreover, since $\{x \in \Lambda : V(x) \le \gamma\} \subseteq \Lambda, \; \{x \in \Lambda : J_d(x) \le \gamma\} \subseteq \Lambda$ and $\Lambda \in \mcl B$ it follows $\{x \in \Lambda : V(x) \le \gamma\} \in \mcl B \text{ and } \{x \in \Lambda : J_d(x) \le \gamma\} \in \mcl B$.
%since $J_d(x) \le V(x)$ for all $x \in \Lambda$ and $ d \in \N$.

Now for $d \in \N$
\begin{align} \label{close in v norm}
&	D_V \bigg(\{x \in \Lambda : V(x) \le \gamma\}, \{x \in \Lambda : J_d(x) \le \gamma\} \bigg)\\ \nonumber
& = \mu(\{x \in \Lambda : J_d(x) \le \gamma\}) - \mu(\{x \in \Lambda : V(x) \le \gamma\})\\ \nonumber
&=  \mu(\{x \in \Lambda : J_d(x) \le \gamma\}) - \mu(A_n \cap \{x \in \Lambda : J_d(x) \le \gamma\})\\  \nonumber
&\qquad  +  \mu(A_n \cap \{x \in \Lambda : J_d(x) \le \gamma\} )  - \mu(\{x \in \Lambda : V(x) \le \gamma\})\\  \nonumber
& \le \mu(\{x \in \Lambda : J_d(x) \le \gamma\}) - \mu(A_n \cap \{x \in \Lambda : J_d(x) \le \gamma\})\\  \nonumber
&\qquad  +  \mu(A_n)  - \mu(\{x \in \Lambda : V(x) \le \gamma\})\\  \nonumber
& ={\mu(\{x \in \Lambda : J_d(x) \le \gamma\}/A_n) + \mu(A_n/\{x \in \Lambda : V(x) \le \gamma\})}.
\end{align}
The first equality of Eq.~\eqref{close in v norm} follows by Lemma~\ref{lem: D_V is related to vol} (since the sublevel sets of $V$ and $J_d$ are bounded and satisfy Eq.~\eqref{J contains V}). The first inequality follows as $A_n \cap \{x \in \Lambda : J_d(x) \le \gamma\} \subseteq A_n $ which implies $\mu(A_n \cap \{x \in \Lambda : J_d(x) \le \gamma\}) \le \mu (A_n )$. The third equality follows using Lemma~\ref{lem: D_V is related to vol} and since $A_n \cap \{x \in \Lambda: J_d(x) \le \gamma\} \subseteq \{x \in \Lambda: J_d(x) \le \gamma\}$ and $\{x \in \Lambda: V(x) \le \gamma\} \subseteq A_n$.

To show that Eq.~\eqref{sublevel sets close 2} holds for any $\eps>0$ we will split the remainder of the proof into two parts. In {Part 1} we show that there exists $N_1 \in \N$ such that $\mu(A_n/\{x \in \Lambda : V(x) \le \gamma\}) < \frac{\eps}{2}$ for all $n \ge N_1$. In {Part 2} we show that for any $n \in \N$ there exists $N_2 \in \N$ such that $\mu(\{x \in \Lambda : J_d(x) \le \gamma\}/A_n) < \frac{\eps}{2}$ for all $d \ge N_2$.

\underline{\textbf{Part 1 of proof:}} In this part of the proof we show that there exists $N_1 \in \N$ such that $\mu(A_n/\{x \in \Lambda : V(x) \le \gamma\}) < \frac{\eps}{2}$ for all $ n>N_1$.

Since $\cap_{n=1}^\infty A_n =\{x \in \Lambda : V(x) \le \gamma \}$ and $A_{n+1} \subseteq A_n$ for all $n \in \N$ we have that $\mu(\{x \in \Lambda : V(x) \le \gamma \}) = \mu ( \cap_{n=1}^\infty A_n) = \lim_{n \to \infty} \mu(A_n)$ (using the ``continuity from above" property of measures). Thus there exists $N_1 \in \N$ such that
\begin{align*}
|\mu( A_n ) - \mu( \{ x \in \Lambda : V(x) \le \gamma \})| < \frac{\eps}{2} \text{ for all } n> N_1.
\end{align*}
Therefore it follows
\begin{align*}
& \mu(A_n/\{x \in \Lambda : V(x) \le \gamma\}) \\
& \qquad = \mu( A_n ) - \mu( \{ x \in \Lambda : V(x) \le \gamma \}) < \frac{\eps}{2} \text{ for all } n> N_1.
\end{align*}

\underline{\textbf{Part 2 of proof:}} For fixed $n>N_1$ we now show there exists $N_2 \in \N$ such that $\mu(\{x \in \Lambda : J_d(x) \le \gamma\}/A_n) < \frac{\eps}{2}$ for all $d \ge N_2$.

Now
\begin{align} \label{set inclusion}
\{x \in \Lambda: J_d(x) \le \gamma\}/A_n \subseteq \{x \in \Lambda : n|J_d(x) -V(x)| \ge 1  \}\\ \nonumber
 \text{ for all } d \in \N.
\end{align}
The set containment in Eq.~\eqref{set inclusion} follows since if $y \in \{x \in \Lambda : J_d(x) \le \gamma\}/A_n$ then $y \in \Lambda$, $J_d(y) \le \gamma$ and $y \notin A_n$. Since $y \notin A_n$ we have $V(y) > \gamma + \frac{1}{n}$. Thus
\begin{align*}
n|J_d(y) - V(y)| & \ge n( V(y) - J_d(y) )  \ge n \left( \gamma + \frac{1}{n} - \gamma \right) =1,
\end{align*}
which implies $y \in \{x \in \Lambda : n|J_d(x) -V(x)| \ge 1  \}$.

Since $\lim_{d \to \infty} \int_{\Lambda} |V(x) -J_d(x)| dx =0$ there exists $N_2 \in \N$ such that
\begin{align} \label{F propoerty close in L1}
\int_{\Lambda} |V(x) -J_d(x)| dx< \frac{\eps}{2n} \text{ for all } d \ge N_2.
\end{align}

Therefore,
{\small \begin{align} \nonumber
 \mu(\{x \in \Lambda : & J_\delta(x) \le \gamma\} /A_n)   \le \mu(\{x \in \Lambda : n|J_\delta(x) -V(x)| \ge 1  \})\\ \label{88}
& \le \int_\Lambda n|J_\delta(x) -V(x)| dx  <  \frac{\eps}{2} \text{ for } d \ge N_2.
\end{align} }
\vspace{-0.25cm}

\noindent The first inequality in Eq.~\eqref{88} follows by Eq.~\eqref{set inclusion}. The second inequality follows by {Chebyshev's inequality} (Lemma~\ref{lem: chebyshev}). The third inequality follows by Eq.~\eqref{F propoerty close in L1}.
\end{proof}
Prop.~\ref{prop: close in L1 implies close in V norm} shows if a sequence of functions $\{J_d\}_{d \in \N}$ converges from bellow to some function $V$ with respect to the $L^1$ norm then the sequence sublevel sets $\{x \in \Lambda : J_d(x) \le \gamma\}$ converge to $\{x \in \Lambda: V(x) \le \gamma \}$ with respect to the volume metric. However, this does not imply the sequence of ``strict" sublevel sets $\{x \in \Lambda : J_d(x) < \gamma\}$ converge to $\{x \in \Lambda : V(x) < \gamma \}$ (even if $\{J_d\}_{d \in \N}$ converges from bellow to $V$ with respect to the $L^\infty$ norm). To see this we next consider a counterexample where $\{J_d\}_{d \in \N}$ is a family of functions that can uniformly approximate some given $V: \in Lip( (0,1) ,\R)$ but $\{x \in \Lambda : J_d(x) < \gamma\}$ does not converge to $\{x \in \Lambda: V(x) < \gamma \}$.
\begin{cex} \label{cex: strict dublevel sets do not approx}
	We show there exists $\gamma \in \R$, $\Lambda \subset \R$, $V \in Lip(\Lambda,\R)$ and $\{J_d\}_{d \in \N} \subset Lip(\Lambda,\R)$ such that $J_d(x)\le V(x)$ for all $x \in \Lambda$ and $\lim_{d \to \infty} \int_\Lambda |V(x) - J_d(x)| dx=0$ but
	\begin{align*}
	\lim_{d \to \infty} D_V\bigg(\{ x \in \Lambda : V(x) < \gamma\}, \{x \in \Lambda : J_d(x) < \gamma\} \bigg)  \ne 0
	\end{align*}
 Let
 \vspace{-0.6cm}
	\begin{align*}
	& \Lambda=(0,1), \quad  V(x)= \begin{cases} 0 \text{ if } x \in (0,0.25]\\ 2(x-0.25) \text{ if } x \in (0.25,0.75), \\ 1 \text{ if } x \in [0.75, 1) \end{cases}\\
	& J_d(x) = \begin{cases} 0 \text{ if } x \in (0,0.25]\\ 2(x-0.25) \text{ if } x \in (0.25,0.75), \\ 1 - \frac{1}{d} \text{ if } x \in [0.75, 1) \end{cases}  \gamma=1.
	\end{align*}
	Now for all $d \in \N$ it is clear that we have $J_d(x) \le V(x)$ and $V(x) -J_d(x)< \frac{1}{d}$ for all $x \in \Lambda $. This implies
	\vspace{-0.1cm}\begin{align*}
	\lim_{d \to \infty} \int_\Lambda V(x) - J_d (x) dx \le 	\lim_{d \to \infty} \sup_{x \in \Lambda} (V(x)-J_d(x)) \le  	\lim_{d \to \infty} \frac{1}{d}=0.
	\end{align*}
	However $\{x \in \Lambda : V(x)< \gamma \} = (0,0.75)$ and for all $d \in \N$ $\{x \in \Lambda: J_d(x)< \gamma \} = (0,1)$. Therefore
	\begin{align*}
	& D_V(\{x \in \Lambda:  V(x)<\gamma \},  \{x \in \Lambda: J_d(x)<\gamma \} )\\
	& \hspace{2cm} = D_V((0,0.75),(0,1)) = 0.25 \text{ for all } d \in \N.\\
	 &  \text{Hence,}\\
	&\lim_{d \to \infty} D_V\left(\{ x \in \Lambda : V(x) < \gamma\}, \{x \in \Lambda : J_d(x) < \gamma\} \right)  = 0.25 \ne 0.
	\end{align*}
\end{cex}

\vspace{-0.2cm}
\section{Appendix B: Value Functions Characterize Reachable Sets} \label{app: reachable sets}
In this appendix we present several reachable set results required in our numerical approximation of the Lorenz attractor (Example~\ref{ex: Lorrenz}). Similarly to forward reachable sets (Defn.~\ref{defn: reachbale set}) we now define backward reachable sets.
\begin{defn} \label{defn: back reachbale set}
	For $X_0 \subset \R^n$, $\Omega \subseteq \R^n$, $U \subset \R^m$, $f: \R^n \times \R^m \to \R^n$ and $S \subset \R^+$, let
	\vspace{-0.3cm}{ \begin{align*}
		BR_f &  (X_0,\Omega,U,  S):= \bigg\{y \in \R^n \;: \;\text{there exists } x \in X_0,  T \in S,  \\ & \text{and } \mbf u \in \mcl U_{\Omega,U,f,T}(y,0)
		\text{ such that } \phi_f(y,T,\mbf u)=x  \bigg\}.
		\end{align*} } \normalsize
\end{defn}

\begin{thm}[VFs characterize backward reachable sets \cite{jones2019relaxing}] \label{thm: HJB to characterize reachable sets}
	Given $\{0,g,f,\Omega,U,T\} \in \mcl M_{Lip}$ define $X_0:= \{x \in \R^n: g(x)<0 \}$. Then
	\begin{equation} \label{forward reach is sublevel set}
	BR_f(X_0,\Omega,U,\{T\}) = \{x \in \Omega : V^*(x,0) < 0 \},
	\end{equation}
	where $V^*: \R^n \times \R \to \R$ is any function that satisfies Eq.~\eqref{opt: optimal control general}.
\end{thm}
\begin{cor}[Sub-VFs contain reachable sets] \label{cor: sub value subleve; sets contain reachable sets}
	Given $\{0,g,f,\Omega,U,T\} \in \mcl M_{Lip}$ and suppose $V_l: \R^n \times \R \to \R$ is a sub-VF (Defn.~\ref{defn: sub and super value functions}), then
	%$\{x \in \R^n : V^*(x,0) \le 1\} \subseteq B_r$ then
	\vspace{-0.1cm} \begin{equation} \label{approx reachbalke set}
	BR_f(X_0,\Omega,U,\{T\}) \subseteq \{x \in \Omega : V_l(x,0) < 0 \},
	\end{equation}
	where $X_0:= \{x \in \R^n: g(x)<0 \}$.
\end{cor}
%In the next Lemma we give a relationship between the backward reachable set and forward reachable set; intuitively derived by reversing the time direction of the ODE. This relationship shows finding the set $FR_{f}(X_0,\Omega,U,\{T\})$ is equivalent to finding the set $BR_{-f}(X_0,\Omega,U,\{T\})$.
%
\begin{lem}[Equivalence of computation of backward and forward reachable sets \cite{jones2019relaxing}] \label{lem: backward and forward reach sets}
	Suppose $X_0 \subset \R^n$, $\Omega \subset \R^n$, $U \subset \R^m$, $f: \R^n \times \R^m \to \R^n$, and $T \in \R^+$. Then %$FR_{-f}(X_0,\Omega,U,\{T\})=BR_f(X_0,\Omega,U,\{T\}).$
		\vspace{-0.2cm}\begin{align*}
		FR_{-f}(X_0,\Omega,U,\{T\})=BR_f(X_0,\Omega,U,\{T\}).
		\end{align*}
\end{lem}

\vspace{-0.3cm}
\section{Appendix C} \label{sec: appendix 3}
In this appendix we present several miscellaneous results required in various places throughout the paper and not previously found in any of the other appendices.
%\begin{thm}[Kirszbraun Theorem \cite{mcshane1934extension}] \label{thm: Kirszbraun}
%Consider a Lipschitz continuous function $F: E \to \R$ with Lipschitz constant $L_F>0$, where $E \subset \R^n$. There exists a function $\tilde{F}: \R^n \to \R$ such that
%\begin{enumerate}
%	\item $F(x)=\tilde{F}(x) \quad \forall x \in E$.
%	\item $\tilde{F}$ is Lipschitz continuous over the whole of $\R^n$ with Lipschitz constant $L_{\tilde{F}}=L_F>0$.
%\end{enumerate}
%\end{thm}

%\begin{defn} \label{defn: sub-alg}
%	The family of functions $\mcl A \subset C^k(\Omega, \R)$ is a sub-algebra if for all $f,g \in \mcl A$ and $\alpha \in \R$ the following holds: 1) $fg \in \mcl A$. 2) $f + g \in \mcl A$. 3) $\alpha f \in \mcl A$.
%\end{defn}

\begin{thm}[Polynomial Approximation \cite{Peet2009ExpStablePolyLyap}] \label{thm:Nachbin}
Let $E \subset \R^n$ be an open set and $f \in C^1(E, \R)$. For any compact set $K \subseteq E$ and $\eps>0$ there exists  $g \in \mcl P(\R^n, \R)$ such that
\begin{align*}
\sup_{x \in K}|D^\alpha f(x) - D^\alpha g(x)| < \eps \text{ for all } |\alpha|\le 1.
\end{align*}
% if and only if the following conditions are satisfied
%\begin{enumerate}
%	\item given $x,y \in E$ with $x \ne y$, there exists $h \in \mcl A$ such that $h(x) \ne h(y)$;
%	\item given $x \in E$ there exists $h \in \mcl A$ such that $h(x) \ne 0$;
%	\item given $y \in E$ there exists $h \in \mcl A$ such that $\nabla_x h(y) \ne 0$.
%\end{enumerate}
\end{thm}

\begin{thm}[Rademacher's Theorem \cite{maly1997fine} \cite{evans2010partial}] \label{thm: Rademacher theorem}
	If $\Omega \subset \R^n$ is an open subset and $V \in Lip(\Omega, \R)$, then $V$ is differentiable almost everywhere in $\Omega$ with point-wise derivative corresponding to the weak derivative almost everywhere; that is the set of points in $\Omega$ where $V$ is not differentiable has Lebesgue measure zero. Moreover,
	\vspace{-0.1cm}\begin{align*}
	\esssup_{x \in \Omega} \bigg|\frac{\partial}{\partial x_i}V(x) \bigg| \le L_V \text{ for all } 1 \le i \le n,
	\end{align*}
	where $L_V>0$ is the Lipschitz constant of $V$ and $\frac{\partial}{\partial x_i}V(x)$ is the weak derivative of $V$.
\end{thm}

\begin{lem}[Infimum of family of Lipschitz functions is Lipschitz \cite{lasz2014GEOMETRICA}] \label{lem: family of lip is lip}
	Suppose $\{h_\alpha\}_{\alpha \in I} \subset LocLip(\R^n, \R)$ is a family of locally Lipschitz continuous functions. Then $h: \R^n \to \R$ defined as $h(x):= \inf_{\alpha \in I} h_\alpha(x)$ is such that $h \in LocLip(\R^n, \R)$ provided there exists $ x \in \R^n$ such that $h(x)< \infty$.
\end{lem}

\begin{thm}[Putinar's Positivstellesatz \cite{Putinar_1993}] \label{thm: Psatz}
	Consider the semialgebriac set $X = \{x \in \R^n: g_i(x) \ge 0 \text{ for } i=1,...,k\}$. Further suppose $\{x  \in \R^n : g_i(x) \ge 0 \}$ is compact for some $i \in \{1,..,k\}$. If the polynomial $f: \R^n \to \R$ satisfies $f(x)>0$ for all $x \in X$, then there exists SOS polynomials $\{s_i\}_{i \in \{1,..,m\}} \subset \sum_{SOS}$ such that,
	\vspace{-0.4cm}\begin{equation*}
	f - \sum_{i=1}^m s_ig_i \in \sum_{SOS}.
	\end{equation*}
\end{thm}

\begin{defn} \label{defn: open cover}
	Let $\Omega \subset \R^n$. We say $\{U_i \}_{i=1}^\infty$ is an open cover for $\Omega$ if $U_i \subset \R^n$ is an open set for each $i \in \N$ and $\Omega \subseteq \{U_i \}_{i=1}^\infty$.
\end{defn}

\begin{thm}[Existence of Partitions of Unity \cite{spivak1965calculus}] \label{thm: partition of unity}
Let $E \subseteq \R^n$ and let $\{E_i\}_{i=1}^\infty$ be an {open cover} of $E$. Then there exists a collection of $C^\infty(E, \R)$ functions, denoted by $\{\psi\}_{i=1}^\infty$, with the following properties:
\begin{enumerate}
	\item For all $x \in E$ and $i \in \N$ we have $0 \le \psi_i(x) \le 1$.
	\item For all $x \in E$ there exists an open set $S \subseteq E$ containing $x$ such that all but finitely many $\psi_i$ are 0 on $S$.
	\item For each $x \in E$ we have $\sum_{i=1}^\infty \psi_i(x) =1$.
	\item For each $i \in \N$ we have $\{x \in E: \psi_i(x) \ne 0\} \subseteq E_i$.
\end{enumerate}
\end{thm}

\begin{lem}[Chebyshev's Inequality] \label{lem: chebyshev}
Let $(X, \Sigma, \mu)$ be a measurable space and $f\in L^1(X,\R)$. For any $\eps>0$ and $0<p<\infty$,
\vspace{-0.2cm} \begin{align*}
\mu(\{x \in X: |f(x)|>\eps\}) \le \frac{1}{\eps^p}\int_X |f(x)|^p dx.
\end{align*}
\end{lem}

\begin{lem}[Equivalence of essential supremum and supremum \cite{StackExhange}] \label{lem: esssup=sup}
	Let $E \subset \R^n$ be an open set and $f \in C(E, \R)$. Then $\esssup_{x \in E}|f(x)| = \sup_{x \in E} |f(x)|$.
\end{lem}

%\begin{thm}[Sobolev Inequality] \label{thm:sobolev inequality}
%	content...
%\end{thm}

%\begin{lem}
%	Consider the function $V: \R^n \to \R$.
%	Suppose for some $\eps>0$ the functions and $J: \R^n \to \R$ are such that
%	\begin{align*}
%	|V(x) - J(x)|< \eps \quad \forall x \in \Omega,
%	\end{align*}
%	then it follows for all $\gamma \in \R$
%	\begin{align*}
%	D_H \bigg(\{x \in \Omega: V(x)< \gamma\}, \{x \in \Omega: J(x)< \gamma\} \bigg)< \eps.
%	\end{align*}
%\end{lem}
%\begin{proof}
%	content...
%\end{proof}

%\textcolor{red}{Possibly need to add Rachmarcher theorem, weistrass approximation theorem, estimation lemma, fundamental theorem of Lebesgue integral calculus}

%\vspace{-2cm}

\end{document}